\documentclass{amsart}

\usepackage[centertags]{amsmath}
\usepackage{amsfonts}
\usepackage{amssymb}
\usepackage{amsthm}
\usepackage[bookmarks=true,hyperindex,pdftex,colorlinks,citecolor=red, linkcolor=blue]{hyperref}
\usepackage{tikz}
\usepackage{tikz-cd}
\usepackage[normalem]{ulem}
\usepackage[shortlabels]{enumitem} % more enumerate/itemize options
\usepackage{marginnote}
\usepackage[all]{xy}
\usepackage[dvipsnames]{xcolor}
\usepackage{bm}
\usepackage{cancel}

\newcommand{\N}{\mathbb{N}}
\newcommand{\R}{\mathbb{R}}

\newcommand{\Z}{\mathbb{Z}}

\newcommand{\restricted}{\mathord{\upharpoonright}}

\newcommand{\ep}{\varepsilon}
\newcommand\restr[2]{{% we make the whole thing an ordinary symbol
		\left.\kern-\nulldelimiterspace % automatically resize the bar with \right
		#1 % the function
		\littletaller % pretend it's a little taller at normal size
		\right|_{#2} % this is the delimiter
}}
\newcommand{\littletaller}{\mathchoice{\vphantom{\big|}}{}{}{}}
\definecolor{egraf}{rgb}{0.2,0.4,0}
\newcommand{\diam}{\mathop{\mathrm{diam}}\nolimits}
\newcommand{\dist}{\mathop{\mathrm{dist}}\nolimits}

\newcommand{\rad}{\mathop{\mathrm{rad}}\nolimits}

\newcommand{\lspan}{{\mathrm{span}}}

\newcommand{\Lip}{{\mathrm{Lip}}}

\newcommand{\F}{\mathcal{F}}              
\newcommand{\conv}{\mathop\mathrm{conv}}
\newcommand{\len}{\mathrm{len}}

\def\<{\langle}
\def\>{\rangle}

\DeclareMathOperator{\lip}{lip}                             % little Lipschitz functions

\theoremstyle{plain}
\newtheorem{thm}{Theorem}[section]
\newtheorem{theorem}[thm]{Theorem}

\newtheorem{corollary}[thm]{Corollary}

\newtheorem{lemma}[thm]{Lemma}

\newtheorem{proposition}[thm]{Proposition}

\newtheorem{maintheorem}{Theorem} % for fancy main theorems with letter numbering

\theoremstyle{definition}
\newtheorem{defn}[thm]{Definition}
\newtheorem{definition}[thm]{Definition}

\newtheorem{remark}[thm]{Remark}

\newtheorem{example}[thm]{Example}

\newtheorem{question}{Question}

\author[M. Jung]{Mingu Jung}
\address[M. Jung]{Department of Mathematics \& Research Institute for Natural Sciences, Hanyang University, 04763 Seoul, Republic of Korea}
\email{mingujung@hanyang.ac.kr}

\author[C. Petitjean]{Colin Petitjean}
\address[C. Petitjean]{Univ Gustave Eiffel, Univ Paris Est Creteil, CNRS, LAMA UMR8050, F-77447 Marne-la-Vallée, France}
\email{colin.petitjean@univ-eiffel.fr}

\author[A. Prochazka]{Anton\'in Prochazka}
\address[A. Prochazka]{Universit\'e de Franche-Comt\'e, CNRS, LmB (UMR 6623), F-25000 Besan\c con, France.}
\email{antonin.prochazka@univ-fcomte.fr}

\author[A. Quilis]{Andr\'es Quilis}
\address[A. Quilis]{Instituto Universitario de Matem\'atica Pura y Aplicada, Universitat
Polit\`ecnica de Val\`encia, Camino de Vera S/N, 46022, Valencia, Spain}
\email{anquisan@upv.edu.es}

\begin{document}
	\title[Affine approx. on finite Nagata dimension spaces and applications]{Affine Approximation in Finite Nagata Dimension and Applications to Lipschitz-free spaces}
	
	% KEYWORDS 
	%46B20 Geometry and structure of normed linear spaces
	%54E35 Metric spaces, metrizability
	%54E50 - Complete metric spaces
	%60B05 - Probability measures on topological spaces
	%57R12 Smooth approximations in differential topology
	
	%Not used
	%46T99 - None of the above, but in this section (Nonlinear functional analysis)
	%28A15 Abstract differentiation theory, differentiation of set functions 
	%51F99 None of the above, but in this section (Metric geometry)
	
	
\begin{abstract}
	We show that if $M$ is a metric space of Nagata dimension at most $d$, then there exists an atlas on $M$ modeled on $\R^d$ such that every Lipschitz map $f:M\to Y$ (with values in an arbitrary Banach space $Y$) can be uniformly approximated by maps that are affine, and thus $\mathcal{C}^1$-smooth, with respect to this atlas. The construction relies on random metric partitions and stochastic retractions inside Lipschitz-free spaces. As an application, we introduce approximate continuous upper gradient $X$-structures (ACUG $X$-structures) on metric spaces and prove that every space of finite Nagata dimension carries an ACUG structure modeled on a superreflexive Banach space. Finally, adapting a proof due to Bourgain, we show that if $M$ has an ACUG superreflexive-structure, then the Lipschitz-free space $\mathcal{F}(M)$ has Pe\l{}czyński's property~(V*). In particular, at least in the compact case, our result recovers all previously known examples of metric spaces \(M\) for which \(\mathcal{F}(M)\) has property~(V*).
	
\end{abstract}

\maketitle

\section*{Introduction}
In Euclidean spaces, Rademacher's theorem asserts that every real-valued Lipschitz function is differentiable almost everywhere with respect to Lebesgue measure. One of the main themes in analysis on metric spaces is to understand to what extent a first-order differentiability structure survives when the underlying space is no longer linear. In his seminal work~\cite{Che99}, Cheeger introduced a notion of differentiability for real-valued Lipschitz functions on certain metric measure spaces, by means of a countable family of Lipschitz coordinate maps 
%\mnote{\tch{The where do the charts have values, in $\R^n$?} \cch{Yes in Cheeger} \ach{We say this explicitly shortly after, so I wouldn't say it here} \cch{\checkmark} \mch{\checkmark}}
(now referred to as charts), and proved an analogue of Rademacher's theorem under a doubling assumption and Poincaré-type inequalities. Following the work of several authors, most notably Keith \cite{Keith}, this led to the notion of \emph{Lipschitz differentiability spaces} (LDS), namely metric measure spaces for which there exists a countable family of charts covering the space (called atlas) and with respect to which all real-valued Lipschitz functions are differentiable almost everywhere, in the sense of Cheeger.

%\mnote{\ach{I changed a bit this paragraph and added a reference.} \cch{\checkmark} \mch{\checkmark I removed the crossed-out words}}
Since then, the study of LDS has become a well established field. Notable contributions include for instance~\cite{Bate15}, where Bate developed the structure theory of such spaces, showing in particular that the differentiable structure can be reconstructed from the measure via Alberti representations. More recently, Bate and Li established that an appropriate Poincaré-type inequality is in fact necessary for differentiability of Lipschitz maps with values in Banach spaces with the Radon–Nikodým property (RNP)~\cite{BateLi-RNP}, and Eriksson-Bique provided a characterization of complete RNP-Lipschitz differentiability spaces through so-called PI-rectifiability~\cite{Erik19}. More details about Radon-Nikodým spaces are provided below.

Collectively, this circle of ideas provides a robust metric version of first-order calculus, which successfully identifies those metric measure spaces in which the Lipschitz condition implies almost everywhere differentiability. In this paper we adopt a different point of view, motivated by the theory of \emph{Lipschitz-free spaces} (also known as Arens$-$Eells spaces, or transportation cost spaces). In this setting, Lipschitz functions are seen as linear functionals acting on a Banach space, and their optimal Lipschitz constant coincides with their norm as operators. As a consequence, results concerning the approximation of Lipschitz functions by ``better" ones play a much more prominent role (we will later expand on this statement). 

In this article, we borrow some of the foundational ideas of Cheeger's differentiability theory in order to define a stronger notion in a purely metric setting, with an eye towards approximation of Lipschitz functions rather than intrinsic manifold structure. Recall that in the mainstream formulation of Cheeger's theory, and in essentially all subsequent work on LDS, charts are Lipschitz maps $\varphi_i : U_i \to \mathbb{R}^{n_i}$, defined on measurable subsets $U_i$ of the considered metric space $M$, and the differential of a Lipschitz function $f : M \to \mathbb{R}$ at a point $x \in U_i$ is encoded by a linear functional $Df(x) \in (\mathbb{R}^{n_i})^*$.

There are three main differences in our setting. First, we do not restrict ourselves to metric measure spaces. Second, we \emph{model} the differentiable structure on a fixed Banach space $X$, not necessarily Euclidean. More concretely, we consider atlases whose charts are Lipschitz maps $\varphi_i : U_i \to X$, where $U_i \subset M$ are Borel sets and $(U_i)_i$ covers $M$. All charts take values in the \emph{same} Banach space $X$, although they may generate different closed subspaces. Third, although we will briefly consider $\mathcal C^1$ maps with respect to such an atlas (see Section~\ref{section:atlas}), our focus will be on \emph{affine maps}: Fix an atlas $\mathcal{U} = (U_i,\varphi_i)_{i\in I}$, and set
$X_i := \overline{\lspan}\bigl(\varphi_i(U_i)\bigr) \subset X$ for each $i\in I$. Given a Banach space $Y$ and a
map $f\colon M \to Y$, we say that $f$ is \emph{affine with respect to $\mathcal{U}$} if, for every $i\in I$, there exists a unique bounded linear operator $(Af)_i \colon X_i \to Y$ such that
\begin{equation}f(y) - f(x)
= (Af)_i\bigl(\varphi_i(y) - \varphi_i(x)\bigr)
\qquad \text{for all } x,y \in U_i.\end{equation}
This viewpoint, especially when applied to $\mathcal C^1$ maps relative to $\mathcal{U}$, is closer in spirit to the usual differential calculus on Banach manifolds.

%--------------------------------------------------------------------------------------------------------

\medskip

Next, we try to find a natural setting where Lipschitz maps can be well approximated by Lipschitz affine maps. Many approximation and differentiability results available in the literature either exploit the linear structure of the ambient space (the classical case), or rely on strong geometric hypotheses, typically assuming that the underlying space is a doubling metric measure space satisfying a Poincaré-type inequality (a so-called PI space). Cheeger's theorem for PI spaces, the characterizations of RNP-LDS by Bate and Li (using asymptotic nonhomogeneous Poincaré inequalities) and by Eriksson-Bique (using PI-rectifiability), and related results for maps into Banach spaces with the RNP all fall into this second category~\cite{Che99,CheKleiner-RNP,BateLi-RNP,Erik19}. Our aim is to obtain a first-order approximation theory for Lipschitz maps that does \emph{not} rely on any sort of Poincaré-type inequality (indeed, that it does not depend on any measure). Instead, we work in the class of metric spaces of finite Nagata dimension, which is strictly more general than doubling spaces (and thus PI spaces).

The Nagata dimension of a metric space is a covering dimension defined by uniform control over bounded covers at every scale. More precisely, a metric space has Nagata dimension at most $d$ if there is a constant $C$ such that, for every scale $r>0$, the space admits a cover by sets of diameter at most $Cr$ with $r$-multiplicity at most $d+1$ (see Section~\ref{subsection:Nagata} for more details). This notion was systematically studied by Lang and Schlichenmaier~\cite{LangSchli05}, who showed that finiteness of the Nagata dimension holds for a wide class of spaces: all doubling spaces, metric trees, Euclidean buildings, and homogeneous or pinched negatively curved Hadamard manifolds, among others. Thus, finite Nagata dimension provides a natural framework that captures a wide variety of geometries while remaining compatible with Lipschitz extension phenomena, which we describe below.
\medskip

A central ingredient in our approach is a careful use of Lipschitz \emph{almost-extensions}
%\tch{\footnote{for $N\subset M$ and $f:N \to Y$, the map $\tilde{f}:M\to Y$ is an almost-extension if $\sup_{x\in N}\norm{f(x)-\tilde{f}(x)}\leq \varepsilon$ for some $\varepsilon>0$}} 
%% on second thought, let's keep it informal
from nets, driven by random partitions. The method of \emph{random metric partitions} introduced by Lee and Naor~\cite{LeeNaorInv} provides a general way to build Lipschitz extensions by using probabilistic partitions of the space. In our setting, however, we cannot simply invoke the abstract extension theorems of Lee--Naor: they do not directly encode the fine structural information required for our purposes, namely the existence of an atlas modeled on a Banach space $X$ and $\mathcal{C}^1$-regularity with respect to that atlas. Instead, we return to the underlying random partition constructions and build almost-extensions in a way that naturally produces an atlas modeled on $X$ and yields Lipschitz and affine almost-extensions with precise control in the uniform norm.

More precisely, these almost-extensions are obtained via (close relatives of) \emph{stochastic retractions} (also called random projections): given a closed subset $S \subset M$, one considers a map $x \mapsto \mu_x$ assigning to each $x \in M$ a probability measure $\mu_x$ supported on $S$, with quantitative control on the Wasserstein-$1$ distance between $\mu_x$ and the Dirac mass $\delta_x$. This idea already underlies the $K$-gentle partitions of unity of Lee--Naor and was made explicit by Ohta~\cite{Ohta09} under the name of stochastic retraction. Later, Ambrosio--Puglisi~\cite{AmbPug20} introduced a similar notion of (strong) random projection, formulated in the language of Lipschitz-free spaces, and showed that the existence of such objects is essentially equivalent to the existence of linear extension operators between spaces of Lipschitz functions. Building on these ideas, we construct, for metric spaces of finite Nagata dimension, \emph{stochastic almost-retractions} adapted to increasingly dense nets, and use them to produce an atlas modeled on a Banach space $X$ together with Lipschitz and affine almost-extensions with quantitative uniform control, providing good uniform approximations of the original Lipschitz function.

Altogether, combining the geometric control provided by finite Nagata dimension with the probabilistic extension scheme, we arrive at the following result, which is the first main outcome of the paper.

\begin{maintheorem}\label{thmA}
Let $M$ be a separable metric space of Nagata dimension $d$. Then every 1-Lipschitz function from $M$ to any Banach space $Y$ can be uniformly approximated by a $O(\gamma_d(M)d^{7/2})$-Lipschitz function which is affine with respect to some atlas $\mathcal U$ modeled on $\R^d$. 
\end{maintheorem}

In the above statement, $\gamma_d(M)$ is the optimal constant that witnesses that $M$ has Nagata dimension $d$ (see Section \ref{subsection:Nagata} for more details on the Nagata dimension constant, and Theorem \ref{Theorem:Nagata_iff_approx_C1_smooth} for a precise statement of the above Theorem). We also point out already that the atlas $\mathcal{U}$ does not depend on the function one seeks to approximate, and instead only depends on the tolerance allowed for the error in the approximations. 

We now turn to the second part of the paper, where we apply the theory developed above to questions in the geometry of Banach spaces. Differentiation techniques in Banach spaces have a long history and already underlie a number of fundamental phenomena. A basic example is the Radon--Nikodým property (RNP): a Banach space $X$ is said to have the RNP if every Lipschitz curve $f \colon [0,1] \to X$ is differentiable almost everywhere. At first sight this appears to be a purely analytic requirement, but it is in fact equivalent to several geometric properties of $X$; see, for instance, \cite{DiestelUhl}. Another important role of differentiability arises in the nonlinear classification of Banach spaces. Under suitable hypotheses (for example, when $X$ is separable and $Y$ has the RNP), differentiating a bi-Lipschitz map $f \colon X \to Y$ gives rise to a linear embedding of $X$ into $Y$; see \cite[Chapter~7]{BenLin} and \cite[Chapter~14]{AlbiacKalton}. Differentiability also plays a central role in renorming theory, and we refer to the book of Deville, Godefroy and Zizler \cite{DGZ} for a comprehensive account.

In the second part of this paper we focus on a particular class of Banach spaces, the
\emph{Lipschitz-free spaces}. Given a pointed metric space $M$, its Lipschitz-free space $\F(M)$ is a Banach space that contains a canonical isometric copy of $M$ and is linearly generated by this copy; see Section~\ref{subsection:Free} for details. From the dual viewpoint, $\F(M)$ is the canonical predual of $\Lip_0(M)$, the Banach space of real-valued Lipschitz functions on $M$ that vanish at the base point. It can also be realized as the space generated by the $1$-Wasserstein space of probability measures on $M$ when these are viewed as functionals on $\Lip_0(M)$. Understanding the linear and topological structure of $\F(M)$ is a challenging and active area of research, at the intersection of Banach space theory, metric geometry, measure theory and optimal transport.

Within this framework, uniform approximation of Lipschitz maps by functions with additional regularity or geometric properties has proved to be a powerful tool. For instance, a result of Bate \cite[Lemma~3.4]{Bate20} shows that, on a compact purely $1$-unrectifiable metric space, every Lipschitz map can be uniformly approximated by locally flat Lipschitz functions (also known as \emph{little Lipschitz} functions; see Section~\ref{subsection:notation}). This has strong consequences for the structure of Lipschitz-free spaces, including a characterization of those $\F(M)$ having the RNP and a refined description of their geometry; see \cite{AGPP,curveflat}. Smoothing and approximation techniques have likewise played an important role in the study of the approximation property for Lipschitz-free spaces; see, for example, \cite{PS,ST}.

More closely related to the present work is the paper \cite{KP18} which deals with Pe\l{}czy\'nski's property~(V*), a strong form of weak sequential completeness (see Section~\ref{subsection:Vstar}). Using deep results of Hájek and Johanis on the approximation of Lipschitz functions by smooth functions on suitably smooth Banach spaces~\cite{HajekJohanis,HajekJohanisBook} and an argument of Bourgain~\cite{Bourgain,Bourgain2}, the authors show that $\F(M)$ has property (V*) whenever $M$ is a compact subset of a superreflexive Banach space. Our main new result in this direction is the following theorem.

\begin{maintheorem}
	If a metric space $M$ has finite Nagata dimension, then the Lipschitz-free space $\F(M)$ has Pe\l{}czy\'nski's property~\emph{(V*)}.
\end{maintheorem}

The proof of Theorem~B relies on a geometric-analytic structure of a metric space that we call an \emph{approximate continuous upper gradient $X$-structure} (ACUG $X$-structure), where $X$ is a Banach space. Recall that in the metric setting upper gradients provide a substitute for the pointwise norm of the gradient: a Borel function $g \colon M \to [0,\infty]$ is an upper gradient of a map $f$ if along every rectifiable curve $\gamma: [0,\ell] \to M$ one has \begin{equation}d\bigl(f(\gamma(0)),f(\gamma(1))\bigr) \le \int_0^\ell g(\gamma(s))ds.\end{equation}
Such objects form the basis of first-order calculus and Sobolev theory on metric measure spaces. An ACUG $X$-structure formalizes the idea that Lipschitz maps on $M$ can be uniformly approximated by Lipschitz maps that admit continuous $X^*$-valued ``gradients'', whose $X^*$-norms are upper gradients with quantitative bounds. This is precisely the analytic input needed to adapt Bourgain's argument \cite{Bourgain,Bourgain2}, as refined in \cite{KP18}, to our setting. 
Using Theorem~\ref{thmA},
we show that every compact metric space of finite Nagata dimension admits an ACUG structure modeled on a superreflexive Banach space. 
Further, we show that whenever $M$ carries such a structure the space $\F(M)$ has property~(V*). In particular, the implication
\begin{equation}
\text{compact and ACUG superreflexive structure} \,\, \Longrightarrow \,\,
\F(M)\text{ has property (V*)}
\end{equation}
subsumes, at least in the compact case, all previously known classes of metric spaces $M$ such that $\F(M)$ has (V*).

In summary, this work provides a new route to property~(V*), linking together ideas from metric differentiability, random partitions, upper gradient techniques, and the geometry of Banach spaces. Beyond its specific application to Lipschitz-free spaces, the ACUG framework may serve as an analytic tool for further approximation problems in nonlinear functional analysis and metric geometry.

\section{Preliminaries}

\subsection{Notation} \label{subsection:notation}

Throughout this paper, metric spaces will be denoted by $M$, $N$, and Banach spaces by $X$, $Y$. All Banach spaces we consider are real.
\smallskip

If $(M,d)$ is a metric space, we say that $M$ is \emph{pointed} when it comes equipped with a distinguished basepoint, which will always be denoted by $0$. 

If $x\in M$ and $r\geq0$, $B(x,r)$ stands for the closed ball centered at $x$ and of radius~$r$. For subsets $S\subset M$ and $\eta>0$, we denote by
\begin{equation}
[S]_\eta := \{x\in M : d(x,S)\le \eta\}
\end{equation}
the closed $\eta$-neighborhood of $S$.  
The \emph{diameter} and \emph{radius} of $S$ are given respectively by
\begin{equation}
\diam(S) := \sup\{d(x,y):x,y\in S\},
\qquad
\rad(S) := \inf_{x\in S}\sup_{y\in S} d(x,y).
\end{equation}
We denote by $|S|$ the cardinality of the set $S$.  

A subset $S\subset M$ is said to be \emph{$\varepsilon$-dense} if $[S]_\varepsilon = M$, that is, every point of $M$ lies at distance at most $\varepsilon$ from $S$.

\smallskip

For a Banach space $X$, $X^*$ denotes its topological dual. We write $B_X$ and $S_X$ for the closed unit ball and the unit sphere of $X$, respectively, and $\conv(A)$ for the convex hull of a set $A\subset X$.

\smallskip

Given metric spaces $(M,d)$ and $(N, \rho)$, we denote by $\Lip(M,N)$ the set of all Lipschitz maps from $M$ to $N$. If $f \in \Lip(M,N)$, then $\Lip(f)$ denotes the optimal (or best) Lipschitz constant of $f$. That is, 
\begin{equation*}
\Lip(f) := \sup_{x \neq y \in M} \frac{ \rho (f(x), f(y) )}{d(x,y)}.
\end{equation*}
When $M$ and $N$ are pointed, we write $\Lip_0(M,N)$ for the subset of Lipschitz maps $f:M\to N$ such that $f(0)=0$.

If the range $N=Y$ is a Banach space, then $\Lip_0(M,Y)$ becomes a Banach space when endowed with the Lipschitz norm
\begin{equation}
\|f\|_L := \Lip(f) = \sup_{x \neq y \in M} \frac{\|f(x)-f(y)\|_Y}{d(x,y)}.
\end{equation}
The quantity $\|f\|_L$ defines a seminorm in $\Lip(M,Y)$. We use the notation $\|f\|_L$ and $\Lip(f)$ interchangeably, though we typically reserve $\|f\|_L$ when regarding $f$ as an element of a Lipschitz function (semi)normed space. 

In particular, we set $\Lip(M):=\Lip(M,\R)$ and $\Lip_0(M):=\Lip_0(M,\R)$. 
\smallskip

We also consider the spaces of \emph{little Lipschitz} functions, or \emph{uniformly locally flat} Lipschitz maps, defined by
\begin{equation}
\lip(M) := \bigl\{ f\in \Lip(M) : 
\lim_{r\to 0^+} \sup_{\substack{x,y\in M\\ 0<d(x,y)\le r}} 
\frac{|f(x)-f(y)|}{d(x,y)} = 0 \bigr\}.
\end{equation}
When $M$ is pointed, we set $\lip_0(M):=\{f\in\lip(M): f(0)=0\}$.

\subsection{Upper gradients}

The notion of \emph{upper gradient} was introduced by J.~Heinonen and P.~Koskela in~\cite{HeinonenUpper}; see also~\cite{Che99, Heinonen} for an extensive exposition and applications to analysis on metric spaces. We will use a natural extension of the classical definition that works for Banach-space valued functions.

\begin{defn} \label{def:UpperGradient}
	Let $M$ be a metric space, let $A \subset M$, and let $Y$ be a Banach space. 
	  Let $f\colon A\to Y$ be a function. An \emph{upper gradient} for $f$ is a Borel function $g : A \to [0,\infty]$ such that, for every pair of points $p,q \in A$ and every $1$-Lipschitz curve $\gamma : [0,\ell] \to A$ with $\gamma(0)=p$ and $\gamma(\ell)=q$, one has
	\begin{equation}
	\|f(q)-f(p)\|_Y \le \int_0^\ell g(\gamma(s))\,ds.
	\end{equation}
\end{defn}

\smallskip

Intuitively, an upper gradient plays the role of a \emph{metric derivative bound}: it controls the oscillation of the function $f$ along all Lipschitz curves in $A$.

It is trivial that if $f$ is a constant function, then $g \equiv 0$ is an upper gradient of $f$, and
if $f$ is $L$-Lipschitz on $A$, then the constant function $g \equiv L$ is an upper gradient of $f$. 

A more refined and important example is obtained by considering the \emph{local Lipschitz constant} (or \emph{pointwise Lipschitz constant}). Given a Banach space $Y$ and a function $f:A \to Y$, it is defined for each $x\in A$ by
\begin{equation}
\Lip(f,x) := \limsup_{\substack{y\to x \\ y\ne x}} \frac{\|f(y)-f(x)\|_Y}{d(x,y)}.
\end{equation}

 It follows from \cite[Proposition~1.11]{Che99} that, for every extended real-valued Lipschitz function $f$, $\Lip(f,\cdot)$ is an upper gradient of $f$. Since we are also dealing with vector-valued maps, we include a short proof of the following lemma for convenience of the reader.

\begin{lemma}\label{lem:local-Lip-upper-gradient}
	Let $M$ be a metric space, $A \subset M$, and let $Y$ be a Banach space and $f : A \to Y$ a Lipschitz function. 
	Then $\Lip(f,\cdot)$ is an upper gradient of $f$.
\end{lemma}

\begin{proof}
    It suffices to prove the result for $Y=\R$. Indeed, assume that the local Lipschitz constant is an upper gradient for real-valued Lipschitz functions, and consider $p,q\in A$. The function $h_p\colon Y\rightarrow \R$ given by $h_p(y)=\|f(p)-y\|$ is $1$-Lipschitz, and thus $\Lip((h_p\circ f),\cdot)\leq \Lip(f,\cdot)$ pointwise in $A$. Since we are assuming that the result holds for the Lipschitz function $h_p\circ f$, we get that for every $1$-Lipschitz curve $\gamma\colon [0,\ell]\rightarrow A$ with $\gamma(0)=p$ and $\gamma(\ell)=q$ it holds
    \begin{align}
        \|f(p)-f(q)\|&=\|(h_p\circ f)(p)-(h_p\circ f)(q)\|\\&\leq \int_{0}^\ell\Lip((h_p\circ f),\gamma(s))ds\leq \int_0^\ell \Lip(f,\gamma(s))ds\nonumber,
    \end{align}
    as sought.

	Hence, assume that $f\colon A\rightarrow \R$ is a real-valued Lipschitz function. Let $p,q\in A$ and let $\gamma : [0,\ell]\to A$ be a $1$-Lipschitz curve with $\gamma(0)=p$ and $\gamma(\ell)=q$. 
	Set $h := f\circ\gamma : [0,\ell]\to\R$. 
	
	Fix $t\neq s\in[0,\ell]$. Then
	\begin{equation}
	\frac{|h(s)-h(t)|}{|s-t|}
	= \frac{|f(\gamma(s))-f(\gamma(t))|}{|s-t|}
	\le \frac{|f(\gamma(s))-f(\gamma(t))|}{d(\gamma(s),\gamma(t))},
	\end{equation}
	because $\gamma$ is $1$-Lipschitz.
	Taking $\limsup$ as $s \to t$ with $s\neq t$ gives
	\begin{equation}
	\Lip(h,t) \le \Lip\big(f,\gamma(t)\big)
	\quad\text{for every }t\in[0,\ell],
	\end{equation}
	Since $h$ is Lipschitz on an interval, it is differentiable almost everywhere (Rademacher's theorem on $\R$) and satisfies
	\begin{equation}
	|h'(t)| \le \Lip(h,t) \le \Lip(f,\gamma(t))
	\quad\text{for a.e. }t\in[0,\ell].
	\end{equation}
	Integrating, we obtain
	\begin{equation}
	|f(q)-f(p)|
	= |h(\ell)-h(0)|
	\le \int_0^\ell |h'(t)|\,dt
	\le \int_0^\ell \Lip(f,\gamma(t))\,dt.\qedhere
	\end{equation}
\end{proof}

\subsection{Geodesic and quasiconvex metric spaces}

We recall the basic notions of geodesic and quasiconvex metric spaces, which will play a central role in what follows.

\begin{defn}
	Let $(M,d)$ be a metric space.
	\begin{enumerate}
		\item We say that $M$ is \emph{geodesic} if for every pair of points $x,y\in M$ there exists a map 
		$\gamma:[0,\ell]\to M$ such that $\gamma(0)=x$, $\gamma(\ell)=y$, and
		\begin{equation}
		d(\gamma(s),\gamma(t)) = |s-t|
		\qquad\text{for all }s,t\in[0,\ell],
		\end{equation}
		where $\ell = d(x,y)$. Such a curve $\gamma$ is called a \emph{geodesic} joining $x$ and $y$.
		\item For $C\ge1$, we say that $M$ is \emph{$C$-quasiconvex} if for all $x,y\in M$ there exists a rectifiable curve $\gamma:[0,1]\to M$ with $\gamma(0)=x$, $\gamma(1)=y$, and
		\begin{equation}
		\mathrm{length}(\gamma) \le C\, d(x,y).
		\end{equation}
		When this holds for some constant $C$, we simply say that $M$ is \emph{quasiconvex}.
	\end{enumerate}
\end{defn}

\smallskip

Geodesic spaces are precisely the $1$-quasiconvex spaces. 
Intuitively, quasiconvexity means that any two points can be joined by a curve whose length is comparable to their distance.
Clearly, every convex subset of a normed space is geodesic for the induced metric.  Every length space is $(1+\ep)$-quasiconvex, for every $\ep>0$. Recall that $M$ is a \emph{length space} if the distance between any two points equals the infimum of the lengths of curves joining them. 

The following classical lemma will be used repeatedly in the sequel. 

\begin{lemma}
	\label{lemma:quasiconvexTOgeo}
	Let $M$ be a $C$-quasiconvex metric space. Then there exists an equivalent metric $d_\gamma$ on $M$ such that
	\begin{equation}
	d \le d_\gamma \le C\,d,
	\end{equation}
	and $(M,d_\gamma)$ is a length space. In particular, if $M$ is compact, then $(M,d_\gamma)$ is a geodesic metric space.
\end{lemma}

We include the following sketch of the proof for the reader's convenience.

\begin{proof}
	Define a new metric $d_\gamma$ on $M$ by setting, for all $x,y\in M$,
	\begin{equation}
	d_\gamma(x,y):=\inf\{\mathrm{len}_d(\gamma): \gamma \text{ is a rectifiable curve joining } x \text{ and } y\},
	\end{equation}
    where $\mathrm{len}_d(\gamma)$ denotes the length of the curve $\gamma$ with respect to the metric $d$. 
	By the definition of rectifiable curves, we immediately have $d\le d_\gamma$. Moreover, since $M$ is $C$-quasiconvex, every pair of points $x,y$ can be joined by a rectifiable curve $\gamma$ satisfying $\mathrm{len}_d(\gamma)\le C\,d(x,y)$, hence $d_\gamma\le C\,d$.
	
	It remains to verify that $(M,d_\gamma)$ is a length space. This follows from the fact that for every rectifiable curve $\gamma:[a,b]\to M$, its length with respect to $d$ coincides with its length with respect to $d_\gamma$. Indeed, from $d\le d_\gamma$ it follows directly that $\len_d(\gamma)\le \len_{d_\gamma}(\gamma)$. For the converse inequality, let $\{a=t_1<t_2<\dots<t_n=b\}$ be a finite partition of $[a,b]$. By definition of $d_\gamma$, for each $i=1,\ldots,n$,
	\begin{equation}
	d_\gamma(\gamma(t_i),\gamma(t_{i+1}))\le \len_d(\gamma|_{[t_i,t_{i+1}]}).
	\end{equation}
	Summing over $i$ gives
	\begin{equation}
	\sum_{i=1}^{n-1} d_\gamma(\gamma(t_i),\gamma(t_{i+1}))
	\le \sum_{i=1}^{n-1} \len_d(\gamma|_{[t_i,t_{i+1}]})
	=\len_d(\gamma).
	\end{equation}
	Taking the supremum over all partitions of $[a,b]$ yields $\len_{d_\gamma}(\gamma)\le \len_d(\gamma)$, which proves the desired equality.
	
	Therefore $(M,d_\gamma)$ is a length space. By a general result (see, for instance, \cite[Lemma~8.3.11]{Heinonen}), if $M$ is compact, then $(M,d_\gamma)$ is a geodesic metric space. 
\end{proof}

\subsection{Nagata dimension}
\label{subsection:Nagata}

We begin by recalling some standard terminology concerning covers of metric spaces.  
Let $(M,d)$ be a metric space and let $\mathcal{B}$ be a family of subsets of $M$.  
We say that $\mathcal{B}$ is \emph{$D$-bounded} for some $D \ge 0$ if $\diam(B) \le D$ for every $B \in \mathcal{B}$.  
If, in addition, there exists a constant $s > 0$ such that every subset $A \subset M$ with $\diam(A) < s$ meets at most $d+1$ members of $\mathcal{B}$, we say that $\mathcal{B}$ has \emph{$s$-multiplicity at most $d+1$}.

\begin{defn}\label{def:nagata_dimension}
	Let $M$ be a metric space and $d \in \mathbb{N}$.  
	We denote by $\gamma_d(M)$ the smallest $\gamma > 0$ such that for every $s > 0$ there exists a $(\gamma s)$-bounded cover $\mathcal{C}$ of $M$ with $s$-multiplicity at most $d+1$, i.e.
	\begin{equation}
	\forall A \subset M, \ \diam(A) < s 
	\ \Rightarrow\ 
	\big|\{\, C \in \mathcal{C} : A \cap C \neq \emptyset \,\}\big| \le d+1.
	\end{equation}
	If no such $\gamma$ exists, we set $\gamma_d(M) = \infty$.  
	If $\gamma_d(M) < \infty$ for some $d$, we say that $M$ has \emph{finite Nagata dimension}, and we define the \emph{Nagata dimension} of $M$ by
	\begin{equation}
	\dim_N(M) := \min\{\, d \in \mathbb{N} : \gamma_d(M) < \infty \,\}.
	\end{equation}
\end{defn}

The Nagata dimension provides a metric analogue of the classical Lebesgue covering dimension.  It can also be seen as a variation of Gromov's asymptotic dimension. The class of metric spaces with finite Nagata dimension
includes all finite dimensional Banach spaces, doubling spaces, metric trees, among others (see \cite{LangSchli05} for more examples).

In the sequel, we will require the following result of Lang and Schlichenmaier~\cite[Proposition~4.1]{LangSchli05}, which provides a refinement of the covering properties implied by finite Nagata dimension.

\begin{proposition}[{\cite[Proposition~4.1]{LangSchli05}}]
	\label{prop:LangSchli}
	Let $M$ be a metric space with $\dim_N(M)\le d<\infty$.
	Then there exists a constant $c>0$ such that, for all sufficiently large $r>1$, there exists a sequence of coverings $\{\mathcal{B}^j\}_{j\in\mathbb{Z}}$ of $M$ satisfying the following properties:
	\begin{enumerate}[(i)]
		\item For every $j\in\mathbb{Z}$, one has $\mathcal{B}^j = \bigcup_{k=0}^d \mathcal{B}^j_k$,
		where each $\mathcal{B}^j_k$ is $c r^j$-bounded and has $r^j$-multiplicity at most~$1$.
		
		\item For every $x\in M$ and $j\in\mathbb{Z}$, there exists $C\in\mathcal{B}^j$ such that $B(x,r^j)\subset C$.
		
		\item For every $k\in\{0,\dots,d\}$ and every bounded set $B\subset M$, there exists $C\in\bigcup_{j\in\mathbb{Z}}\mathcal{B}^j_k$ such that $B\subset C$.
		
		\item Whenever $B\in\mathcal{B}^i_k$ and $C\in\mathcal{B}^j_k$ for some $k$ and $i<j$, one has either $B\subset C$ or $d(x,y)>r^i$ for every pair of points $x\in B$, $y \in C$.
	\end{enumerate}
\end{proposition}

Tracing back the constants in \cite[Proposition~4.1]{LangSchli05}, one finds that the conclusion holds with $c = 5c' + 4$ provided that $r \ge 5c' + 6$, where $c'$ is the covering constant derived from the equivalent formulations of the Nagata dimension $d$ in \cite[Proposition~2.5]{LangSchli05}. Examining the quantitative dependencies in that proof reveals that $c'$ grows at most quadratically with $d$, that is $c' \lesssim \gamma d^2$, where $\gamma$ denotes the Nagata constant. More precisely, the explicit choice $r = 50\,\gamma d^2$, also adopted in \cite[Lemma~5.1]{NaorSilberman}, is sufficiently large to ensure that Proposition~\ref{prop:LangSchli} can be applied without loss of generality.

\subsection{Wasserstein-1 distance and the Lipschitz-free space viewpoint}
\label{subsection:Free}

Let $M$ be a complete separable metric space.  
We denote by $\mathcal{P}_1(M)$ the set of all Borel probability measures on $M$ with finite first moment, i.e.
\begin{equation}
\int_M d(x_0,x)\, d\mu(x) < \infty
\end{equation}
for some (and hence every) $x_0 \in M$.  
The \emph{Wasserstein-1 distance} between two measures $\mu,\nu \in \mathcal{P}_1(M)$ is defined by
\begin{equation}
W_1(\mu,\nu)
:= \inf_{\pi \in \Pi(\mu,\nu)} 
\int_{M \times M} d(x,y)\, d\pi(x,y),
\end{equation}
where $\Pi(\mu,\nu)$ denotes the set of all couplings of $\mu$ and $\nu$, that is, all Borel probability measures on $M\times M$ whose marginals are $\mu$ and $\nu$ respectively.  
The space $(\mathcal{P}_1(M),W_1)$ is called the \emph{Wasserstein space} over~$M$.

\medskip

The fundamental \emph{Kantorovich--Rubinstein duality theorem} provides an equivalent dual formulation of $W_1$ in terms of Lipschitz functions:
\begin{equation}\label{eq:KRdual}
	W_1(\mu,\nu)
	= \sup\left\{
	\left| \int_M f\, d(\mu - \nu) \right|,\;  f:M\to \R \text{ 1-Lipschitz}\right\}.
\end{equation} 
This dual representation naturally suggests a linear-functional point of view, which is precisely captured by the \emph{Lipschitz-free space} construction.

\medskip

Let now $M$ be a pointed metric space.  
For each $x \in M$, every Dirac measure $\delta(x)$ can be seen as a linear functional acting on the space $\Lip_0(M)$ as follows:
\begin{equation}
\delta(x):\Lip_0(M)\to\mathbb{R}, \qquad \delta(x)(f) = f(x).
\end{equation} 
Each $\delta(x)$ is thus an element of $\Lip_0(M)^*$ and it is readily seen that $\|\delta(x)\| = d(x,0)$.  
The \emph{Lipschitz-free space} over $M$ is then defined by
\begin{equation}
\mathcal{F}(M)
:= \overline{\mathrm{span}} \{ \delta(x) : x \in M \}
\subset \Lip_0(M)^*,
\end{equation}
where the closure is taken with respect to the norm topology of $\Lip_0(M)^*$.
\medskip

This linear framework provides a powerful viewpoint: the map $x \mapsto \delta(x)$ is an isometric embedding of $(M,d)$ into the Banach space $\mathcal{F}(M)$, and the image $\delta(M)$ linearly generates $\mathcal{F}(M)$.  
Moreover, $\mathcal{F}(M)$ enjoys a fundamental \emph{universal (or extension) property}:  
for every Banach space $X$ and every Lipschitz map $f : M \to X$ with $f(0)=0$, there exists a unique bounded linear operator 
\begin{equation}
\overline{f} : \mathcal{F}(M) \to X
\end{equation}
such that $\overline{f} \circ \delta = f$ and $\|\overline{f}\| = \|f\|_L$.  
This property turns nonlinear Lipschitz maps into linear operators and allows one to study Lipschitz analysis within the framework of Banach space theory.
\medskip

As an immediate consequence of this universal property (applied with $X=\mathbb{R}$), the space $\mathcal{F}(M)$ is an isometric predual of $\Lip_0(M)$, i.e.
\begin{equation}
\Lip_0(M) = \mathcal{F}(M)^*.
\end{equation}
The corresponding weak$^*$ topology on $B_{\Lip_0(M)}$ coincides with the topology of pointwise convergence on $M$.  
When $M$ is compact, this topology also agrees with the topology of uniform convergence on~$M$.  
Another useful fact is that if $N\subset M$ contains the base point, then $\mathcal{F}(N)$ can be canonically identified with the closed subspace $\F_M(N)$ of $\mathcal{F}(M)$ generated by $\{\delta(x): x\in N\}$. This is an easy consequence of McShane-Whitney's extension theorem for Lipschitz maps.
\medskip

Finally, for $\gamma \in \mathcal{F}(M)$, we define its \emph{support} as the smallest closed set $K \subset M$ such that 
\begin{equation}
\gamma \in \mathcal{F}_M(K)
= \overline{\mathrm{span}}\{\delta(x): x \in K\}
\subset \mathcal{F}(M).
\end{equation}
For more on the support, we refer the reader to \cite{AP20, APPP2019}.

\subsection{Pe\l{}czy\'nski's property (V*)}
\label{subsection:Vstar}

Let $X$ be a Banach space. We begin by recalling some standard notions that are used to define Pe\l{}czy\'nski's property~(V*).

A series $\sum_{n=1}^\infty x_n$ in $X$ is said to be \emph{weakly unconditionally Cauchy} (abbreviated \emph{WUC}) if
\begin{equation}
\sum_{n=1}^\infty |\langle x^*,x_n\rangle| < \infty
\quad\text{for every } x^* \in X^*.
\end{equation}
We refer the reader to \cite[Fact~1.2]{APQ24} for some equivalent formulations. 
Every absolutely convergent series is WUC, but the converse is false in general.
\medskip

A subset $\Gamma \subset X$ is said to be a \emph{$V^*$-set} if for every WUC series $\sum x_n^*$ in $X^*$ one has
\begin{equation}
\lim_{n\to\infty} \sup_{x \in \Gamma} |\langle x, x_n^* \rangle| = 0.
\end{equation}
Every relatively weakly compact set is a $V^*$-set, while the converse is not true in general.  
This motivates the following definition.

\begin{definition}
	We say that $X$ has \emph{Pe\l{}czy\'nski's property~(V*)} if every $V^*$-set in $X$ is relatively weakly compact.  
\end{definition}

This property was introduced by Pe\l{}czy\'nski in~\cite{Pelczynski62}.  
All reflexive Banach spaces and all $L_1(\mu)$ spaces enjoy property~(V*), whereas the space $c_0$ of null sequences fails it.  Moreover, property~(V*) is stable under linear isomorphisms and under taking subspaces.

The relevance of property~(V*) in the theory of Lipschitz-free spaces is twofold. 
On the one hand, it provides criteria that help to identify weakly compact subsets for certain classes of metric spaces, thereby offering deeper insight into the intricate geometry and structure of Lipschitz-free spaces. 
On the other hand, property~(V*) implies \emph{weak sequential completeness}, that is, every weakly Cauchy sequence in $X$ is weakly convergent. 
This implication is particularly significant in the context of the nonlinear classification of Banach spaces.
Indeed, it remains an open problem whether $\mathcal{F}(\ell_1)$ is weakly sequentially complete. 
A positive answer would entail that $c_0$ does not embed into $\mathcal{F}(\ell_1)$, thereby removing the main obstruction to $\mathcal{F}(\ell_1)$ being complemented in its bidual. 
In turn, this would imply that $\ell_1$ has a unique Lipschitz structure, a long-standing open question in nonlinear Banach space theory. 
A brief overview of property~(V*) in the setting of Lipschitz-free spaces will be given at the beginning of Sections~\ref{section:ACUG} and~\ref{subsection:V*}.

\section{Affine Approximation Relative to an Atlas and Stochastic Almost Retractions}
\label{section:atlas}

The notions introduced in this section are strongly inspired by the seminal work of Cheeger~\cite{Che99}, which later led to the development of the theory of \emph{Lipschitz differentiability spaces}, following the contributions of several authors, most notably Keith~\cite{Keith}.

\subsection{Definitions}

Let $M$ be a metric space and $X$ a Banach space. A pair $(U,\varphi)$ is called a \textit{chart modeled on $X$ for $M$} if $U$ is a Borel subset of $M$ and $\varphi\colon U\rightarrow X$ is a Lipschitz map. An \textit{atlas of $M$ modeled on $X$} is a countable collection of charts $(U_i,\varphi_i)_{i\in I}$ modeled on closed subspaces $X_i\subset X$, such that $M=\bigcup_{i}U_i$ and the Lipschitz constants of the maps $(\varphi_i)_{i\in I}$ are uniformly bounded.

An atlas of a metric space $M$ modeled on $X$ provides a coordinate structure on $M$, which is particularly useful when dealing with Lipschitz curves or differentiability-type properties.

\begin{defn}[Affine relative to an atlas]\label{def:affine_wrt_atlas}
	Let $\mathcal{U}=(U_i,\varphi_i)_{i\in I}$ be an atlas of a metric space $M$ modeled on a Banach space $X$. Given a Banach space $Y$ and a map $f\colon M\rightarrow Y$, we say that $f$ is \textit{affine with respect to $\mathcal{U}$} if, for every $i\in I$, there exists a unique bounded linear operator $(Af)_i\colon X_i\rightarrow Y$ such that 
	\begin{equation} f(y)-f(x)=(Af)_i(\varphi_i(y)-\varphi_i(x)) \quad \text{for all } x,y\in U_i.\end{equation}
    The operators $(Af)_i$ will be referred to as \textit{the affine coefficients of $f$.}
	For such a map $f\colon M\rightarrow Y$, define $\| f \|_{\mathcal{U}} \colon M\rightarrow \R^+\cup\{+\infty\}$ by 
	\begin{equation}\|f\|_{\mathcal{U}}(x):=\sup\{\Lip(\varphi_i)\|(Af)_i\|\colon i\in I, \; x\in \text{int}(U_i)\},\end{equation}
	if $x$ belongs to $\text{int}(U_i)$ for some $i$, and $\|f\|_\mathcal{U}(x):=+\infty$ otherwise. We call $\|f\|_{\mathcal{U}}$ the \emph{maximal norm of $f$ associated to $\mathcal{U}$}.
\end{defn}

Note that for every metric space $M$, every Banach space $Y$, and every Lipschitz map $f\colon M\rightarrow Y$, the map $f$ is affine (indeed, linear) with respect to the atlas modeled on the Lipschitz-free space $\mathcal{F}(M)$, consisting of a single chart given by the canonical isometric embedding of $M$ into $\mathcal{F}(M)$. In this case, $Af$ coincides with the canonical linearization $\overline{f}$. However, our interest typically lies in functions that are affine with respect to \emph{more regular} atlases, namely those modeled on Banach spaces possessing additional geometric or structural properties (such as finite-dimensional, reflexive, or superreflexive spaces). 

The uniqueness of the operator $(Af)_i$ associated to a Lipschitz function that is affine with respect to an atlas $(U_i,\varphi_i)_{i\in I}$ depends only on the atlas, and not on the particular function $f$. 
This parallels the situation in the Lipschitz differentiability framework, as shown in \cite[Lemma~2.1]{BatSpe13}. 
In the case of Lipschitz affine functions, the characterization of uniqueness is as follows.

\begin{lemma}
\label{Lemma:Uniqueness_affine}
Let $M$ be a metric space, and let $(U,\varphi)$ be a chart of $M$ modeled on a Banach space~$X$. 
Let $Y$ be another Banach space. 
Suppose that $f \colon M \to Y$ is Lipschitz and that there exists a bounded linear operator $Af \colon X \to Y$ such that 
\begin{equation}\forall\, x,y \in U, \qquad 
f(y) - f(x) = Af\big(\varphi(y) - \varphi(x)\big).\end{equation}
Then $Af$ is the unique bounded linear operator satisfying these conditions 
if and only if the linear span of $\varphi(U) - \varphi(U) = \{ v - w \in X : v,w \in \varphi(U) \}$ is dense in $X$.
\end{lemma}

\begin{proof}
Assume first that $\operatorname{span}(\varphi(U) - \varphi(U))$ is dense in~$X$. 
Let $G \colon X \to Y$ be another bounded linear operator such that 
\begin{equation}
f(y) - f(x) = G\big(\varphi(y) - \varphi(x)\big)
\qquad \text{for all } x,y \in U.
\end{equation}
Then, for every $x,y \in U$, \begin{equation}G\big(\varphi(y) - \varphi(x)\big)
= Af\big(\varphi(y) - \varphi(x)\big).\end{equation}
Hence $Af(v) = G(v)$ for all $v \in \varphi(U) - \varphi(U)$. 
By linearity and continuity, this equality extends to all $v \in X$, and therefore $G = Af$.

Conversely, suppose that the span of $\varphi(U) - \varphi(U)$ is not dense in~$X$. 
Then there exists $w \in X \setminus \overline{\operatorname{span}}(\varphi(U) - \varphi(U))$. 
By the Hahn--Banach separation theorem, we can find a functional $x^* \in X^*$ such that 
$x^*(w) \neq 0$ and $x^*(v) = 0$ for all $v \in \overline{\operatorname{span}}(\varphi(U) - \varphi(U))$. Fix any nonzero vector $y_0 \in Y$, and define $G \colon X \to Y$ by
\begin{equation}\forall v \in X,  \quad G(v) := Af(v) + x^*(v) \, y_0.\end{equation}
Then $G$ is a bounded linear operator distinct from $Af$. 
Moreover, for every $x,y \in U$,
\begin{equation}
G\big(\varphi(y) - \varphi(x)\big)
= Af\big(\varphi(y) - \varphi(x)\big) 
  + x^*\big(\varphi(y) - \varphi(x)\big) y_0
= f(y) - f(x),
\end{equation}
since $x^*$ vanishes on $\varphi(U) - \varphi(U)$. 
Hence $Af$ is not unique.
\end{proof}

\medskip

In light of Lemma~\ref{Lemma:Uniqueness_affine}, we will always assume that, 
for every chart $(U,\varphi)$, the modeling Banach space is taken to be 
$\overline{\operatorname{span}}(\varphi(U) - \varphi(U))$. 
This assumption entails no loss of generality, since $\varphi(U)$ 
can always be translated to include the origin, so that $\varphi(U) \subset \overline{\lspan}(\varphi(U) - \varphi(U))$ (that is, $\overline{\lspan}(\varphi(U)) =\overline{\lspan}(\varphi(U) - \varphi(U))$). With this convention, the operator $Af$ associated to a Lipschitz affine function is automatically unique.

Finally, note that whenever a function $f$ is affine with respect to an atlas $(U_i,\varphi_i)_{i\in I}$, 
the assignments $f \mapsto (Af)_i$ are linear for each $i \in I$. 
Indeed, for any pair of Lipschitz affine functions $f,g \colon M \to Y$ with respect to $(U_i,\varphi_i)_{i\in I}$ 
and for any $\alpha,\beta \in \mathbb{R}$, 
the function $h:=\alpha f + \beta g$ is affine with respect to the same atlas, with
\begin{equation}
h(y) - h(x)
= (\alpha (Af)_i + \beta (Ag)_i)\big(\varphi_i(y) - \varphi_i(x)\big)
\quad \forall x,y \in U_i, ~\forall i \in I.
\end{equation}
By uniqueness, we obtain 
$(Ah)_i = \alpha (Af)_i + \beta (Ag)_i$ for all $i \in I$.
\medskip

% \noindent \textbf{Notation.} We denote by $A(M,\mathcal{U},X,Y)$ the set of Lipschitz functions from a metric space $M$ into a Banach space $Y$ that are affine with respect to an atlas $\mathcal{U}$ modeled on a Banach space $X$. Clearly, $A(M,\mathcal{U},X,Y)$ is a linear subspace of $\Lip(M,Y)$.

\begin{remark}\label{rem:refinement}
	If a function $f\colon M\rightarrow Y$ is affine with respect to an atlas $\mathcal{U}$, then it remains affine with respect to any \emph{refinement} of $\mathcal{U}$, that is, any other atlas $\mathcal{V}=(V_j,\psi_j)_{j \in J}$ where $(V_j)_{j \in J}$ refines $(U_i)_{i\in I}$ and $\psi_j$ is the restriction of $\varphi_i$ to $V_j$ for some $i$ such that $V_j\subset U_i$ (with a possible restriction of the range of $\psi_j$ to ensure uniqueness as per Lemma \ref{Lemma:Uniqueness_affine}). 
    
    % For such pairs of atlases, we have $A(M,\mathcal{U},X,Y)\subset A(M,\mathcal{V},X,Y)$.
\end{remark}

Even though $\mathcal C^1$-maps will not play a central role in this paper, it is natural to introduce the corresponding notion of differentiability relative to an atlas and to compare it with the affine notion discussed above.

\begin{defn}[$\mathcal C^1$-smooth relative to an atlas]
	Let $\mathcal{U}=(U_i,\varphi_i)_{i\in I}$ be an atlas of a metric space $M$ modeled on a Banach space $X$. A function $f\colon M \to Y$ is said to be \textit{$\mathcal C^1$-smooth with respect to $\mathcal{U}$} if, for every $i\in I$ and every $x\in U_i$, there exists a unique bounded linear operator $(Df)_i(x)\colon  X_i\rightarrow Y$ such that 
	\begin{equation} \lim_{\substack{y\rightarrow x \\ y \in U_i}}\frac{\|(f(y)-f(x))-(Df)_i(x)(\varphi_i(y)-\varphi_i(x)) \|}{d(y,x)}=0,\end{equation}
	and such that the map $(Df)_i\colon U_i\rightarrow L(X_i,Y)$ is continuous. For $x\in U_i$, $(Df)_i(x)$ is called the derivative of $f$ at $x$ and we say that $f$ is differentiable at $x$ with respect to $(U_i, \varphi_i)$ whenever the operator $(Df)_i(x)$ exists. For such an $f$, define $|Df|_{\mathcal{U}}\colon M\rightarrow \R^+\cup\{+\infty\}$ by 
	\begin{equation}|Df|_{\mathcal{U}}(x):=\sup\{\Lip(\varphi_i)\|(Df)_i(x)\|\colon i\in I, \;  x\in\text{int}(U_i)\},\end{equation}
	if $x$ belongs to $\text{int}(U_i)$ for some $i$, and $|Df|_\mathcal{U}(x):=+\infty$ otherwise. The function $|Df|_{\mathcal{U}}$ is called the \emph{maximal upper gradient of $f$ associated to $\mathcal{U}$}.
\end{defn}

It is clear that every function $f$ which is affine with respect to an atlas $\mathcal{U}=(U_i,\varphi_i)_{i\in I}$ is $\mathcal C^1$-smooth with respect to the same atlas, with $(Df)_i(x)=(Af)_i$ for every $i$ and every $x\in U_i$. In this case, we also have $\|f\|_{\mathcal{U}} = |Df|_{\mathcal{U}}$.

\begin{remark}
	If $\mathcal{U}=(U_i,\varphi_i)_{i \in I}$ is an atlas of a metric space $M$ and $f$ is $\mathcal C^1$-smooth with respect to $\mathcal{U}$, then we can obtain an upper bound on the local Lipschitz constant of $f$ using the maximal upper gradient $|Df|_{\mathcal{U}}(x)$. Indeed, we have for any $x,y\in U_i$:
	\begin{equation}f(y)=f(x)+(Df)_i(x)(\varphi_i(y)-\varphi_i(x))+o(d(x,y)),\end{equation}
	which, if $x\in\text{int}(U_i)$, implies that
	\begin{equation} \Lip(f,x)=\limsup_{\substack{y\rightarrow x\\ y \in U_i}}\frac{\|f(y)-f(x)\|}{d(x,y)}\leq \|(Df)_i(x)\|\Lip(\varphi_i,x).\end{equation}
	In particular, this shows that $|Df|_\mathcal{U}$ is indeed an upper gradient of $f$ thanks to Lemma \ref{lem:local-Lip-upper-gradient}.
\end{remark}
\medskip

As in the affine case, and inspired from \cite[Lemma~2.1]{BatSpe13}, we now address the question of uniqueness of the derivative associated with a given chart.

\begin{lemma}\label{lem:uniqueness-C1-atlas}
Let $(U,\varphi)$ be a chart of a metric space $M$ modeled on a Banach space $X$, and fix $x_0\in U$. Let $Y$ be a Banach space and $f:M\to Y$ be differentiable at $x_0$ with respect to $(U,\varphi)$. 
Then the following are equivalent:
\begin{enumerate}[(i)]
\item  The derivative of $f$ at $x_0$ is unique.

\item For every nonzero $x^*\in X^*$:
\begin{equation}\limsup_{\substack{x \to x_0\\ x \in U}}\frac{\big|\langle x^*,\,\varphi(x)-\varphi(x_0)\rangle\big|}{d(x,x_0)}>0.\end{equation}
\end{enumerate}
\end{lemma}

\begin{proof}
$(ii) \implies (i)$.
Assume $(ii)$ and suppose that the operators $T_1,T_2\in\mathcal L(X,Y)$ are witnessing the fact that $f:M\to Y$ is differentiable at $x_0$. Set $S:=T_1-T_2$. If $S\neq 0$, choose $y^*\in Y^*$ so that $z^*:=y^*\circ S\in X^*$ is nonzero. By $(ii)$, there exists a sequence $(x_n) \subset U$ with $x_n\to x_0$ and
\begin{equation}\limsup_{n}\frac{\big|\langle z^*,\,\varphi(x_n)-\varphi(x_0)\rangle\big|}{d(x_n,x_0)}>0.\end{equation}
On the other hand, differentiability with $T_1$ and $T_2$ yields that
\begin{equation}\lim_{x\to x_0}\frac{\big\|f(x)-f(x_0)-T_i(\varphi(x)-\varphi(x_0))\big\|}{d(x,x_0)}=0\qquad (i=1,2).\end{equation}
Therefore:
\begin{align}
\frac{\big|\langle z^*,\,\varphi(x_n)-\varphi(x_0)\rangle\big|}{d(x_n,x_0)} 
&= \frac{\big|\langle y^*,\,S(\varphi(x_n)-\varphi(x_0))\rangle\big|}{d(x_n,x_0)} \\
&\leq  \sum_{i=1}^2\frac{\big|\langle y^*,\,f(x_n)-f(x_0)-T_i(\varphi(x_n)-\varphi(x_0))\rangle\big|}{d(x_n,x_0)} \nonumber\\
&\leq \|y^*\| \sum_{i=1}^2 \frac{\big\|f(x)-f(x_0)-T_i(\varphi(x)-\varphi(x_0))\big\|}{d(x,x_0)} \underset{n \to 0}{\longrightarrow} 0,\nonumber
\end{align}
which is a contradiction. Hence $S=0$ and $T_1=T_2$.
\medskip

$(i)\implies (ii)$.
Suppose $(ii)$ fails. Then there exists a nonzero $x^*\in X^*$ such that
\begin{equation}\limsup_{x\to x_0}
\frac{\big|\langle x^*,\,\varphi(x)-\varphi(x_0)\rangle\big|}{d(x,x_0)} = 0.\end{equation}
Let $T_1:X \to Y$ be a bounded operator witnessing that $f:M\to Y$ is differentiable at $x_0$ with respect to $(U,\varphi)$. Let $y \in Y\setminus \{0\}$ and define $T_2 : X \to Y$ by 
\begin{equation}
T_2(v) := T_1(v) + x^*(v)y \qquad \text{for all } v \in X.
\end{equation}
It is straightforward to check that $T_2$ is also a derivative of $f$ at $x_0$ and $T_1\neq T_2$, which contradicts $(i)$. 
\end{proof}

In particular, uniqueness of derivatives at $x_0$ depends only on the chart $(U,\varphi)$ and not on the function $f$.
\bigskip

\begin{remark}
\label{Remark:Derivative_0_and_local_flatness}
%\mnote{\ach{Shouldn't this be true for $f\colon M\rightarrow Y$ where $Y$ is any Banach space?} \mch{I changed it just because we defined $\Lip(f,\cdot)$ for real valued $f$... :-)} \ach{Oh I see! Maybe we should define it more generally though, because we also use the Banach-valued version in Remark 2.5. I'll make the change if you agree! :)} \mch{Oh, you are right. Yes, then we need to define it more generally!} \ach{I think everything works if we extend the definitions of local Lipschitz constant and upper gradients to Banach-valued maps. I have made the changes in the preliminaries.} \cch{Thanks a lot for your work. \checkmark} }    
It is easy to see that, for any metric space $M$, given the trivial chart $(M,\varphi)$ modeled on the Banach space $\{0\}$, a function $f\colon M\rightarrow Y$ is $\mathcal C^1$-smooth with respect to $(M,\varphi)$ if and only if it is locally flat (i.e. $\Lip(f,x)=0$ for all $x\in M$). Indeed, both conditions are equivalent to 
    \begin{equation}
        f(y)=f(x)+o(d(x,y))\qquad \forall x,y\in M.
    \end{equation}
\end{remark}

% \mnote{\cch{If you agree, I would remove this notation block. It seems that we never use it. And the order of the letters might be confusing (I guess only for $X$ and $Y$ which are both Banach spaces). What are your thoughts on this?} \ach{I agree, seems unnecessary.}}
% \noindent\textbf{Notation.} We denote by $\mathcal C^1(M,\mathcal{U},X,Y)$ the set of Lipschitz functions from a metric space $M$ into a Banach space $Y$ that are $\mathcal C^1$-smooth with respect to an atlas $\mathcal{U}$ modeled on a Banach space $X$. As in the affine case, $\mathcal C^1(M,\mathcal{U},X,Y)$ is a linear subspace of $\Lip(M,Y)$. 

% Similarly, if $f\colon M\rightarrow Y$ is $\mathcal C^1$-smooth with respect to an atlas $\mathcal{U}$, then it remains $\mathcal C^1$-smooth with respect to any refinement $\mathcal{V}$ of $\mathcal{U}$. Therefore, for such $\mathcal{U}$ and $\mathcal{V}$, we have $\mathcal C^1(M,\mathcal{U},X,Y)\subset \mathcal C^1(M,\mathcal{V},X,Y)$.

\subsection{Moduli of approximation}

We introduce two parameters associated to any metric space, which quantify to what extent Lipschitz functions can be approximated by Lipschitz affine (respectively, $\mathcal C^1$-smooth) functions with respect to an open atlas modeled on a given Banach space.

\begin{defn}
	Let $M$ be a metric space, $X$ be a Banach space, and let $\varepsilon>0$. Given an open atlas $\mathcal{U}$ of $M$ modeled on $X$, we define
	\begin{itemize}[leftmargin=*]
		\item the \textit{modulus of affine $\varepsilon$-approximation of the atlas $\mathcal{U}$}, $AppA_X(M,\mathcal{U},\varepsilon)$, as the infimum over all $K\geq 1$ such that for every Banach space $Y$ and for any $1$-Lipschitz map $f\in\Lip(M,Y)$, there exists a $K$-Lipschitz function $g\in \Lip(M,Y)$ with $\|g-f\|_\infty<\varepsilon$ such that $g$ is affine with respect to $\mathcal{U}$, and $\|g\|_{\mathcal{U}}\leq K$. 
		\smallskip
		
		Further, we define the \textit{modulus of affine approximation of $M$ modeled on $X$} as
		\begin{equation}AppA_X(M):=\sup_{\varepsilon> 0}\inf\{AppA_X(M,\mathcal{U},\varepsilon)\colon \mathcal{U}\text{ open atlas of }M\text{ modeled on }X\}. \end{equation}
		
		For $d\in \N$, we write $AppA_d(M):=AppA_{(\R^d,\|\cdot\|_2)}(M)$. In case $d=0$, we understand $AppA_0 (M)$ as $AppA_X (M)$ where $X=\{0\}$.
		\medskip
		
		\item the \textit{modulus of $\mathcal C^1$-smooth $\varepsilon$-approximation of the atlas $\mathcal{U}$}, $AppC^1_X(M,\mathcal{U},\varepsilon)$, as the infimum over all $K\geq 1$ such that for every Banach space $Y$ and for any $1$-Lipschitz map $f\in\Lip(M,Y)$ there exists a $K$-Lipschitz function $g\in \Lip(M,Y)$ with $\|g-f\|_\infty<\varepsilon$ such that $g$ is $\mathcal C^1$-smooth with respect to $\mathcal{U}$, and $|Dg|_{\mathcal{U}}\leq K$.
		\smallskip
		
		Further, we define the \textit{modulus of $\mathcal C^1$-smooth approximation of $M$ modeled on $X$} as
		\begin{equation}AppC^1_X(M)=\sup_{\varepsilon> 0}\inf\{AppC_X^1(M,\mathcal{U},\varepsilon)\colon \mathcal{U}\text{ is an open atlas of }M\text{ modeled on }X\}. \end{equation}
		For any $d\in \N$, we write $AppC^1_d(M):=AppC^1_{(\R^d,\|\cdot\|_2)}(M)$. Similarly, we write $AppC_0^1 (M):= AppC^1_X (M)$, where $X = \{0\}$.
	\end{itemize}
\end{defn}

Plainly, $AppA_X(M)\leq K$ (respectively, $AppC^1_X(M)\leq K$) if and only if for every $\varepsilon>0$ there exists an atlas $\mathcal{U}=\mathcal{U}(\varepsilon)$ of $M$ modeled on $X$ such that, for any Banach space $Y$, every $1$-Lipschitz map $f\colon M\rightarrow Y$ can be uniformly approximated with error less than $\varepsilon$ by a $K$-Lipschitz map $g\colon M\rightarrow Y$ which is affine (respectively, $\mathcal C^1$-smooth) with respect to $\mathcal{U}$ and satisfies $\|g\|_\mathcal{U}\leq K$ (respectively, $|Dg|_\mathcal{U}\leq K$).

Since every function that is affine with respect to a given atlas is automatically $\mathcal C^1$-smooth with respect to the same atlas (with constant derivative), any upper bound valid for affine approximations also applies to $\mathcal C^1$-approximations. In particular, 
\begin{equation}
AppC^1_X(M)\leq AppA_X(M).
\end{equation} 

\begin{remark}
	 %this holds (for instance) for a purely $1$-unrectifiable compact metric space $M$ which is connected and contains more than one point. On the one hand, it is easily shown that for such spaces, the only affine functions with respect to an atlas modeled on $\{0\}$ are the constant functions, and thus $AppA_0(M)=\infty$. On the other hand, by Bate's result \cite[Lemma~3.4]{Bate20}, every $1$ Lipschitz function can be approximated by a uniformly locally flat one with minimal increase in the Lipschitz constant, and thus $AppC^1_0(M)=1$ thanks to Remark \ref{Remark:Derivative_0_and_local_flatness}. \mnote{Actually there's an issue here... Bate's result works for real-valued functions and we have defined the modulus for Banach-valued ones. Is it open whether a vector-valued approximation with locally flat works too?}
\label{Remark:AppC1_0_diff_AppA_0}
%\mnote{\cch{I moved Andrés remark here, and I added a concrete example which avoids Bate's Lemma, since this result is for real-valued maps. It surely detailed to much so, whenever you guys certify that it works, we can shrink it significantly. } \ach{I think everything works! Do you think this would hold for more general $M$ with the same argument? Surely for $M\subset \R^n$ it would work the same right? Even maybe for $M\subset X$ when $X$ is a Banach space with a $C^1$ bump? It's probably not worth writing here anyway, but do you think it works?} \mch{I also agree with the argument! Nice! Not so important but we used the notation $\mathcal{C}^1((x,y))$ in the proof of Proposition 5.6. Do you want to replace $C_c^\infty$} with $\mathcal{C}_c^\infty$? \ach{\checkmark}  \ach{I shortened the proof. Feel free to shorten it more or undo any changes.} \cch{\checkmark}}
In this paper, we do not attempt to further compare the quantities $AppA_d(M)$ and $AppC^1_d(M)$ in general. One reason is that obtaining nontrivial lower bounds for either modulus appears to be a difficult problem when $d\neq 0$, and pursuing such estimates would take us beyond the scope of this article. Nevertheless, in the case $d=0$ one can exhibit compact connected metric spaces for which $AppA_0(M)=\infty$ while $AppC^1_0(M)$ is finite, using only elementary arguments. In fact, it is a routine exercise to check that if $M$ is connected then $AppA_0 (M) = \infty$.
 
Fix $\alpha\in(0,1)$ and consider the snowflaked interval $M_\alpha:=([0,1],d_\alpha)$, where $d_\alpha(s,t):=|s-t|^\alpha$. The space $M_\alpha$ is compact and connected, and purely $1$-unrectifiable. By connectedness, $AppA_0(M_\alpha)=\infty$. On the other hand, we now show that $AppC^1_0(M_\alpha)=1$.
 
Let $Y$ be a Banach space, let $f:M_\alpha\to Y$ be $1$-Lipschitz, and fix $\varepsilon>0$. Extend $f$ to a map $\tilde f:\mathbb R\to Y$ by setting
 $\tilde f(t)=f(0)$ if $t<0$, $\tilde f(t)=f(t)$ if $t\in[0,1]$, and $\tilde f(t)=f(1)$ if $t>1$. Then $\tilde f$ is still 1-Lipschitz with respect to $d_\alpha$. Let $\rho\in \mathcal C_c^\infty(\mathbb R)$ be a standard mollifier with $\rho\ge 0$ and $\int_{\mathbb R}\rho=1$. For $\delta>0$ set $\rho_\delta(u)=\delta^{-1}\rho(u/\delta)$ and define
\begin{equation}
f_\delta(t):=(\tilde f*\rho_\delta)(t)=\int_{\mathbb R}\tilde f(t-u)\rho_\delta(u)\,du \in Y,\qquad t\in\mathbb R,
\end{equation}
where the integral is a Bochner integral. Then $f_\delta\in \mathcal C^\infty(\mathbb R;Y)$ and an easy computation shows that $f_\delta\restriction_{[0,1]}$ is $1$-Lipschitz on $M_\alpha$.

% . Moreover, for $s,t\in\mathbb R$,
%  \begin{align*}
%  	\|f_\delta(s)-f_\delta(t)\|
%  	&=\left\|\int_{\mathbb R}\rho_\delta(u)\big(\tilde f(s-u)-\tilde f(t-u)\big)\,du\right\|\\
%  	&\le \int_{\mathbb R}\rho_\delta(u)\,\|\tilde f(s-u)-\tilde f(t-u)\|\,du
%  	\le \int_{\mathbb R}\rho_\delta(u)\,|s-t|^\alpha\,du
%  	=|s-t|^\alpha,
%  \end{align*}
%  so $f_\delta\restriction_{[0,1]}$ is $1$-Lipschitz on $M_\alpha$. 

Next, for $t\in[0,1]$,
 \begin{align}
 	\|f_\delta(t)-f(t)\|
    %&=\left\|\int_{\mathbb R}\rho_\delta(u)\big(\tilde f(t-u)-\tilde f(t)\big)\,du\right\|
 	%\le \int_{\mathbb R}\rho_\delta(u)\,\|\tilde f(t-u)-\tilde f(t)\|\,du \\
 	&\le \int_{\mathbb R}\rho_\delta(u)\,|u|^\alpha\,du
 	= \delta^\alpha\int_{\mathbb R}\rho(v)|v|^\alpha\,dv
 	=:C_\rho\,\delta^\alpha.
 \end{align}
 Hence $\|f_\delta\restriction_{[0,1]}-f\|_\infty\le C_\rho\,\delta^\alpha$, so it is enough to choose $\delta>0$ so that $C_\rho\,\delta^\alpha<\varepsilon$. Finally, since $f_\delta\in \mathcal C^1(\mathbb R;Y)$, the mean value inequality yields
 $\|f_\delta(x)-f_\delta(y)\|\le \|f_\delta'\restriction_{[0,1]}\|_\infty\,|x-y|$ for $x,y\in[0,1]$, and therefore
 \begin{equation}
 \frac{\|f_\delta(x)-f_\delta(y)\|}{|x-y|^\alpha}\le \|f_\delta'\restriction_{[0,1]}\|_\infty\,|x-y|^{1-\alpha}\xrightarrow[y\to x]{}0.\end{equation}
 Thus $f_\delta\restriction_{[0,1]}$ is locally flat, hence $\mathcal C^1$-smooth with respect to the trivial atlas modeled on $\{0\}$ (Remark~\ref{Remark:Derivative_0_and_local_flatness}). This proves $AppC^1_0(M_\alpha)\le 1$, and since $AppC^1_0(M_\alpha)\ge 1$ by definition, we conclude that $AppC^1_0(M_\alpha)=1$.
\end{remark}

\subsection{From Stochastic Almost Retraction to Affine Approximation}
\label{Section:SARtoAA}

%-----------------------------------------------------------

%\sout{The purpose of this section is to} \ach{In Section \ref{section:AppA-FND} we will} show that a space $M$ with Nagata dimension $d$ admits a finite modulus of affine approximation modeled on $\R^d$. The strategy will be to suitably almost-extend certain functions defined on a subset $S\subset M$ to \sout{$\mathcal C^1$-smooth} \ach{affine} functions with respect to an atlas modeled on $\R^{\dim_N(M)}$. 

%----------------------------------------------------------

%\ach{I propose changing the above paragraph for the following one. I think the previous was based on something I wrote at the beginning, but I believe with the current structure of the paper we can modify it to be a bit clearer. The first sentence about Nagata dimension could even be omitted maybe.} \cch{\checkmark}  \ach{I commented out this discussion and the previous version}

%In Section \ref{section:AppA-FND} we will show that a space $M$ with Nagata dimension $d$ admits a finite modulus of affine approximation modeled on $\R^d$. 

Approximation of Lipschitz functions is related to Lipschitz extension problems in a natural way. Indeed, to approximate a Lipschitz map $f\colon M\rightarrow Y$ with a Lipschitz map $g\colon M\rightarrow Y$ enjoying some desirable property $P$, one may first restrict $f$ to a sufficiently dense subset $S\subset M$ such that the restriction $f\restricted_{S}$ satisfies $P$. If $f\restricted_{S}$ can be extended to a Lipschitz map $g\colon M\rightarrow Y$ preserving $P$ and with controlled Lipschitz constant, then the Lipschitz condition on $g$, together with the density of $S$, will ensure that $g$ provides the desired approximation. This is roughly the approach we will follow, with the caveat that $g$ need not be an exact extension of $f\restricted_{S}$: a sufficiently good almost-extension yields the desired approximation. 

%------------------------------------------------------------
%\ach{such that for every $y\in S$, the measure $\mu_x$}

A powerful approach to construct an (almost) extension of a Lipschitz function $f$ defined on a subset $S$ of a metric space $M$ is to associate to each $x\in M$ a probability measure $\mu_x\in\mathcal{P}_1(S)$ that is close (in a suitable sense) to a Dirac mass at points of $S$ near $x$. The map $g(x)$ is then defined by integrating $f$ with respect to $\mu_x$, which yields an (almost) extension of $f$ to $M$. If the map $x\mapsto \mu_x$ is Lipschitz for the Wasserstein $1$-distance in $\mathcal{P}_1(M)$, the the resulting map $g$ will also be Lipschitz.

This probabilistic extension method appeared implicitly in~\cite{LeeNaorInv} and was later formalized in~\cite{Ohta09}, where the map $x\mapsto \mu_x$ is called a \emph{stochastic retraction} onto $S$. In~\cite{AmbPug20}, a similar construction was studied under the name of a \emph{strong random projection}. In that work, the term \emph{random projection} was used for a map $x\mapsto \mu_x$ where $\mu_x$ is merely an element of the Lipschitz-free space $\mathcal{F}(S)$, and not necessarily a probability measure. 

In our setting, although we will indeed construct probability measures supported on $S$, we will nevertheless regard them as elements of the Lipschitz-free space $\mathcal{F}(S)$. This viewpoint is advantageous, since the linear structure of $\mathcal{F}(S)$ allows one, under certain conditions, to use the map $x\mapsto \mu_x$ to obtain coordinate charts from $M$ into $\R^d$, which together form the desired atlas. In particular, this construction bridges the probabilistic nature of stochastic retractions with the linear framework required for affine approximation. With this in mind, we introduce the following definition.

\begin{defn}
	Let $M$ be a pointed metric space, and let $S\subset M$ be a subset containing the base point $0$. For $\varepsilon>0$, a Lipschitz map $P\colon M\rightarrow \mathcal{F}(S)$ is called a \emph{stochastic $\varepsilon$-almost retraction} if $P(x)$ is a probability measure for all $x\in M$ and 
	\begin{equation}\|P(x)-\delta(x)\|_{\mathcal{F}(S)}<\varepsilon \quad \text{for all } x\in S.\end{equation}
	
	The map $P$ is said to have \emph{multiplicity} $d\in\N$ if the support of $P(x)$ contains at most $d$ points for every $x\in M$. 
	
	It has \emph{strong multiplicity} $d\in\N$ if there exists $\tau>0$ such that, for every $x\in M$, there is a set $A_x\subset S$ with $|A_x|\leq d$ and 
	\begin{equation}\text{supp}(P(y))\subset A_x \quad \text{for all } y\in B(x,\tau).\end{equation}
\end{defn}
The choice of the base point $0\in S$ is not important in the previous definition. In particular, if a stochastic $\varepsilon$-almost retraction exists for a specific choice of base point $0$, an equivalent stochastic $\varepsilon$-almost retraction with the same Lipschitz constant can be defined for any other choice of base point.

We now introduce some notation that will be used repeatedly. Given a finite nonempty set $A\subset M$, define
\begin{equation}K_A:=\text{conv}\left\{\delta(x)\colon x\in A\right\}\subset \mathcal{F}(M).\end{equation}
Note that every probability measure supported on $A$ belongs to $K_A$. 

With this notation, a stochastic almost retraction $P\colon M\rightarrow\mathcal{F}(S)$ has strong multiplicity $d\in\N$ if there exists $\tau>0$ such that, for every $x\in M$, there is a set $A_x\subset S$ with $|A_x|\leq d$ satisfying 
\begin{equation}P(B(x,\tau))\subset K_{A_x}.\end{equation}

Under the same assumptions on $A$, define
\begin{equation}\mathcal{F}_A:=\text{span}\left\{\delta(x)-\frac{1}{|A|}\sum_{z\in A}\delta(z)\colon x\in A\right\}\subset\mathcal{F}(M).\end{equation}
Then $\mathcal{F}_A$ is a $(|A|-1)$-dimensional subspace of $\mathcal{F}(M)$, and the convex set $K_A$ is contained in the $(|A|-1)$-dimensional affine space
\begin{equation}\frac{1}{|A|}\sum_{z\in A}\delta(z)+\mathcal{F}_A.\end{equation}

The following theorem describes precisely how the desired atlas and affine extensions can be obtained from a stochastic $\varepsilon$-almost retraction with finite strong multiplicity.
\begin{theorem}[Atlas induced by stochastic almost retractions]
	\label{Theorem:Atlas_from_random_almost_projections}
    Let $M$ be a metric space and  $S\subset M$ be countable. Fix $\varepsilon>0$ and $d\in\N$. Suppose that $P\colon M\rightarrow \mathcal{F}(S)$ is a stochastic $\varepsilon$-almost retraction with strong multiplicity $d+1$. Then there exists an open atlas $\mathcal{U}$ modeled on $(\R^d,\|\cdot\|_2)$ such that 
	\begin{equation}AppA \left(M,\mathcal{U},\varepsilon+D(S)\left(\Lip(P)+1\right)\right)\leq \sqrt{d}\cdot\Lip(P),\end{equation}
	where $D(S):=\sup_{x\in M}d(x,S)$.
\end{theorem}

\begin{proof}
	Given $A\subset S$ with $|A| \leq d+1$, we define
	\begin{equation}
\label{eq:Definition_charts_from_random_proj:sets}    U_A:=\text{int}\left(P^{-1}\left(K_A\right)\right)\subset M.
	\end{equation}
	Let $\tau>0$ witness the $d+1$ strong multiplicity of the stochastic $\varepsilon$-almost retraction $P$. Then, for each $x \in M$, there exists a subset $A_x \subseteq S$ with $|A_x|\leq d+1$ such that $B(x,\tau)\subseteq U_{A_x}$. In particular, the family $(U_A)_{|A|=d+1}$ forms an open cover of $M$.

	Next, since $\mathcal{F}_A$ is $d$-dimensional, there exists a linear isomorphism $\Phi_A\colon \mathcal{F}_A\rightarrow (\R^d,\|\cdot\|_2)$ such that 
	\begin{equation}\|\Phi_A\|\cdot\|\Phi_A^{-1}\|\leq\sqrt{d}.\end{equation} 
	Define the coordinate map
	\begin{align}
		\label{eq:Definition_charts_from_random_proj:maps}
		\varphi_A\colon  U_A&\longrightarrow (\R^d,\|\cdot\|_2), \\
		x&\longmapsto \Phi_A\left(Px-\frac{1}{|A|}\sum_{z\in A}\delta(z)\right).\nonumber
	\end{align}
	Then $\Lip(\varphi_A)\leq \|\Phi_A\|\cdot\Lip(P)$. Hence, the collection $\mathcal{U}:=\left(U_A,\varphi_A\right)_{|A|=d+1}$ is an open atlas of $M$ modeled on $(\R^d,\|\cdot\|_2)$. 

	Let $f\colon M\rightarrow Y$ be a $1$-Lipschitz map. We can fix a distinguished point $0\in S\subset M$ and assume without loss of generality that $f(0)=0$, and consider 
	\begin{equation}F=\overline{f\restricted_{S}}\circ P\colon M\rightarrow Y,\end{equation}
    where $\overline{f\restricted_{S}}\colon \mathcal{F}(S)\rightarrow Y$ is the Lipschitz-free linear extension of $f\restricted_{S}$ described in Section \ref{subsection:Free}. Since $\overline{f\restricted_{S}}$ is linear with $\|f\restricted_{S}\|\leq 1$, we have $\Lip(F)\leq \Lip(P)$. We now show that $F$ is affine with respect to the atlas $\mathcal{U}$. Fix $A\subset S$ with $|A|=d+1$. For every $x,y\in U_A$, we have
	\begin{align}
		F(y)-F(x)&=\left(\overline{f\restricted_{S}}\circ P\right)(y)-\left(\overline{f\restricted_{S}}\circ P\right)(x)
		=\overline{f\restricted_{S}}\big(P(y)-P(x)\big)\\
		&=\overline{f\restricted_{S}}\!\left[\left(Py-\frac{1}{|A|}\sum_{z\in A}\delta(z)\right)
		-\left(Px-\frac{1}{|A|}\sum_{z\in A}\delta(z)\right)\right]\nonumber\\
		&=\big(\overline{f\restricted_{S}}\circ \Phi_A^{-1}\big)\!\left(\Phi_A\!\left(Py-\frac{1}{|A|}\sum_{z\in A}\delta(z)\right)
		-\Phi_A\!\left(Px-\frac{1}{|A|}\sum_{z\in A}\delta(z)\right)\right)\nonumber\\
		&=\big(\overline{f\restricted_{S}}\circ \Phi_A^{-1}\big)\big(\varphi_A(y)-\varphi_A(x)\big).\nonumber
	\end{align}
	Thus $F$ is affine with respect to $\mathcal{U}$, with $(AF)_{A}=\overline{f\restricted_{S}}\circ \Phi_A^{-1}$ for every $A\subset S$ satisfying $|A|=d+1$. A straightforward computation then gives
	\begin{equation}|DF|_\mathcal{U}=\|F\|_{\mathcal{U}}\leq  \sqrt{d}\cdot\Lip(P).\end{equation}
	
	Finally, fix $y\in M$ and $\rho>0$, and choose $x\in S$ such that $d(x,y)< D(S)+\rho$. Since $\|P(x)-\delta(x)\|<\varepsilon$ for all $x\in S$, the triangle inequality yields
	\begin{align}
		\|F(y)-f(y)\|
		&\leq \|F(y)-F(x)\|+\|F(x)-f(x)\|+\|f(x)-f(y)\|\\
		&\leq \Lip(F)(D(S)+\rho)+\varepsilon\Lip(f)+\Lip(f)(D(S)+\rho)\nonumber\\
		&\leq \big(\Lip(P)+1\big)\big(D(S)+\rho\big)+\varepsilon.\nonumber
	\end{align}
	Since $\rho>0$ is arbitrary, we obtain
	\begin{equation}AppA\left(M,\mathcal{U},\varepsilon+D(S)\left(\Lip(P)+1\right)\right)\leq \sqrt{d}\cdot\Lip(P),\end{equation}
	as desired. \qedhere
\end{proof}

\begin{remark}
	In Theorem~\ref{Theorem:Atlas_from_random_almost_projections}, the factor $\sqrt{d}$ arises from our choice of the Euclidean space $\ell_2^d=(\R^d,\|\cdot\|_2)$ as the model space for the atlas. This constant corresponds to the classical upper bound on the Banach--Mazur distance between a $d$-dimensional normed space and the Euclidean space $\ell_2^d$, which follows from John's ellipsoid theorem. 
	The appearance of $\sqrt{d}$ is therefore somewhat conventional rather than intrinsic to the construction.
\end{remark}

\section{Random partitions and affine approximation in spaces of finite Nagata dimension}
\label{section:AppA-FND}

The goal of this section is to prove the next result.

\begin{theorem}
	\label{Theorem:Nagata_iff_approx_C1_smooth}
	Let $M$ be a separable metric space and $d\in\N$. If the Nagata dimension of $M$ is at most $d$, then $AppA_d(M)$ is finite. More precisely:
	\begin{equation}AppC^1_d(M)\leq AppA_d(M)\lesssim O(\gamma_d(M)d^{7/2}).\end{equation}
\end{theorem}

% ----------------------------------------------------------------

% Our general strategy to obtain an approximation of a Lipschitz map $f\colon M\rightarrow Y$ by an affine map $F$ is to define $F$ as a suitable almost-extension to $M$ of the restriction of $f$ to a sufficiently dense subset $S\subset M$. Naor and Silberman~\cite[Corollary~5.2]{NaorSilberman} showed that one may always extend a $1$-Lipschitz map $f\colon S\rightarrow Y$ to a map $F\colon M\rightarrow Y$ satisfying $\|F\|_\Lip\leq O(\gamma_d(S)d^3)$ for any given $d\in\N$. Roughly speaking, their argument proceeds by showing that if $\gamma_d(S)<\infty$, then $S$ admits sufficiently good bounded and padded random partitions~\cite[Lemma~5.1]{NaorSilberman}, allowing one to apply the extension results of Lee and Naor from~\cite{LeeNaorInv}. 

% In what follows, we also use the random partitions obtained in~\cite[Lemma~5.1]{NaorSilberman}. However, although the main idea of constructing extensions from random partitions originates in~\cite{LeeNaorInv}, we do not apply the results there directly. This is because the extension techniques developed in~\cite{LeeNaorInv} apply to a remarkably general setting, while our goal requires more explicit control of the structure of the resulting maps. In particular, to construct a suitable atlas and to establish the $\mathcal C^1$-smoothness of the extensions, we need a more elementary and constructive description of such extensions, which we can easily obtain by extending from sufficiently dense nets.

% ---------------------------------------------------------------

Naor and Silberman~\cite[Corollary~5.2]{NaorSilberman} showed that one may always extend a $1$-Lipschitz map $f\colon S\rightarrow Y$ to a map $F\colon M\rightarrow Y$ satisfying $\Lip(F)\leq O(\gamma_d(S)d^3)$ for any given $d\in\N$. Roughly speaking, their argument proceeds by showing that if $\gamma_d(S)<\infty$, then $S$ admits sufficiently good bounded and padded random partitions~\cite[Lemma~5.1]{NaorSilberman}, allowing one to apply the extension results of Lee and Naor from~\cite{LeeNaorInv}. Applying this approach directly, we could obtain a stochastic retraction $P\colon M\rightarrow \mathcal{F}(S)$ for any subset $S$ of $M$; concretely, this follows by Lemma 5.1 in \cite{NaorSilberman}, Theorem 4.1.2 in \cite{LeeNaorInv} and Lemma 4.3 in \cite{Ohta09}. However, in Theorem \ref{Theorem:Atlas_from_random_almost_projections}, we also require strong multiplicity $d$ for this stochastic retraction, an essential condition for our purposes which is not guaranteed by the previous process. On the flip side, we do not require the map $P\colon M\rightarrow \mathcal{F}(S)$ to be an ``exact" stochastic retraction, and we can ask extra properties on $S$ (namely, that it is sufficiently dense in $M$) in order to simplify the construction. 

In this section, we go back to considering random partitions in a metric space of Nagata dimension $d$, with the goal of constructing almost stochastic retractions of strong multiplicity $d$, so as to apply Theorem \ref{Theorem:Atlas_from_random_almost_projections} to prove Theorem \ref{Theorem:Nagata_iff_approx_C1_smooth}. The essence of the process we follow comes from \cite{LeeNaorInv} and \cite{NaorSilberman}, but along the way we must introduce several adaptations and workarounds to get the desired result. 

%------------------------------------------------------------------

Before describing this construction, we recall several notions related to random partitions, following the exposition in~\cite[Section~1.7 and Chapter~3]{NaorBook}. We also introduce a mild refinement of random partitions subordinated to a given cover.

We say that a cover $\mathcal{C}$ of a metric space $M$ has \emph{multiplicity} $d\in\N$ if 
\begin{equation}\big|\{C\in\mathcal{C}\colon x\in C\}\big|\leq d \quad \text{for every } x\in M.\end{equation} 
The cover $\mathcal{C}$ has \emph{strong multiplicity} $d\in\N$ if there exists $\tau>0$ such that 
\begin{equation}\big|\{C\in\mathcal{C}\colon d(x,C)\leq \tau\}\big|\leq d \quad \text{for every } x\in M.\end{equation}

\begin{defn}
	Let $M$ be a metric space and $(\Omega,\mathbb{P})$ a probability space. A set-valued mapping $\Gamma\colon \Omega\rightarrow 2^M$ is said to be \emph{strongly measurable} if, for every closed set $E\subset M$, the set 
	\begin{equation}\Gamma^{-}(E):=\{\omega\in\Omega\colon E\cap\Gamma(\omega)\neq\emptyset\}\end{equation} 
	is $\mathbb{P}$-measurable.
	
	A \emph{random partition} of $M$ is a family of countable partitions $(\mathcal{P}^\omega)_{\omega\in\Omega}$ of $M$, indexed by $\omega\in\Omega$, such that there exists a sequence of strongly measurable set-valued mappings $(\Gamma_k\colon \Omega\rightarrow 2^M)_{k\in\N}$ satisfying $\mathcal{P}^\omega=(\Gamma_k(\omega))_{k\in\N}$ for all $\omega\in\Omega$. The sets forming each partition $\mathcal{P}^\omega$ are called the \emph{clusters} of $\mathcal{P}^\omega$. For $x\in M$, we denote by $\mathcal{P}^\omega(x)$ the unique cluster of $\mathcal{P}^\omega$ that contains $x$.
	
	%We say that a random partition is \emph{discrete} if it is defined on a discrete probability space.
	
	Given $\Delta>0$, a random partition $\mathcal{P}$ of $M$ is said to be \emph{$\Delta$-bounded} if every cluster of $\mathcal{P}^\omega$ has diameter less than $\Delta$ for every $\omega\in \Omega$.
	
	For $\Delta>0$, a $\Delta$-bounded random partition $\mathcal{P}$ is called:
	\begin{itemize}
		\item \emph{$(p,\delta)$-padded} for $p,\delta>0$ if for all $x\in M$,
		\begin{equation}\mathbb{P}\big[\omega\in\Omega\colon B\!\left(x,\frac{\Delta}{p}\right)\subset \mathcal{P}^\omega(x)\big]\geq \delta;\end{equation}
		\item \emph{$\sigma$-separating} for $\sigma>0$ if for all $x\neq y\in M$,
		\begin{equation}\mathbb{P}\big[\omega\in\Omega\colon \mathcal{P}^\omega(x)\neq\mathcal{P}^\omega(y)\big]\leq \frac{\sigma}{\Delta}\,d(x,y);\end{equation}
		\item \emph{subordinated to a cover} $\mathcal{C}$ of $M$ if every cluster of $\mathcal{P}^\omega$ is contained in some $C\in\mathcal{C}$.
	\end{itemize}
\end{defn}

\begin{remark}
	\label{Remark:Measurability_random_partitions}
	The strong measurability of the set-valued mappings $(\Gamma_k)_{k\in\N}$ in the definition of random partitions ensures that the notions of $(p,\delta)$-padded and $\sigma$-separating random partitions are well defined (see~\cite[Chapter~3]{NaorBook}). In general, when $M$ is separable, for every $x\in M$ and every Borel set $E\subset M$, the events ``$E\subset\mathcal{P}^\omega(x)$'' and ``$\mathcal{P}^\omega(x)\subset E$'' are $\mathbb{P}$-measurable. In what follows, the mappings $(\Gamma_k)_{k\in\N}$ will only be used to ensure these measurability properties. 
	
	It is also useful to note that if $M$ is a separable metric space and $\Gamma\colon \Omega\rightarrow 2^M$ is a set-valued mapping, then to verify that $\Gamma$ is strongly measurable it suffices to check that $\Gamma^{-}(E)$ is measurable for all closed $E\subset M$ with sufficiently small diameter (for instance, less than a fixed $\varepsilon>0$).
\end{remark}

\subsection{From Finite Nagata Dimension to Bounded and Padded Random Partition}

It is immediate from the construction in~\cite[Lemma~5.1]{NaorSilberman} (where random partitions are termed \emph{stochastic decompositions}) that the resulting random partition of a metric space with Nagata dimension $d$ is subordinated to a cover of multiplicity $d+1$. However, since our aim is to produce an \textit{open} atlas of the metric space, we require the random partition to be subordinated to a cover with \emph{strong} multiplicity $d+1$. In fact, the same construction yields strong multiplicity $d+1$ with no modifications; this is the form that will be used below. For future reference, we record the precise statement and include a short proof for completeness.

\begin{lemma}[Naor--Silberman~\cite{NaorSilberman}]
	\label{Lemma:Naor_Silberman}
	Let $M$ be a metric space, let $\gamma\geq 1$, and let $d\in\N$. Suppose that $\gamma_d(M)<\gamma<\infty$. Then, for arbitrarily small $k\in \Z$, there exists a $2^k$-bounded and $\left(100\gamma d^2,\frac{1}{d+1}\right)$-padded random partition of $M$ which is subordinated to a $2^k$-bounded cover of $M$ with strong multiplicity $d+1$.
\end{lemma}

\begin{proof}
	As in~\cite[Lemma~5.1]{NaorSilberman}, we use~\cite[Proposition~4.1]{LangSchli05} to obtain the following. Put $r:=50\gamma\, d^2$ (see the discussion after Proposition~\ref{prop:LangSchli}). Then, for every $j\in \Z$ there exists a cover $\mathcal{C}$ of $M$ such that:
	\begin{enumerate}
		\item $\diam(C)\leq r^{j+1}$ for all $C\in\mathcal{C}$.
		\item For every $x\in M$ there exists $C\in\mathcal{C}$ such that $B(x,r^j)\subset C$.
		\item There exist subfamilies $\{\mathcal{C}_i\}_{i=0}^d$ with $\mathcal{C}=\bigcup_{i=0}^d\mathcal{C}_i$, and for every $i\in\{0,\dots,d\}$, any two distinct sets in $\mathcal{C}_i$ are at distance at least $r^j$.
	\end{enumerate}
	Relative to~\cite[Lemma~5.1]{NaorSilberman}, the additional feature here is the quantitative separation in (3). This implies strong multiplicity $d+1$ as follows. Indeed, fix $x\in M$. If $A,B\in\mathcal{C}_i$ satisfy $d(x,A)\leq r^j/4$ and $d(x,B)\leq r^j/4$, then
	\begin{equation}d(A,B)\leq d(A,x)+d(x,B)\leq \frac{r^j}{2}<r^j,\end{equation}
	which contradicts the separation in (3) unless $A=B$. Hence, for each $i\in\{0,\dots,d\}$, there is \emph{at most one} set $A_i\in\mathcal{C}_i$ with $d(x,A_i)\leq r^j/4$, and therefore
	\begin{equation}\big|\{ C \in \mathcal{C} : d(x,C)\leq r^j/4 \}\big| \leq d+1.\end{equation}
	Thus $\mathcal{C}$ has strong multiplicity $d+1$.
	
	From this point the construction proceeds exactly as in~\cite[Lemma~5.1]{NaorSilberman}. Let $\Omega$ be the set of all permutations of $\{0,\dots,d\}$, endowed with the uniform probability $\mathbb{P}$. For each $\omega\in\Omega$, define a partition $\mathcal{P}^\omega:=\bigcup_{i=0}^d\mathcal{P}^\omega_i$ of $M$ inductively by setting $\mathcal{P}^\omega_0:=\mathcal{C}_{\omega(0)}$, and for $i=1,\dots, d$,
	\begin{equation}\mathcal{P}^\omega_i:=\left\{\,C\setminus\bigcup_{A\in\bigcup_{l=0}^{i-1}\mathcal{P}^\omega_l}A\colon C\in\mathcal{C}_{\omega(i)}\,\right\}.\end{equation}
	Since $\mathcal{C}$ covers $M$, $(\mathcal{P}^\omega)_{\omega\in\Omega}$ is a random partition; by construction it is $r^{j+1}$-bounded and subordinated to $\mathcal{C}$. There are no measurability issues here because $\Omega$ is finite.
	
	We now check that $\mathcal{P}$ is $\left(100\gamma d^2,\frac{1}{d+1}\right)$-padded. Fix $x\in M$. By (2) there exist $i_x\in\{0,\dots,d\}$ and $C\in \mathcal{C}_{i_x}$ with $B(x,r^j)\subset C$. The set $C$ survives in the partition $\mathcal{P}^\omega$ precisely when $\omega(0)=i_x$, which occurs with probability $1/(d+1)$. Hence
	\begin{equation}\mathbb{P}\big[B(x,r^j)\subset \mathcal{P}^\omega(x)\big]\ge \frac{1}{d+1}.\end{equation}
	
	Finally, choose $k\in\Z$ so that $2^{k-1}\le r^{j+1}\le 2^k$. Then, for every $x\in M$,
	\begin{equation}\mathbb{P}\left[B\!\left(x,\frac{2^k}{2r}\right)\subset\mathcal{P}^\omega(x)\right] \geq \mathbb{P}\big[B(x,r^j)\subset\mathcal{P}^\omega(x)\big] \geq \frac{1}{d+1}.\end{equation}
	Since $j\in\Z$ is arbitrary, $r^{j+1}$ can be made arbitrarily small, which gives the desired $2^k$-bounded, $(100\gamma d^2,1/(d+1))$-padded random partition subordinated to a $2^k$-bounded cover of strong multiplicity $d+1$.
\end{proof}

\subsection{From Bounded and Padded to Bounded and Separating Random Partition}
We shall use the following (unpublished) result of Lee and Naor, for which we include a proof for completeness. It corresponds to~\cite[Theorem~2.2]{LeeNao03}. In their manuscript, the statement is formulated in a very general setting (for arbitrary metric spaces and $\Delta$-bounded, $(\frac{1}{\delta},\frac{1}{2})$-padded random partitions) without measurability hypotheses, since random partitions are treated there in a purely combinatorial manner. When using random partitions to define (almost)-extension operators, measurability is essential, and so we assume throughout that the metric space is separable and that the underlying probability space for the random partition is Polish (see the definition below). These additions do not change the substance of the argument, but they are needed here because we must verify that the \emph{newly defined} random partition is \emph{strongly measurable}. In particular, working over a Polish base space lets us ensure that all the events involved are measurable and that standard measure-theoretic tools apply; this measurability verification is the only extra step beyond the original proof in \cite{LeeNao03}.

\begin{defn}
	A \emph{Polish probability space} is a probability space $(\Omega, \mathcal{F}, \mathbb{P})$ such that $\Omega$ is a Polish space (i.e., a separable completely metrizable topological space), $\mathcal{F}$ is the Borel $\sigma$-algebra of $\Omega$, and $\mathbb{P}$ is a probability measure on $(\Omega, \mathcal{F})$. 
	
	We say that a random partition $\mathcal P = (\mathcal{P}^\omega)_{\omega\in\Omega}$ is \textit{Polish} if its underlying probability space $(\Omega,\mathbb{P})$ is Polish.
\end{defn}

\begin{theorem}[Lee, Naor~\cite{LeeNao03}] 
	\label{Thm:From_padded_to_separating}
	Let $M$ be a complete separable metric space, and let $\Delta,p,\delta>0$. If $M$ admits a Polish, $\Delta$-bounded and $(p,\delta)$-padded random partition, then it also admits a Polish, $\Delta$-bounded and $\frac{3p}{\delta}$-separating random partition.
	
	If the original partition was subordinated to a cover $\mathcal{C}$, then the new partition can also be taken subordinated to $\mathcal{C}$.
\end{theorem}

\begin{proof}
	Let $(\Omega,\mathbb{P})$ be a Polish probability space for which there exists a $\Delta$-bounded and $(p,\delta)$-padded random partition $\mathcal{P}=\left(\mathcal{P}^\omega\right)_{\omega\in\Omega}$ of $M$, and let $\left(\Gamma_k\colon \Omega\rightarrow 2^M\right)_{k\in\N}$ be strongly measurable set-valued maps with $\mathcal{P}^\omega=\left(\Gamma_k(\omega)\right)_{k\in\N}$ for all $\omega\in\Omega$. 
	\medskip
	
	\textbf{Step 1. Definition of the random partition $\mathcal{Q}$.}
	Fix $0<\rho<\frac{1}{2}$ (to be specified later), and set $\Omega_S:=\left(\Omega\times[\rho,1-\rho]\right)^\N$ with the product probability measure $\mathbb{P}_S$, where each $[\rho,1-\rho]$ carries the uniform distribution. Clearly $\Omega_S$ is Polish. For $n\in\N$, let $P_n\colon\Omega_S\rightarrow \left(\Omega\times[\rho,1-\rho]\right)^n$ denote the projection onto the first $n$ coordinates; then $P_n^{-1}(A)$ is measurable whenever $A\subset\left(\Omega\times[\rho,1-\rho]\right)^n$ is measurable.
	
	For $\tau=((\omega_1,t_1),(\omega_2,t_2),\dots)\in\Omega_S$ and $x\in M$, define
	\begin{equation}
		\label{eq:Def_n(tau,x)}
		n(\tau,x):=\min\left\{n\in\N\colon \exists\varepsilon>0\ \text{with}\ B\left(x,(t_n+\varepsilon)\frac{\Delta}{p}\right)\subset\mathcal{P}^{\omega_n}(x)\right\},
	\end{equation}
	with the convention $n(\tau,x):=\infty$ if the set is empty.
	
	For $(k,l)\in\N^2$ and $\tau\in\Omega_S$, set 
	\begin{equation}
		\label{eq:Def_Lambda_(k,l)}
		\Lambda_{(k,l)}(\tau):=\left\{x\in M\colon n(\tau,x)=k\ \text{and}\ \mathcal{P}^{\omega_k}(x)=\Gamma_l(\omega_k)\right\}\subset\Gamma_l(\omega_k).
	\end{equation}
	
	Define $\mathcal{Q}:=\{\mathcal{Q}^\tau\}_{\tau\in\Omega_S}$ by $\mathcal{Q}^\tau:=\{\Lambda_{(k,l)}(\tau)\}_{(k,l)\in\N^2}$ for each $\tau\in\Omega_S$. We will show in Step~3 that $\mathcal{Q}^\tau$ is a partition of $M$ for $\mathbb{P}_S$-almost every $\tau$. Heuristically, the construction proceeds by repeatedly sampling a partition $\mathcal P^{\omega}$ and a “thickness” parameter $t\in[\rho,1-\rho]$ until the ball $B\!\left(x,t\frac{\Delta}{p}\right)$ fits strictly inside the cluster of $x$; the point $x$ is then assigned to that cluster at the first successful step. In particular, the probability that two points $x,y\in M$ end up in the same cluster of $\mathcal Q$ is the probability that, while sampling partitions and parameters in parallel, the first time that $B\!\left(x,t\frac{\Delta}{p}\right)$ and $B\!\left(y,t\frac{\Delta}{p}\right)$ are both contained in the interiors of their respective clusters occurs simultaneously, and that these clusters coincide.
	
	\medskip
	
	\textbf{Step 2. Strong measurability of $\Lambda_{(k,l)}$.}
	We claim that each $\Lambda_{(k,l)}\colon \Omega_S\rightarrow 2^M$ is strongly measurable. The argument adapts~\cite[Lemma~119]{NaorBook}. Fix $(k,l)\in\N^2$ and a closed set $E\subset M$ with $\diam(E)\leq \rho\,\frac{\Delta}{p}$. Since $\Lambda_{(k,l)}(\tau)$ is determined by the first $k$ coordinates of $\tau$, we may write $\Lambda_{(k,l)}\!\left((\omega_n,t_n)_{n=1}^k\right)$ for this dependence. Then
	\begin{equation}
		\Lambda_{(k,l)}^{-}(E)=P_k^{-1}\!\left(\left\{(\omega_n,t_n)_{n=1}^k\colon \Lambda_{(k,l)}\!\left((\omega_n,t_n)_{n=1}^k\right)\cap E\neq\emptyset\right\}\right).
	\end{equation}
	Thus it suffices to show measurability of 
	\begin{equation}
		\label{eq:Event_for_fixed_omega}
		H:=\left\{(\omega_n,t_n)_{n=1}^k\colon \Lambda_{(k,l)}\!\left((\omega_n,t_n)_{n=1}^k\right)\cap E\neq\emptyset\right\}.
	\end{equation}
	Consider 
	\begin{align}
		B:=\Big\{ ((\omega_n, t_n)_{n=1}^k,x) &\in (\Omega \times [\rho,1-\rho])^{k}\times E \colon \exists \ep > 0 ,\ B(x,(t_k+\ep)\tfrac{\Delta}{p}) \subset \Gamma_l(\omega_k),\\
		&\text{ and } \forall j \le k-1,\ \forall \ep > 0,\ B(x,(t_j+\ep)\tfrac{\Delta}{p}) \nsubseteq \mathcal{P}^{\omega_j}(x)\Big\},\nonumber
	\end{align}
	which is Borel measurable. By construction, $H$ is the image of $B$ by the projection onto $(\Omega \times [\rho,1-\rho])^{k}$. Hence $H$ is analytic; in particular it is universally measurable by Lusin's theorem, and therefore measurable. Consequently, $\Lambda_{(k,l)}^{-}(E)$ is measurable.
	
	For completeness, we indicate why $B$ is Borel. First, the set
	\begin{equation}B_1:= \{(\omega,t,x)  : \exists \ep > 0 ,\ B(x,(t+\ep)\tfrac{\Delta}{p}) \subset \Gamma_l(\omega) \} \end{equation}
	is Borel because it equals $\{F>0\}$ for the measurable map
	\begin{equation}F(\omega,t,x):= d\!\left(x,\overline{M\setminus \Gamma_l(\omega)}\right)- t \tfrac{\Delta}{p},\end{equation}
	where measurability follows from the strong measurability of $\Gamma_l$ and a standard approximation of open-ball inclusions by rational radii. Next, the sets
	\begin{equation}
		B_2 := \{(\omega,t,x)  : \forall \ep > 0,\ B(x,(t+\ep)\tfrac{\Delta}{p}) \nsubseteq \mathcal P^{\omega}(x) \}
	\end{equation}
	are Borel because their complements are countable unions over $j\in\N$ of sets of the form in $B_1$, using $\mathcal P^\omega=\{\Gamma_j(\omega)\}_{j\in\N}$. Thus $B$ is Borel, as required.
	\medskip
	
	\textbf{Step 3. $\mathcal{Q}^\tau$ is a partition of $M$ $\mathbb{P}_S$-a.s.} We show that $n(\tau,x)<\infty$ for all $x$, for $\mathbb{P}_S$-a.e.\ $\tau$. For $E\subset M$, set
	\begin{equation}
		\label{eq:Def_Problematic_E}
		\Upsilon(E):=\left\{\tau\in\Omega_S\colon \exists x\in E \text{ with } n(\tau,x)=\infty\right\},
	\end{equation}
	and 
	\begin{equation}
		\label{eq:Def_Pseudo_Problematic_E}
		\Theta(E):=\left\{(\omega_n,t_n)_{n\in\N}\in\Omega_S\colon \exists x\in E,\; \forall n\in \N, \; B\!\left(x,\tfrac{\Delta}{p}\right)\nsubseteq\mathcal{P}^{\omega_n}(x)\ \right\}.
	\end{equation}
	Clearly $\Upsilon(E)\subset \Theta(E)$. Since $\mathcal{P}$ is $(p,\delta)$-padded, for fixed $x$ and $k\in \N$,
	\begin{equation}
		\mathbb{P}_S(\Theta(\{x\}))\leq \mathbb{P}_S\Big(\bigcap_{n=1}^k \{B(x,\tfrac{\Delta}{p})\nsubseteq\mathcal{P}^{\omega_n}(x)\}\Big)\le (1-\delta)^k.
	\end{equation}
	Hence $\Theta(\{x\})$ is null for each $x$. Let $D\subset M$ be countable and dense; then $\Theta(D)$ is null. We claim $\Upsilon(M)\subset \Theta(D)$, which forces $\mathbb{P}_S(\Upsilon(M))=0$.
	
	Indeed, if $\tau\in \Upsilon(M)\setminus\Theta(D)$, there exists $x\in M$ with $B\!\left(x,(t_n+\varepsilon)\tfrac{\Delta}{p}\right)\nsubseteq\mathcal{P}^{\omega_n}(x)$ for all $n$ and all $\varepsilon>0$. Choose $y\in D$ with $d(x,y)<\frac{\rho}{2}\tfrac{\Delta}{p}$. For $0<\varepsilon<\frac{\rho}{2}$ one has
	\begin{equation}B \!\left ( x, (t_n+\varepsilon)\tfrac{\Delta}{p} \right) \subset B\!\left( y, (t_n+\rho)\tfrac{\Delta}{p} \right),\end{equation}
	whence $B\!\left(y,(t_n+\rho)\tfrac{\Delta}{p}\right)\nsubseteq\mathcal{P}^{\omega_n}(y)$ for all $n$. Since $\tau\notin\Theta(D)$, there exists $n_0$ with $B\!\left(y,\tfrac{\Delta}{p}\right)\subset\mathcal{P}^{\omega_{n_0}}(y)$, forcing $t_{n_0}>1-\rho$, a contradiction to $t_{n_0}\in[\rho,1-\rho]$.
	\medskip
	
	\textbf{Step 4. $\mathcal{Q}$ is $\Delta$-bounded, $\bm{p\min\{\rho,(1-2\rho)\delta\}^{-1}}$-separating, and (if applicable) subordinated to $\mathcal{C}$.} By~\eqref{eq:Def_Lambda_(k,l)}, each cluster of $\mathcal{Q}$ lies in a cluster of $\mathcal{P}$, so $\mathcal{Q}$ is $\Delta$-bounded and, if $\mathcal{P}$ is subordinated to $\mathcal{C}$, then so is $\mathcal{Q}$. 
	
	Fix $x\neq y\in M$ and denote by $q$ the probability that $\mathcal{Q}^\tau(x)\neq\mathcal{Q}^\tau(y)$. We may assume $d(x,y)\le \frac{\rho\Delta}{p}$ (otherwise the desired estimate is trivial). Let $T$ be the backward shift on $\Omega_S$. Independence of coordinates yields
	\begin{equation}
		\label{eq:q=q_k}
		q= \mathbb{P}_S\left[\mathcal{Q}^{T(\tau)}(x)\neq\mathcal{Q}^{T(\tau)}(y)\right].
	\end{equation}
	Consider the first coordinate $(\omega_1,t_1)$. Two cases can occur:
	\begin{itemize}
		\item[(a)] $\min\{n(\tau,x),n(\tau,y)\}>1$. Then the first coordinate is ignored and 
		\begin{equation}\mathbb{P}_S\!\left(\big[\mathcal{Q}^\tau(x)\neq\mathcal{Q}^\tau(y)\big]\mid\big[\text{(a)}\big]\right)=\mathbb{P}_S\!\left(\big[\mathcal{Q}^{T(\tau)}(x)\neq\mathcal{Q}^{T(\tau)}(y)\big]\mid\big[\text{(a)}\big]\right)=q,\end{equation}
		by independence. Moreover,
		\begin{equation}
			\label{eq:Prob_a}
			\mathbb{P}_S\left[(a)\right]\le 1-\delta,
		\end{equation}
		because (a) forces $B\!\left(x,\tfrac{\Delta}{p}\right)\nsubseteq\mathcal{P}^{\omega_1}(x)$.
		\item[(b)] $\min\{n(\tau,x),n(\tau,y)\}=1$. In this case, either $x$ or $y$ already has $B(\cdot,\rho\frac{\Delta}{p})$ contained in its cluster in $\mathcal{P}^{\omega_1}$; since $d(x,y)\le \rho\Delta/p$, both points lie in the \emph{same} cluster $\mathcal{P}^{\omega_1}(x)=\mathcal{P}^{\omega_1}(y)$. Defining
		\begin{align}
			t_x&:=\sup\Big\{t\in[0,1-\rho]\colon B\!\left(x,t\tfrac{\Delta}{p}\right)\subset\mathcal{P}^{\omega_1} (x)\Big\},\\
			t_y&:=\sup\Big\{t\in[0,1-\rho]\colon B\!\left(y,t\tfrac{\Delta}{p}\right)\subset\mathcal{P}^{\omega_1}(y)\Big\},
		\end{align}
		one checks that $|t_x-t_y|\le \tfrac{p}{\Delta}d(x,y)$. Therefore, for both (b) and the event $\mathcal{Q}^\tau(x)\neq\mathcal{Q}^\tau(y)$ to occur simultaneously, it must happen that $t_1$ falls in an interval of $[\rho,1-\rho]$ of length at most $\tfrac{p}{\Delta}d(x,y)$, hence
		\begin{equation}\label{eq:Prob_if_b}
			\mathbb{P}_S\Big(\big[\mathcal{Q}^\tau(x)\neq\mathcal{Q}^\tau(y)\big]\cap\big[\text{(b)}\big]\Big) \leq \frac{p}{\Delta(1-2\rho)}\,d(x,y).
		\end{equation}
	\end{itemize}
	Combining~\eqref{eq:q=q_k}, \eqref{eq:Prob_a}, and~\eqref{eq:Prob_if_b} yields
	\begin{equation}
	q \le q(1-\delta)+ \frac{p}{\Delta(1-2\rho)}\,d(x,y),
	\end{equation}
	so
	\begin{equation}
	q \le \frac{p}{\Delta(1-2\rho)\delta}\,d(x,y)\ \le\ \frac{p}{\Delta\min\{\rho,(1-2\rho)\delta\}}\,d(x,y).
	\end{equation}
	Choosing $\rho=\frac{1}{3}$ (and recalling $\delta\le 1$) gives $q\le \frac{3p}{\Delta\delta}d(x,y)$, i.e., the random partition is $\frac{3p}{\delta}$-separating and $\Delta$-bounded, as required. \qedhere
\end{proof}

\subsection{From Bounded and Separating Random Partition to Stochastic Almost Retraction}

In the next result, we show how to obtain stochastic almost retractions with strong multiplicity $d$ from separating random partitions subordinated to covers with strong multiplicity $d$. The main idea of the proof is inspired by~\cite{LeeNaorInv} (see also~\cite[Chapter~5]{NaorBook}).

Recall that for any subset $A\subset M$, we write 
\begin{equation}K_A:=\text{conv}\{\delta(x)\colon x\in A\}\subset \mathcal{F}(M),\end{equation} 
as introduced in Section~\ref{Section:SARtoAA}.

\begin{theorem}[Stochastic almost retraction from separating random partition]
	\label{Theorem:Almost_retraction_from_sep_random_partition}
	Let $M$ be a separable metric space, let $\Delta,\sigma,\rho,\varepsilon>0$, and let $d\in\N$. If $M$ admits a $\Delta$-bounded $\sigma$-separating random partition subordinated to a $\rho$-bounded cover of $M$ with strong multiplicity $d$, then for every $\varepsilon$-dense subset $N\subset M$ there exists a stochastic $(\rho+\varepsilon)$-almost retraction $P\colon M\rightarrow\mathcal{F}(N)$ with strong multiplicity $d$ such that
	\begin{equation}\Lip(P)\leq\frac{2\sigma(\rho+\varepsilon)}{\Delta}+1.\end{equation}
\end{theorem}

\begin{proof}
	Let $\mathcal{P}=(\mathcal{P}^\omega)_{\omega\in\Omega}$ be a $\Delta$-bounded and $\sigma$-separating random partition of $M$, and let $\mathcal{C}$ be a $\rho$-bounded cover of $M$ with strong multiplicity $d$ such that $\mathcal{P}$ is subordinated to $\mathcal{C}$. We may assume without loss of generality that every set in $\mathcal{C}$ is closed (hence Borel). For technical reasons, fix a well-ordering $\preccurlyeq$ on $\mathcal{C}$. Choose a countable subset $S\subset N$ which remains $\varepsilon$-dense in $M$. 
	
	For each $C\in \mathcal{C}$, select a point $z_C\in S$ with $d(z_C,C)\leq\varepsilon$. For $x\in M$ and $\omega\in\Omega$, define $C(x,\omega)\in\mathcal{C}$ as the minimal element (for $\preccurlyeq$) such that $\mathcal{P}^\omega(x)\subset C(x,\omega)$.  Note that for all $x\in M$ and $\omega\in\Omega$:
	\begin{equation}
		\label{eq:x_close_to_z_C(x,omega)}
		d(x,z_{C(x,\omega)})\leq \rho+\varepsilon.
	\end{equation}
	
	Fix $x\in M$. Since $\mathcal{C}$ has strong multiplicity $d$, there exists $\tau>0$ and $\{C_1(x)\preccurlyeq\dots\preccurlyeq C_d(x)\}\subset \mathcal{C}$ such that 
	\begin{equation}B(x,\tau)\cap C=\emptyset\quad\text{for all }C\in\mathcal{C}\setminus\{C_1(x),\dots,C_d(x)\}.\end{equation}
	For each $k=1,\dots,d$ and each $y\in B(x,\tau)$,
	\begin{align}
		\label{eq:C(y,w)=C_k(x)_are_measurable}
		\{\omega\in\Omega\colon C(y,\omega)=C_k(x)\}=\{\omega\in\Omega\colon \mathcal{P}^\omega(y)\subset C_k(x)\}\setminus\bigcup_{j<k}\{\omega\in\Omega\colon\mathcal{P}^\omega(y)\subset C_j(x)\}.
	\end{align}
	By Remark~\ref{Remark:Measurability_random_partitions}, the event “$\mathcal{P}^\omega(y)\subset C_k(x)$ and $\mathcal{P}^\omega(y)\nsubseteq C_j(x)$ for $j<k$” is $\mathbb{P}$-measurable for each $k$, hence so is $\{\omega\in\Omega\colon C(y,\omega)=C_k(x)\}$. Furthermore, since each $\mathcal{P}^\omega$ is a partition of $M$, we have
	\begin{align}
	\label{eq:Omega_as_partition_of_C(y,omega)=C_k(x)_sets}
		\Omega=\bigsqcup_{k=1}^{d}\{\omega\in\Omega\colon C(y,\omega)=C_k(x)\},
	\end{align}
	where $\bigsqcup$ denotes disjoint union.
	
	Define, for each $x\in M$,
	\begin{align}
		R_x\colon \Omega&\longrightarrow \mathcal{F}(N),\\
		\omega&\longmapsto \delta\big(z_{C(x,\omega)}\big).\nonumber
	\end{align}
	Since $S$ is countable, the range of $R_x$ is countable. Moreover, for every $z\in S$,
	\begin{equation}R_x^{-1}(\delta(z))=\bigcup_{k=1}^d \Upsilon(z,x,k),\end{equation}
	where $\Upsilon(z,x,k):=\emptyset$ if $z\neq z_{C_k(x)}$, and $\Upsilon(z,x,k):=\{\omega\in\Omega\colon C(x,\omega)= C_k(x)\}$ otherwise. By~\eqref{eq:C(y,w)=C_k(x)_are_measurable}, each $\Upsilon(z,x,k)$ is $\mathbb{P}$-measurable, and therefore $R_x^{-1}(\delta(z))$ is measurable for all $z\in S$. Hence $R_x$ is $\mathbb{P}$-Bochner measurable. Define
	\begin{equation}Px:=\int_{\Omega}R_x(\omega)\,d\mathbb{P}(\omega)\in\mathcal{F}(N).\end{equation}
	
	We now verify that $P$ is a stochastic $(\rho+\varepsilon)$-almost retraction with \begin{equation}\Lip(P)\leq \frac{2\sigma(\rho+\varepsilon)}{\Delta}+1,\end{equation} and strong multiplicity $d$.
	\medskip
	
	\textbf{Step 1. $P$ has strong multiplicity $d$.}
	From~\eqref{eq:Omega_as_partition_of_C(y,omega)=C_k(x)_sets}, for $y\in B(x,\tau)$ we have
	\begin{align}
		Py&=\sum_{k=1}^d \int_{\{\omega\colon C(y,\omega)=C_k(x)\}}\delta\big(z_{C(y,\omega)}\big)\,d\mathbb{P}(\omega)\\
		&=\sum_{k=1}^d \mathbb{P}\!\left[\omega\in\Omega\colon C(y,\omega)=C_k(x)\right]\delta(z_{C_k(x)}).\nonumber
	\end{align}
	Hence $Py$ is a convex combination of $\delta(z_{C_1(x)}),\dots,\delta(z_{C_d(x)})$. Writing 
	\begin{equation}A_x:=\{z_{C_1(x)},\dots,z_{C_d(x)}\},\end{equation} we have $P(B(x,\tau))\subset K_{A_x}$. In particular, $P$ has strong multiplicity $d$.
	\medskip
	
	\textbf{Step 2. Lipschitz estimate.}
	For $x\neq y\in M$,
	\begin{align}
		\|Px-Py\|
		&=\left\|\int_{\Omega}\big(\delta(z_{C(x,\omega)})-\delta(z_{C(y,\omega)})\big)d\mathbb{P}(\omega)\right\|\\
		&=\left\|\int_{\{\omega\colon \mathcal{P}^\omega(x)\neq\mathcal{P}^\omega(y)\}}\big(\delta(z_{C(x,\omega)})-\delta(z_{C(y,\omega)})\big)d\mathbb{P}(\omega)\right\|\nonumber\\
		&\leq \int_{\left\{\omega\in\Omega\colon \mathcal{P}^\omega(x)\neq\mathcal{P}^\omega(y)\right\}}\|\delta(z_{C(x,\omega)})-\delta(x)\|d\mathbb{P}(\omega)\nonumber\\
		&+\int_{\left\{\omega\in\Omega\colon \mathcal{P}^\omega(x)\neq\mathcal{P}^\omega(y)\right\}}\|\delta(x)-\delta(y)\|d\mathbb{P}(\omega) \nonumber\\
		&+\int_{\left\{\omega\in\Omega\colon \mathcal{P}^\omega(x)\neq\mathcal{P}^\omega(y)\right\}}\|\delta(y)-\delta(z_{C(y,\omega)})\|d\mathbb{P}(\omega).\nonumber
	\end{align}
	By~\eqref{eq:x_close_to_z_C(x,omega)} and the $\sigma$-separating property of $\mathcal{P}$, we obtain
	\begin{equation}\|Px-Py\|\le \left(\frac{2\sigma(\rho+\varepsilon)}{\Delta}+1\right)d(x,y).\end{equation}
	\smallskip
	
	\textbf{Step 3. Almost retraction property.}
	If $x\in N$, then by~\eqref{eq:x_close_to_z_C(x,omega)},
	\begin{equation}\|Px-\delta(x)\|\leq \int_{\Omega}\|\delta(z_{C(x,\omega)})-\delta(x)\|\,d\mathbb{P}(\omega)\leq \rho+\varepsilon.\end{equation}
	This completes the proof.
\end{proof}

\begin{remark}
	In the estimate above, one could refine the bound by using the separation property of $\mathcal P$:
	\begin{equation}
	\int_{\{\omega:\,\mathcal{P}^\omega(x)\neq\mathcal{P}^\omega(y)\}}\|\delta(x)-\delta(y)\|\,d\mathbb{P}(\omega)
	\leq \frac{\sigma}{\Delta}\,d(x,y)^2.
	\end{equation}
	This yields the alternative inequality
	\begin{equation}
	\|Px-Py\|
	\leq \left(\frac{2\sigma(\rho+\varepsilon)}{\Delta}+\frac{\sigma}{\Delta}d(x,y)\right)d(x,y).
	\end{equation}
	Such a refinement provides a slightly sharper control at small scales (when $d(x,y)\ll \tfrac{\Delta}{\sigma}$), although at large scales it leads to a less favorable constant, since $\Delta$ typically appears as a small parameter in the denominator. Note also that in applications (for instance, in the proof of Theorem~\ref{Theorem:Nagata_iff_approx_C1_smooth}), we take $\rho=\varepsilon=\Delta$, so the main term $\tfrac{2\sigma(\rho+\varepsilon)}{\Delta}$ remains of order~$1$.
\end{remark}

\begin{remark}
	The well-ordering $\preccurlyeq$ is used as a fixed tie-breaking rule to make the choice
    $C(x,\omega)\in\mathcal C$ with $\mathcal P^\omega(x)\subset C(x,\omega)$ single-valued.
    Moreover, because the cover $\mathcal C$ has strong multiplicity $d$, for each fixed
    $x\in M$ there exist $\tau>0$ and cover elements
    \begin{equation}
    C_1(x)\preccurlyeq \cdots \preccurlyeq C_d(x)
    \end{equation}
    such that every $C\in\mathcal C$ with $C\cap B(x,\tau)\neq\varnothing$ belongs to
    $\{C_1(x),\dots,C_d(x)\}$. Hence, for every $y\in B(x,\tau)$ and $\omega\in\Omega$ one has
    $C(y,\omega)\in\{C_1(x),\dots,C_d(x)\}$, and the events $\{C(y,\omega)=C_k(x)\}$ admit the
    finite decomposition
    \begin{equation}
    \{\omega:\,C(y,\omega)=C_k(x)\}
    =\{\omega:\,\mathcal P^\omega(y)\subset C_k(x)\}\setminus
    \bigcup_{j<k}\{\omega:\,\mathcal P^\omega(y)\subset C_j(x)\},
    \end{equation}
    so they are measurable by finite Boolean operations.
\end{remark}

\subsection{Assembling the Argument}

All the ingredients required for the proof of Theorem~\ref{Theorem:Nagata_iff_approx_C1_smooth} are now in place. We can therefore assemble them to complete the argument.

\begin{proof}[Proof of Theorem~\ref{Theorem:Nagata_iff_approx_C1_smooth}]
	Let $M$ be a separable metric space with Nagata dimension $d\in\N$. We will show that
	\begin{equation}AppA_d(M)\lesssim O(\gamma_d(M)d^{7/2}).\end{equation}
	
	Fix $\gamma>\gamma_d(M)$ finite, and let $k\in\Z$ be sufficiently small.  
	By Lemma~\ref{Lemma:Naor_Silberman}, there exists a $2^k$-bounded cover $\mathcal{C}$ of $M$ with strong multiplicity $d+1$, and a discrete $2^k$-bounded $(100\gamma d^2,\frac{1}{d+1})$-padded random partition $\mathcal{P}$ of $M$ subordinated to $\mathcal{C}$. 
	Applying Theorem~\ref{Thm:From_padded_to_separating}, we obtain a $2^k$-bounded and $300\gamma d^2(d+1)$-separating random partition $\mathcal{Q}$ of $M$, still subordinated to $\mathcal{C}$. 
	Next, choose a $2^k$-dense subset $N\subset M$.  
	By Theorem~\ref{Theorem:Almost_retraction_from_sep_random_partition}, there exists a stochastic $2^{k+1}$-almost retraction $P\colon M\rightarrow\mathcal{F}(N)$ with strong multiplicity $d+1$, satisfying
	\begin{equation}\Lip(P)\leq 1200\gamma d^2(d+1)+1.\end{equation}
	We can now invoke Theorem~\ref{Theorem:Atlas_from_random_almost_projections}, which provides an open atlas $\mathcal{U}$ of $M$ modeled on $(\R^d,\|\cdot\|_2)$ such that
	\begin{equation}
	AppA(M,\mathcal{U},\,2^k+2^k(1200\gamma d^2(d+1)+1))
	\leq \sqrt{d}\,\big(1200\gamma d^2(d+1)+1\big).
	\end{equation}
	Since, for any $\varepsilon>0$, one can choose $k\in\Z$ small enough so that 
	\begin{equation}2^k+2^k(1200\gamma d^2(d+1)+1)<\varepsilon,\end{equation} 
	it follows that 
	\begin{equation}AppA_d(M)\leq \sqrt{d}\,\big(1200\gamma d^2(d+1)+1\big).\end{equation}
	As the above bound holds for every $\gamma>\gamma_d(M)$, the desired conclusion follows.
\end{proof}

\begin{remark}\label{rem:ae(M)}
	Recall that the \emph{absolute extendability constant} of a metric space $M$, denoted $ae(M)$, is the least constant $C\ge 1$ such that for every Banach space $Y$, every subset $A\subset M$, and every $1$-Lipschitz map $f\colon A\to Y$, there exists an extension $F\colon M\to Y$ with $\Lip(F)\le C$. 
	
	The bound 
	\begin{equation}AppA_d(M)\lesssim O\!\left(\gamma_d(M)d^{7/2}\right)\end{equation} 
	obtained in Theorem~\ref{Theorem:Nagata_iff_approx_C1_smooth} is unlikely to be optimal.  
	In \cite{NaorSilberman}, the authors proved that the absolute extendability constant satisfies 
	\begin{equation}ae(M)\lesssim O\!\left(\gamma_d(M)d^3\right)\end{equation} 
	starting from essentially the same padded random partitions employed here. The additional factor $d^{1/2}$ in our estimate arises from the choice of atlases modeled on $(\R^d,\|\cdot\|_2)$: in the proof of Theorem~\ref{Theorem:Atlas_from_random_almost_projections}, the linear isomorphism between $\mathcal{F}(A)$ and $\ell_2^d$ for sets $A$ with $|A|=d+1$ introduces a $\sqrt{d}$-loss due to the general upper bound on the Banach-Mazur distance between these spaces.

	Sharper quantitative bounds for Lipschitz extensions in finite Nagata dimension were later obtained by Basso in \cite{Basso24}.  
	In particular, Theorem~1.5 therein shows that 
	\begin{equation}ae(M)\lesssim O\!\left(\gamma_d(M)\log_2(d)\right),\end{equation} 
	which substantially improves the dependence on $d$. This improvement, however, relies on the Whitney-type extension scheme and partitions of unity, and it is not straightforward to adapt these methods to the approach we employ here. Moreover, for the purposes of the present work, specifically the application to property~(V$^*$) in Lipschitz-free spaces discussed in the next section, the precise value of the constant is immaterial.  
	The essential point is that $AppA_d(M)$ is finite for every metric space $M$ of finite Nagata dimension. However, it would be interesting to determine whether our estimates can be refined to match Basso's bound on $ae(M)$.
\end{remark}

Let us point out that the converse of Theorem~\ref{Theorem:Nagata_iff_approx_C1_smooth} does not hold in general: the finiteness of $AppA_d(M)$ does not imply that $\dim_N(M)\le d$. A simple counterexample is given by the sequence 
\begin{equation}
M:=\Big\{\frac{1}{n}: n\in\N\Big\}\cup\{0\}
\end{equation}
endowed with the usual metric of the real line. In this case $\dim_N(M)=1$ (easy to check), while $AppA_0(M)=1$. Indeed, given $\varepsilon>0$, there exists $n_0\in\N$ such that $\frac{1}{n}\in B(0,\varepsilon)$ for all $n\ge n_0$. Consider the atlas $\mathcal{U}=(U_k,\varphi_k)_{k=1}^{n_0}$ defined by 
\begin{equation}
U_k:=\Big\{\frac{1}{k}\Big\}\quad \text{for } k=1,\dots,n_0-1, 
\qquad U_{n_0}:=B(0,\varepsilon),
\end{equation}
and where each $\varphi_k\colon U_k\to\{0\}$ is the constant function. Then $\mathcal{U}$ is an open atlas modeled on $\R^0=\{0\}$. 

Clearly, any Lipschitz map $g\colon M\to Y$ such that $g(1/n)=g(0)$ for all $n\ge n_0$ is affine with respect to $\mathcal{U}$. Keeping this in mind, note that for every $1$-Lipschitz map $f\in\Lip(M,Y)$, we can find a $1$-Lipschitz map $g\in\Lip(M,Y)$ which is affine with respect to $\mathcal{U}$ and satisfies $\|g-f\|_\infty\le\varepsilon$. It follows that $AppA_0(M)=1$.
\medskip

In fact, there exist compact metric spaces of infinite Nagata dimension whose $0$-dimensional affine approximation modulus is still finite. In what follows, we describe such an example, inspired by the simple construction above. Before doing so, we first use Kuhn's cubic variant of Sperner's lemma to prove the following result which implies, incidentally, that $\dim_N(\R^{d})>d-1$.

\begin{lemma}
	Let $d,n\in\N$, and set $[n]:=\{0,1,\dots,n\}$. Consider the metric space $M=([n]^d,d_\infty)\subset \ell_\infty^n$. Then
	\begin{equation}
	\gamma_{d-1}(M)=n.
	\end{equation}
\end{lemma}

\begin{proof}
	We begin with the easy direction $\gamma_{d-1}(M)\le n$. Indeed, $\diam(M)=n$ and $d_\infty(x,y)\ge 1$ for all distinct $x,y\in M$, so the conclusion is immediate.
	
	Let us now prove that $\gamma_{d-1}(M)\ge n$. This follows from Kuhn's cubic version of Sperner's lemma (see \cite{Kuhn60}), which asserts the following.  
	If $\Gamma\colon [n]^d\to\{1,\dots,d+1\}$ is a coloring satisfying:
	\begin{itemize}
		\item[(i)] $\Gamma((x(j))_{j=1}^d)\le i$ whenever $x(i)=0$, and
		\item[(ii)] $\Gamma((x(j))_{j=1}^d)\neq i$ whenever $x(i)=n$,
	\end{itemize}
	then there exists a $d$-dimensional cube of side length $1$ containing at least one vertex of each color.  
	Equivalently, there exists an equilateral set $A\subset M$ (recall we are working with the $\ell_\infty$ metric) of diameter $1$ such that for every $i\in\{1,\dots,d+1\}$ there exists $x\in A$ with $\Gamma(x)=i$.
	
	To apply Kuhn's lemma in our setting, for each $i\in\{1,\dots,d\}$ define
	\begin{equation}
	F_i:=\big\{(x(j))_{j=1}^d\in[n]^d : x(i)=0\big\}.
	\end{equation}
	Given a cover $\mathcal{C}$ of $([n]^d,d_\infty)$, define the coloring $\Gamma_{\mathcal{C}}\colon[n]^d\to\{1,\dots,d+1\}$ by
	\begin{equation}
	\Gamma_\mathcal{C}(x):=
	\begin{cases}
		\min\{i\in\{1,\dots,d\} : \exists\, C\in\mathcal{C},~x\in C,~C\cap F_i\neq\emptyset\}, & \text{if such $i$ exists},\\[4pt]
		d+1, & \text{otherwise}.
	\end{cases}
	\end{equation}
	By construction, $\Gamma_{\mathcal{C}}$ always satisfies condition (i).  
	Condition (ii) holds provided $\diam(C)<n$ for every $C\in\mathcal{C}$. Indeed, if $x=(x(j))_{j=1}^d$ satisfies $x(i)=n$ for some $i$, and if $\Gamma_{\mathcal{C}}(x)=i$, then there exists $C\in\mathcal{C}$ containing both $x$ and a point $s=(s(j))_{j=1}^d\in C\cap F_i$. In that case,
	\begin{equation}
	d_\infty(s,x)\ge |s(i)-x(i)|=n,
	\end{equation}
	contradicting $\diam(C)<n$.
	
	Hence, if $\gamma<n$ and $s>1$ satisfy $\gamma s<n$, then for every cover $\mathcal{C}$ of $([n]^d,d_\infty)$ with $\diam(C)<\gamma s$ there exists an equilateral subset $A\subset M$ of diameter $1$ such that for every $i\in\{1,\dots,d+1\}$ there exists $x\in A$ with $\Gamma_\mathcal{C}(x)=i$. In particular, $\diam(A)<s$.
	
	Moreover, if two vertices $x_1,x_2\in M$ have distinct colors according to $\Gamma_{\mathcal{C}}$, then they must lie in distinct sets of the cover $\mathcal{C}$. Indeed, suppose that $\Gamma_{\mathcal{C}}(x_1)\neq \Gamma_{\mathcal{C}}(x_2)$, and without loss of generality assume $\Gamma_{\mathcal{C}}(x_1)<\Gamma_{\mathcal{C}}(x_2)$. Choose $C_1\in\mathcal{C}$ such that $x_1\in C_1$ and $C_1\cap F_{\Gamma_{\mathcal{C}}(x_1)}\neq\emptyset$. By minimality of $\Gamma_{\mathcal{C}}(x_2)$, we must have $x_2\notin C_1$, so any set $C_2\in\mathcal{C}$ containing $x_2$ satisfies $C_2\ne C_1$.  
	Therefore,
	\begin{equation}
	\big|\{C\in\mathcal{C}:A\cap C\ne\emptyset\}\big|\ge d+1,
	\end{equation}
	which shows that $\gamma_{d-1}(M)\ge n$, as claimed.
\end{proof}

\begin{example}
	Consider the Banach space $c_0$. Let $(g_n)_{n\in\N}$ be a strictly increasing sequence of natural numbers, and define $m_1:=1$ and
	\begin{equation}
	m_{n+1}:=\frac{1}{2}\,\frac{m_n}{g_{n+1}}, \qquad n\geq 1.
	\end{equation}
	Then $(m_n)_{n\in\N}$ decreases to $0$. For each $n\in\N$, with $[g_n]:=\{0,1,\dots,g_n\}$, set
	\begin{equation}
	M_n:=\big\{\,m_n\,(x(1),\dots,x(n),0,\dots)\in c_0 : (x(j))_{j=1}^n\in[g_n]^n\,\big\}.
	\end{equation}
	The space $M_n$ is (up to scaling) isometric to the grid $([g_n]^n,d_\infty)$ in $\ell_\infty^n$, so by the previous lemma $\gamma_{n-1}(M_n)=g_n$. Note that $M_n$ is finite and contained in $m_n g_n B_{c_0}$. By construction, $m_n g_n\to0$, hence
	\begin{equation}
	M:=\bigcup_{n\in\N}M_n
	\end{equation}
	is compact but has infinite Nagata dimension.
	
	However, $AppA_0(M)\le 2$. Indeed, let $\varepsilon>0$ and choose $n_0\in\N$ such that $m_{n_0}\le \varepsilon<m_{n_0-1}$. Observe that
	\begin{equation}
	B(0,\varepsilon)\subset\bigcup_{n\ge n_0}M_n\subset m_{n_0}g_{n_0}B_{c_0}.
	\end{equation}
	Since $M\setminus B(0,\varepsilon)=\{x_i\}_{i=1}^l$ is finite, consider the atlas $\mathcal{U}=(U_i,\varphi_i)_{i=0}^l$ defined by
	\begin{equation}
	U_0:=B(0,\varepsilon), \qquad U_i:=\{x_i\}\ (i=1,\dots,l),
	\end{equation}
	and $\varphi_i\colon U_i\to\{0\}$ the constant map for all $i$. Then $\mathcal{U}$ is an open atlas of $M$ modeled on the trivial Banach space.
	
	Any Lipschitz map that is constant on $U_0$ is affine with respect to $\mathcal{U}$. Hence, for any Banach space $Y$ and any $1$-Lipschitz map $f\colon M\to Y$, define
	\begin{equation}
	h(x):=
	\begin{cases}
		f(0), & x\in U_0,\\
		f(x), & x\in M\setminus U_0.
	\end{cases}
	\end{equation}
	Then $h$ is affine with respect to $\mathcal{U}$ and $\|f-h\|_\infty<\varepsilon$. Since $\|h\|_\mathcal{U}=0$ trivially, it remains only to estimate $\Lip(h)$.
	
	It is enough to consider pairs $(x,y)$ with $x\in U_0$ and $y\in M\setminus U_0$. Write
	\begin{equation}
	y=m_{n_1}(c(1),\dots,c(n_1),0,\dots)
	\end{equation}
	for some $n_1<n_0$ and $(c(j))_{j=1}^{n_1}\in[g_{n_1}]^{n_1}$. Then, since $y \neq 0$, there exists $j_1\in\{1,\dots,n_1\}$ such that $y(j_1)=m_{n_1}c(j_1)\geq m_{n_1}$. We thus obtain
	\begin{align}
		d(x,y)&=\|x-y\|_\infty
		\ge y(j_1)-x(j_1)
		\ge m_{n_1}-g_{n_0}m_{n_0}\\
		&\ge m_{n_0-1}-g_{n_0}m_{n_0}
		=2m_{n_0}g_{n_0}-g_{n_0}m_{n_0}
		=g_{n_0}m_{n_0}
		\ge d(x,0).\nonumber
	\end{align}
	Therefore,
	\begin{equation}
	\frac{\|h(x)-h(y)\|_Y}{d(x,y)}
	=\frac{\|f(0)-f(y)\|_Y}{d(x,y)}
	\le \frac{\|f(0)-f(x)\|_Y+\|f(x)-f(y)\|_Y}{d(x,y)}
	\le 2,
	\end{equation}
	and hence $\Lip(h)\le 2$, as required.
\end{example}

\section{Approximate Continuous Upper Gradient Structure} \label{section:ACUG}

In the remainder of this paper, we apply our results on affine approximation for metric spaces with finite Nagata dimension to extract information on the linear structure of the associated Lipschitz-free spaces $\F(M)$. In particular, we will focus on Pe\l{}czy\'nski's property~(V*). Basic facts about Lipschitz-free spaces and this property are recalled in Sections~\ref{subsection:Free} and~\ref{subsection:Vstar}.

Property~(V*) has been studied in several works within the theory of Lipschitz-free spaces. In particular, \cite{KP18} established that $\F(K)$ has property~(V*) whenever $K$ is a compact subset of a superreflexive Banach space. More recently, \cite{APQ24} provided a detailed analysis of property~(V*) for various classes of metric spaces, proving among other results that $\F(M)$ enjoys property~(V*) when $M$ is either locally compact and purely $1$-unrectifiable, an $\ell_p$ space for $1<p<\infty$, or a Carnot group. In both works, the argument ultimately traces back to Bourgain's method~\cite{Bourgain, Bourgain2}, which shows that the dual space $(\mathcal{C}^1([0,1]^n))^*$ possesses property~(V*).

A central ingredient in all these developments is the existence of uniform approximations of Lipschitz functions on $M$ by $\mathcal{C}^1$-smooth functions, with control of the Lipschitz constant up to a universal multiplicative factor. As will become apparent, the main result from Section~\ref{section:AppA-FND} provides precisely the kind of approximation mechanism required to adapt Bourgain's argument to a much broader geometric setting.

In what follows, we will isolate the structural features of $M$ that ensure Bourgain's method applies. To that end, we introduce the notion of an \emph{approximate continuous upper gradient $X$-structure}.

Given a topological space $X$ and a Banach space $Y$, let $C_b(X,Y)$ denote the Banach space of all bounded continuous maps $X\to Y$ endowed with the supremum norm.

\begin{definition}
	\label{Def:App_cont_upp_grad_structure}
	Let $M$ be a complete metric space and $X$ a Banach space. We say that $M$ has an \emph{approximate continuous upper gradient $X$-structure} (abbreviated \emph{ACUG $X$-structure}) if there exists $C>0$ such that for every $\varepsilon>0$ one can find a (not necessarily closed) linear subspace $H\subset \Lip(M)$ and a linear operator $\Phi_H\colon H\to \mathcal{C}_b(M,X^*)$ satisfying:
	\begin{enumerate}
		\item For every $f\in \Lip(M)$ with $\|f\|_L\leq 1$, there exists $g\in H$ such that $\|g\|_L\le C$, $\|\Phi_H g\|_\infty\le C$, and $\|f-g\|_\infty\le \varepsilon$.
		\item For every $f\in H$, the function $\|\cdot\|_{X^*}\circ (\Phi_H f)\colon M\to [0,+\infty)$ is an upper gradient of $f$.
	\end{enumerate}
\end{definition}

\begin{remark}
	The definition remains unchanged if we replace $\Lip(M)$ with $\Lip_0(M)$ throughout. 
	Indeed, for any $f\in\Lip(M)$, the function $f-f(0)$ belongs to $\Lip_0(M)$ and has the same Lipschitz constant and the same upper gradients as $f$. 
	Hence the existence of an approximate continuous upper gradient $X$-structure can be equivalently formulated within $\Lip_0(M)$.
\end{remark}

This definition precisely captures the analytic ingredients needed for Bourgain's scheme to establish property~(V*), as formalized in the following theorem.

\begin{theorem}
	\label{Th:Superreflexive_ACUG_implies_V*}
	Let $M$ be a compact quasiconvex metric space admitting an ACUG $X$-structure for some superreflexive Banach space~$X$. Then the Lipschitz-free space $\F(M)$ has property~(V*).
\end{theorem}

Since the proof of Theorem~\ref{Th:Superreflexive_ACUG_implies_V*} is rather technical and closely parallels the argument of, for instance, \cite[Theorem~9]{KP18}, we defer it to the final section for the sake of readability. In the next sections, we show that the quasiconvexity assumption can in fact be removed for a wide class of metric spaces. Before that, we identify a broad and natural family of spaces satisfying the hypotheses of Theorem~\ref{Th:Superreflexive_ACUG_implies_V*}.

\begin{theorem}
	\label{thm:FNDimpliesACUG}
	Let $M$ be a compact metric space with finite Nagata dimension. Then $M$ admits an ACUG $\ell_2$-structure. Moreover, the operators $\Phi_H$ in Definition~\ref{Def:App_cont_upp_grad_structure} may be chosen with range contained in $\Lip(M,\ell_2)\subset \mathcal{C}_b(M,\ell_2)$.
\end{theorem}

\begin{proof}
	Set $d:=\dim_N(M)$. Fix $\varepsilon>0$. By Theorem~\ref{Theorem:Nagata_iff_approx_C1_smooth}, there exists a constant $C_1\ge 1$ (depending only on $M$) and an open atlas $\mathcal{U}=(U_i,\varphi_i)_{i\in I}$ modeled on $\ell_2^d=(\R^d,\|\cdot\|_2)$ such that for every $1$-Lipschitz $f\in\Lip(M)$ there exists $g\in\Lip(M)$ with
	\begin{equation}
	\|g\|_{L}\le C_1,\quad \|f-g\|_\infty\le \varepsilon,\quad \text{and}\quad g \text{ affine with respect to }\mathcal{U},
	\end{equation}
	and, writing $(Ag)_i\in \ell_2^d = (\ell_2^d)^*$ for the affine coefficient on $U_i$ (see Definition~\ref{def:affine_wrt_atlas}),
	\begin{equation}
	\| g \|_{\mathcal U}:=\sup\big\{\Lip(\varphi_i)\,\|(Ag)_i\|_2 : i \in I, \; x\in\operatorname{int}(U_i)\big\}\le C_1.
	\end{equation}
	
	\emph{Step 1: Fixing a finite atlas with controlled overlap.}
	By compactness of $M$, we may assume $\mathcal{U}=(U_i,\varphi_i)_{i=1}^n$ is finite (by restricting to a finite subatlas and using Remark~\ref{rem:refinement}). Moreover, there exists $\delta>0$ such that for every $x\in M$, there exists $i$ with $B(x,\delta)\subset U_i$. By the definition of Nagata dimension, there exist $s>0$ and a cover $\mathcal{A}=(A_j)_{j=1}^m$ with $\diam(A_j)\le \delta$ such that any set of diameter $\le s$ meets at most $d+1$ members of $\mathcal{A}$. Refining $\mathcal{U}$ by $\mathcal{A}$ (again using Remark~\ref{rem:refinement}), we may and do assume that
	\begin{equation}\label{eq:overlap}
		\text{every subset of $M$ of diameter $\le s$ meets at most $d+1$ of the $U_i$'s.}
	\end{equation}

	\emph{Step 2: The subspace $H$.}
	Let $H\subset \Lip(M)$ be the linear subspace of all functions affine with respect to $\mathcal{U}$. By the first paragraph, for every $f\in \Lip(M)$ with $\|f\|_L\le 1$ there exists $g\in H$ such that
	\begin{equation}\label{eq:approx-g}
		\|g\|_{L}\le C_1,\qquad \|f-g\|_\infty\le \varepsilon,\qquad \| g \|_{\mathcal U}\le C_1.
	\end{equation}
	Thus $H$ already meets the approximation part of item (1) in Definition~\ref{Def:App_cont_upp_grad_structure}. It remains to construct a \emph{linear} operator $\Phi_H\colon H\to \mathcal{C}_b(M,\ell_2)$ satisfying the bounds in (1) and the upper gradient property in (2), and to see that actually $\Phi_H(H)\subset \Lip(M,\ell_2)$.
	\smallskip
	
	\emph{Step 3: Auxiliary domains $\widehat{U_i}$ and piecewise constant differentials.}
	For each $i\in\{1,\dots,n\}$ set
	\begin{equation}
	\widehat{U_i}:=U_i\ \cup\ \Big(M\setminus B\!\Big(U_i,\frac{s}{2}\Big)\Big).
	\end{equation}
	Thus $\widehat{U_i}$ is obtained by removing the $\frac{s}{2}$-neighborhood “ring” around $U_i$. For $f\in H$ write $(Af)_i\in \ell_2^d$ for the (constant) affine coefficient on $U_i$, so that
	\begin{equation}
	f(y)-f(x)=\langle (Af)_{i},\ \varphi_i(y)-\varphi_i(x)\rangle\qquad (x,y\in U_i).
	\end{equation}
	Recall that $f\mapsto (Af)_i$ is linear. Define a map
	\begin{equation}
	D_i f\colon \widehat{U_i}\longrightarrow \ell_2^d,\qquad
	D_i f(x):=
	\begin{cases}
		\Lip(\varphi_i)\,(Af)_i,& x\in U_i,\\[2pt]
		0,& x\in M\setminus B(U_i,\frac{s}{2}).
	\end{cases}
	\end{equation}
	Then the map $D_i\colon H\to \Lip(\widehat{U_i},\ell_2^d)$, defined by $f \mapsto D_i f$, is linear, satisfies $\|D_i f\|_\infty=\Lip(\varphi_i)\,\|(Af)_i\|_2\le \|f\|_{\mathcal U}$, and
	\begin{equation}
	\|D_i f\|_{L}\ \le\ \frac{2}{s}\,\|f\|_{\mathcal U},
	\end{equation}
	because $D_i f$ takes two constant values on sets at distance at least $\frac{s}{2}$.
	\smallskip
	
	\emph{Step 4: Lipschitz extension through free spaces.}
	Since $\widehat{U_i}\subset M$ and $M$ has finite Nagata dimension, by the extension theorem of Naor--Silberman \cite[Corollary 5.2]{NaorSilberman} there exists a constant $C_2\ge 1$ (depending only on $M$) so that the natural inclusion $\delta\colon \widehat{U_i}\rightarrow \mathcal{F}(\widehat{U_i})$ extends to a Lipschitz map
	\begin{equation}
	P_i\colon M\longrightarrow \mathcal{F}(\widehat{U_i}) \text{ with } \|P_i\|_L\le C_2.
	\end{equation}

	Let $\overline{D_i f}\in \mathcal{L}\!\big(\mathcal{F}(\widehat{U_i}),\ell_2^d\big)$ denote the canonical linearization of $D_i f$ (so that $\|\overline{D_i f}\|=\|D_i f\|_{L}$ and $\overline{D_i f}(\delta(x))=D_i f(x)$). Define
	\begin{equation}
	G_i f\ :=\ \overline{D_i f}\circ P_i\ \in \ \Lip(M,\ell_2^d).
	\end{equation}
	Then $G_i\colon H\to \Lip(M,\ell_2^d)$ is linear. Moreover, for $x\in \widehat{U_i}$ we have $G_i f(x)=D_i f(x)$, so in particular $G_i f$ vanishes on $M\setminus B(U_i,\frac{s}{2})$. Now if $x \in B(U_i, \frac{s}{2})$ and $y \in U_i$ is such that $d(x,y) \leq \frac{s}{2}$, then:
	\begin{align}
		\| G_i f(x) \|_2 &= \|\overline{D_i f} (P_i (y)) + \overline{D_i f} (P_i (x) - P_i (y)) \|_2 \\ 
		&\leq \|D_i f (y)\|_2 + \| \overline{D_i f} \| \|P_i (x) - P_i (y) \| \nonumber\\
		&\leq \|D_i f \|_\infty + \|D_i f \|_L \|P_i\|_L d(x,y) \nonumber\\ 
		&\leq \|f\|_{\mathcal{U}} + \frac{2\|f\|_{\mathcal{U}}}{s} C_2 \frac{s}{2} = \|f\|_{\mathcal{U}} + \|f\|_{\mathcal{U}} C_2.\nonumber
	\end{align}
	Thus we obtain
	\begin{equation}\label{eq:Gi-bounds}
		\|G_i f\|_\infty  \le \|f\|_{\mathcal U}\,(1+C_2),\quad \text{and} \quad
		\|G_i f\|_{L} \le \|\overline{D_if}\|_{L}\|P_i\|_{L} \le \frac{2C_2}{s}\,\|f\|_{\mathcal U}.
	\end{equation}
	
	\emph{Step 5: The operator $\Phi_H$.}
	Define the map $\Phi_H$ on $H$ by 
	\begin{equation}
	\Phi_H f \ :=\ \big(G_1 f,\dots,G_n f\big)\ \in\ \Lip\Big(M,\bigoplus_{i=1}^n \ell_2^d\Big),
	\end{equation}
	where we equip $\bigoplus_{i=1}^n \ell_2^d$ with the $\ell_2$-sum norm and identify it isometrically with a subspace of $\ell_2$. Then $\Phi_H\colon H\to \Lip(M,\ell_2)$ is linear.
	
	By~\eqref{eq:overlap}, for each $x\in M$ the index set
	\begin{equation}
	I(x):=\Big\{\, i\in\{1,\dots,n\} : x\in B\!\big(U_i,\tfrac{s}{2}\big)\,\Big\}
	\end{equation}
	has cardinality at most $d+1$. For $i\notin I(x)$ we have $G_i f(x)=0$, because $G_i f$ vanishes on $M\setminus B(U_i,\frac{s}{2})$. Therefore, combined with~\eqref{eq:Gi-bounds}, we obtain:
	\begin{align}
		\|(\Phi_H f)(x)\|_{\ell_2}
		&=\Big(\sum_{i\in I(x)}\|G_i f(x)\|_{\ell_2^d}^2\Big)^{1/2}
		\le \sqrt{d+1}\sup_i \|G_i f\|_\infty\\
		&\le \sqrt{d+1}(1+C_2)\|f\|_{\mathcal U}.\nonumber
	\end{align}
	In particular, if $g\in H$ satisfies $\|g\|_{\mathcal U}\le C_1$ (as in~\eqref{eq:approx-g}), then
	\begin{equation}\label{eq:PhiH-infty}
		\|\Phi_H g\|_\infty\ \le\ \sqrt{d+1}\,(1+C_2)\,C_1.
	\end{equation}
	This provides the uniform bound required in item (1) of Definition~\ref{Def:App_cont_upp_grad_structure}. 
	\smallskip
	
	\emph{Step 6: Upper gradient domination.}
	Fix $f\in H$ and $x\in \operatorname{int}(U_i)$. Then $G_i f(x)=D_i f(x)=\Lip(\varphi_i)\,(Af)_i$, hence
	\begin{equation}
	\|(\Phi_H f)(x)\|_{\ell_2}\ \ge\ \|G_i f(x)\|_{\ell_2^d}\ =\ \Lip(\varphi_i)\,\|(Af)_i\|.
	\end{equation}
	Taking the supremum over all $i$ with $x\in \operatorname{int}(U_i)$, we obtain
	\begin{equation}
	\|(\Phi_H f)(x)\|_{\ell_2}\ \ge\ |Df|_{\mathcal U}(x),
	\end{equation}
	where $|Df|_{\mathcal U}$ is the maximal upper gradient associated to $\mathcal{U}$. Since $|Df|_{\mathcal U}$ is an upper gradient of $f$, so is $\|\cdot\|_{\ell_2}\circ (\Phi_H f)$.
	\smallskip
	
	\emph{Step 7: Conclusion.}
	Given $f\in \Lip(M)$ with $\|f\|_L\le 1$, choose $g\in H$ as in~\eqref{eq:approx-g}. Then $g\in H$, $\|g\|_L\le C_1$, $\|f-g\|_\infty\le \varepsilon$, and by~\eqref{eq:PhiH-infty} we have $\|\Phi_H g\|_\infty\le C:=\sqrt{d+1}\,(1+C_2)\,C_1$. This verifies item (1). Item (2) was shown in Step~6. Finally, $\Phi_H(H)\subset \Lip(M,\ell_2)$ by construction (Step~5), which gives the “moreover” assertion.
\end{proof}

\section{Almost isometric embeddings into compact geodesic spaces}

Theorem~\ref{Th:Superreflexive_ACUG_implies_V*} applies only to compact quasiconvex spaces, and its proof crucially relies on integration along rectifiable curves. On the other hand, using substantially different techniques, Aliaga, Pernecká and Quero~\cite{APQ24} recently proved that Lipschitz-free spaces over locally compact purely $1$-unrectifiable spaces, where rectifiable curves of positive measure do not exist, also satisfy property~(V*). This naturally raises the question of whether the quasiconvexity assumption in Theorem~\ref{Th:Superreflexive_ACUG_implies_V*} can be dispensed with.
Since property~(V*) is inherited by subspaces, a natural strategy for removing this hypothesis is to establish an appropriate embedding result, as formalized below.

\begin{question}
	\label{Question:ACUG_embeds_into_geodesic_ACUG}
	Let $M$ be a compact metric space admitting a superreflexive ACUG structure. Does $M$ admit a bi-Lipschitz embedding into a compact quasiconvex metric space that also carries a superreflexive ACUG structure?
\end{question}

We do not know the answer to Question~\ref{Question:ACUG_embeds_into_geodesic_ACUG} in full generality. However, in this section we introduce a general construction which provides partial progress toward this goal. This method applies to all classes of compact metric spaces whose Lipschitz-free spaces are already known to possess property~(V*).

\begin{proposition}
	\label{Prop:Embedding_into_geodesic_ACUG}
	For every metric space $M$ and every $\varepsilon>0$, there exists a $(1+\varepsilon)$-quasiconvex metric space $G_\varepsilon(M)$ satisfying the following properties:
	\begin{enumerate}
		\item $M$ embeds isometrically into $G_\varepsilon(M)$;
		\item $G_\varepsilon(M)$ has the same density character as $M$, and is compact whenever $M$ is compact;
		\item If $M$ has finite Nagata dimension, then $\dim_{N}(G_\varepsilon(M)) \le \dim_N(M)+1$;
		\item If $M$ is compact and purely $1$-unrectifiable, then $G_\varepsilon(M)$ admits an ACUG $\mathbb{R}$-structure.
	\end{enumerate}
\end{proposition}

Recall that, by Lemma~\ref{lemma:quasiconvexTOgeo}, Proposition~\ref{Prop:Embedding_into_geodesic_ACUG} ensures that, for every $\varepsilon>0$, $M$ admits a bi-Lipschitz embedding into a geodesic metric space with distortion at most $1+\varepsilon$.

\subsection{Complete geodesic graph generated by a metric space}

Fix a metric space $M$. Since $M$ embeds isometrically into its completion, we may assume without loss of generality that $M$ is complete. For every $\varepsilon>0$, the space $G_\varepsilon(M)$ will be defined as a subset of what we call the \emph{complete geodesic graph generated by $M$}, denoted by $G(M)$, which is a canonical object naturally associated to $M$.

To define $G(M)$, embed $M$ into a vector space $X$ in such a way that the elements of $M$ are linearly independent in $X$ (for instance, one may take $X=\mathcal{F}(M)$ with the canonical embedding of $M$ into its Lipschitz-free space). We then set
\begin{equation}
G(M):=\bigcup_{x,y\in M} [x,y]\subset X,
\end{equation}
where $[x,y]:=\{\lambda y+(1-\lambda)x : \lambda\in[0,1]\}\subset X$. Since the points of $M$ are linearly independent in $X$, any two distinct segments intersect (at most) only at their endpoints; that is, for $x_1,x_2,y_1,y_2\in M$,
\begin{equation}
[x_1,x_2]\cap[y_1,y_2]=\{x_1,x_2\}\cap \{y_1,y_2\}.
\end{equation}

For $x\neq y\in M$, consider the natural bijection $\varphi_{x,y}\colon [x,y]\to [0,d(x,y)]$ defined by
\begin{equation}
\varphi_{x,y}(\lambda y+(1-\lambda)x):=\lambda\, d(x,y).
\end{equation}
We endow $G(M)$ with the largest metric $d$ making each $\varphi_{x,y}$ an isometry. In particular, for two distinct segments $[x_1,x_2],[y_1,y_2]\subset G(M)$ and points $p\in [x_1,x_2]$, $q\in [y_1,y_2]$, one has
\begin{equation}
d(p,q)=\min\{\,d(p,x_i)+d(x_i,y_j)+d(y_j,q): i,j\in\{1,2\}\,\}.
\end{equation}

Intuitively, $G(M)$ can be viewed as the metric realization of the complete weighted graph on vertex set $M$, where each edge $[x,y]$ has length $d(x,y)$, endowed with the corresponding shortest-path metric. 

The space $G(M)$ is clearly geodesic and satisfies item (1) of Proposition~\ref{Prop:Embedding_into_geodesic_ACUG}. Moreover, less obviously, it also satisfies (3) and (4). However, it is in general \emph{too big} to satisfy condition (2).

\subsection{Construction of \texorpdfstring{$G_\varepsilon(M)$}{Gε(M)}}

The complete geodesic graph $G(M)$ need not be compact even when $M$ is; in fact, $G(M)$ is non-compact whenever $M$ is infinite. To address this, we will define a compact subspace $G_\varepsilon(M)\subset G(M)$ by removing suitably chosen edges in a controlled manner.

Let $M$ be a metric space and fix $\varepsilon>0$. Choose $\delta\in (0,\ep)$ such that $1+\delta<\frac{3}{2}$ and $(1+\delta)(1+4\delta)<1+\varepsilon$. For each integer $n\ge 1$, set $r_n := (\frac{3}{4})^{n-1}$ and take a decreasing sequence $(\varepsilon_n)_n$ of positive numbers such that $\varepsilon_1<\delta$ and the ratio $\frac{\varepsilon_n}{r_{n+1}}$ decreases sufficiently fast to ensure
\begin{equation}
	\label{eq:Ineq_prod_quasiconv_1}
	\prod_{n=1}^{\infty}\Big(1+6\frac{\varepsilon_n}{r_{n+1}}\Big)<1+\delta.
\end{equation}
In particular, for every $n \in \N$,
\begin{equation}
	\label{eq:Inequality_prod_quasiconv_2}
	1+6\frac{\varepsilon_n}{r_{n+1}}<\frac{3}{2}\quad\text{and therefore}\quad
	\prod_{k=1}^n\Big(1+6\frac{\varepsilon_k}{r_{k+1}}\Big)<\Big(\frac{3}{2}\Big)^n.
\end{equation}

For each $n\in\mathbb{N}$, let $D_n$ be an $\varepsilon_n$-dense subset of $M$ of minimal cardinality. In particular, if $M$ is compact then each $D_n$ is finite. Define
\begin{equation}
S_n:=\big\{[x,y]: x,y\in D_n,\ d(x,y)\le r_n+2\varepsilon_n\big\}\subset G(M),
\end{equation}
and
\begin{equation}
T_n:=\big\{[x,y]: x,y\in D_1,\ n-2\delta\le d(x,y)\le n+1+2\delta\big\}\subset G(M)
\end{equation}
for all $n\ge 1$. Finally, set
\begin{equation}
G_\varepsilon(M):=M\ \cup\ \Big(\bigcup_{n\in\mathbb{N}}S_n\Big)\ \cup\ \Big(\bigcup_{n\in\mathbb{N}}T_n\Big)\ \subset G(M).
\end{equation}
Note that $G_\varepsilon(M)$ is a closed subspace of $G(M)$. Indeed, if $x \in G(M) \setminus G_\varepsilon(M)$, then $x \notin M$, and there exist distinct points $p,q \in M$ such that $x \in (p,q) := [p,q] \setminus \{p,q\}$. By the definition of the metric on $G(M)$, $d(x, G_\varepsilon(M)) = \min\{d(x,p),\, d(x,q)\} > 0$. Hence $G(M)\setminus G_\varepsilon(M)$ is open, and $G_\varepsilon(M)$ is closed. It is also immediate that $G_\varepsilon(M)$ has the same density character as $M$.

\begin{proposition}
	If $M$ is a compact metric space, then so is $G_\varepsilon(M)$.
\end{proposition}

\begin{proof}
	Since $G(M)$ is complete, it suffices to verify that $G_\varepsilon(M)$ is totally bounded. As $M$ is compact, it is bounded, so there exists $n_1\in\mathbb{N}$ such that $T_n=\emptyset$ for all $n\ge n_1$.
	
	Fix $\eta>0$. Choose $n_2\in\mathbb{N}$ such that $r_{n_2}<\eta$. It is readily seen that
	\begin{equation}
	K:=M\cup\Big(\bigcup_{k=1}^{n_2}S_k\Big)\cup\Big(\bigcup_{k=1}^{n_1}T_k\Big)
	\end{equation}
	is compact and $r_{n_2}$-dense in $G_\varepsilon(M)$. Hence $G_\varepsilon(M)\subset [K]_\eta$,
	which proves that $G_\varepsilon(M)$ is totally bounded, and therefore compact.
\end{proof}

We now prove that $G_\varepsilon(M)$ is quasiconvex. To this end, we introduce the following auxiliary function. Let $f : [0,1] \to [1 , 1+\delta]$ be the nondecreasing map defined by $f(0):=1$ and
\begin{equation}
f(t) := 1 + 6 \frac{\varepsilon_n}{r_{n+1}} \quad \text{whenever } t \in (r_{n+1}, r_n], \; n \in \mathbb{N}.
\end{equation}
It follows from \eqref{eq:Ineq_prod_quasiconv_1}, \eqref{eq:Inequality_prod_quasiconv_2}, and the choice of the sequence $(r_n)_n$ that, for every $n \in \mathbb{N}$,
\begin{equation}
	\label{eq:ineq_prod_function_f_1}
	\prod_{k=1}^n f(r_k)\leq 1+\delta,
\end{equation}
and
\begin{equation}
	\label{eq:ineq_prod_function_f_2}
	\frac{\prod_{k=1}^n f(r_k)}{2^n}\leq r_{n+1}.
\end{equation}

We begin with a technical lemma.

\begin{lemma}\label{lemma:f}
	For every $p,q \in G_\varepsilon(M)$ with $d(p,q) \leq 1$, one has
	\begin{equation}
	I_{p,q} := B\!\left(p,\, f(d(p,q)) \frac{d(p,q)}{2}\right)
	\cap
	B\!\left(q,\, f(d(p,q)) \frac{d(p,q)}{2}\right)
	\neq \emptyset.
	\end{equation}
\end{lemma}

\begin{proof}
	Fix $p \neq q \in G_\varepsilon(M)$ with $d(p,q) \le 1$.  
	If $p$ and $q$ lie on the same segment $[x,y]$ contained in some $S_n$ or $T_n$, then since $f \ge 1$ and $[x,y]\subset G_\varepsilon(M)$, the conclusion is immediate.
	
	Otherwise, there exist $m_p, m_q \in M$ such that
	\begin{equation}
	d(p,q) = d(p,m_p) + d(m_p , m_q) + d(m_q ,q).
	\end{equation}
	If $m_p=m_q$, then $p$ and $q$ belong to two distinct edges intersecting at $m_p$, hence the geodesic $[p,m_p]\cup[m_p,q]$ lies in $G_\varepsilon(M)$, and the claim again follows.  
	
	Assume now that $m_p \neq m_q$. Then there exists $n\in\mathbb{N}$ such that $d(m_p,m_q)\in (r_{n+1},r_n]$, and points $d_p,d_q \in D_n$ satisfying $d(m_p,d_p)\le \varepsilon_n$ and $d(m_q,d_q)\le \varepsilon_n$. Note that $[d_p,d_q]\in S_n\subset G_\varepsilon(M)$.  
	Let $c_G$ denote the midpoint between $p$ and $q$ in the metric space $G(M)$, lying on the geodesic $[p,m_p]\cup[m_p,m_q]\cup[m_q,q]$.  
	Up to switching $p$ and $q$, we may assume $c_G\in [m_p,m_q]\cup[m_q,q]$.  
	If $c_G\in [m_q,q]$, then $[m_q,q]\subset G_\varepsilon(M)$, so $c_G\in I_{p,q}$ and the lemma is proved.  
	Otherwise, suppose $c_G\in [m_p,m_q]$. Let $\lambda\in[0,1]$ be such that $d(m_p,c_G)=\lambda d(m_p,m_q)$, and let $c_\varepsilon\in [d_p,d_q]\subset G_\varepsilon(M)$ satisfy $d(d_p,c_\varepsilon)=\lambda d(d_p,d_q)$.  
	Then:
	\begin{align}
		d(p,c_\varepsilon)
		&\le d(p,m_p)+d(m_p,d_p)+d(d_p,c_\varepsilon) \\
		&\le d(p,m_p)+\varepsilon_n+\lambda d(d_p,d_q)\nonumber \\
		&\le d(p,m_p)+\varepsilon_n+\lambda(d(m_p,m_q)+2\varepsilon_n)\nonumber \\
		&\le d(p,m_p)+d(m_p,c_G)+3\varepsilon_n
		\le \frac{d(p,q)}{2}+3\varepsilon_n.\nonumber
	\end{align}
	A similar estimate holds for $d(q,c_\varepsilon)$.  
	Finally,
	\begin{equation}
	f(d(p,q)) \frac{d(p,q)}{2}
	\ge f(d(m_p,m_q)) \frac{d(p,q)}{2}
	= \Big(1+6\frac{\varepsilon_n}{r_{n+1}}\Big)\frac{d(p,q)}{2}
	\ge \frac{d(p,q)}{2}+3\varepsilon_n,
	\end{equation}
	so $c_\varepsilon \in I_{p,q}$, completing the proof.
\end{proof}

\begin{proposition}
	The space $G_\varepsilon(M)$ is $(1+\varepsilon)$-quasiconvex.
\end{proposition}

\begin{proof}
	Let $x\neq y\in G_\varepsilon(M)$.  
	Assume first that $d(x,y)\le 1$.	We construct a curve $\gamma:[0,1]\to G_\varepsilon(M)$ joining $x$ and $y$ with $\operatorname{Lip}(\gamma)\le 1+\varepsilon$.
	
	By Lemma~\ref{lemma:f}, there exists $\gamma(1/2)\in G_\varepsilon(M)$ such that
	\begin{equation}
	\gamma(\tfrac{1}{2}) \in B\!\left(x, f(d(x,y))\frac{d(x,y)}{2}\right)
	\cap
	B\!\left(y, f(d(x,y))\frac{d(x,y)}{2}\right).
	\end{equation}
	Since $f(d(x,y))\le f(r_1)$, we obtain from \eqref{eq:ineq_prod_function_f_1} and \eqref{eq:ineq_prod_function_f_2} that
	\begin{equation}\label{eq:Quasiconv_step_1}
		d\!\left(x,\gamma\!\left(\tfrac{1}{2}\right)\right)
		\le f(r_1) \frac{d(x,y)}{2} 
		\le \min\!\Big\{(1+\delta)d\!\left(0,\tfrac{1}{2}\right),\,r_2\Big\},
	\end{equation}
	and the same bound holds for $d(\gamma(1/2),y)$.
	
	Proceeding recursively, we define $\gamma(\tfrac{k}{2^n})$ at dyadic points so that, for every $n\in\mathbb{N}$ and $k\in\{0,\ldots,2^n-1\}$,
	\begin{equation}
	d\!\big(\gamma(\tfrac{k}{2^n}), \gamma(\tfrac{k+1}{2^n})\big)
	\le \Big(\prod_{j=1}^n f(r_j)\Big)\frac{d(x,y)}{2^n}
	\le \min\!\Big\{(1+\delta)d\!\left(\tfrac{k}{2^n},\tfrac{k+1}{2^n}\right),r_{n+1}\Big\}.
	\end{equation}
	By density of dyadic rationals, $\gamma$ extends continuously to $[0,1]$ with 
	\begin{equation}
	\operatorname{Lip}(\gamma)
	\le \prod_{n=1}^\infty \Big(1+6\frac{\varepsilon_n}{r_{n+1}}\Big)
	<1+\delta.
	\end{equation}
	Since $1+\delta<1+\varepsilon$, the conclusion follows for $d(x,y)\le 1$.
	
	If $d(x,y)\ge 1$, choose $n_0\in\mathbb{N}$ with $n_0\le d(x,y)< n_0+1$ and points $p,q\in D_1$ such that $d(x,p),d(y,q)\le\varepsilon_1<\delta$. Then
	\begin{equation}
	n_0-2\delta\le d(p,q)\le n_0+1+2\delta,
	\end{equation}
	hence $[p,q]\in T_{n_0}$.  
	Let $\gamma_{p,q}:[0,d(p,q)]\to G_\varepsilon(M)$ be the corresponding $1$-Lipschitz geodesic joining $p$ and $q$.  
	By the first part, there exist $1$-Lipschitz maps $\gamma_{x,p}$ and $\gamma_{y,q}$ joining $x$ to $p$ and $y$ to $q$, respectively, each parametrized on an interval of length at most $(1+\delta)d(x,p)$ and $(1+\delta)d(y,q)$.  
	Concatenating these three arcs yields a curve $\gamma_{x,y}$, joining $x$ to $y$, and of length
	\begin{align}
		\text{length}(\gamma_{x,y})&=(1+\delta)(d(x,p)+d(q,y))+d(p,q)\\
		&\leq (1+\delta)(d(x,p)+d(q,y)+d(p,q))\nonumber\\
		&\leq (1+\delta)(2\delta+d(x,y)+2\delta)\nonumber\\
		&\leq (1+\delta)(1+4\delta)d(x,y)\nonumber\\
		&\leq(1+\varepsilon)d(x,y),\nonumber
	\end{align}
	and the result follows.
\end{proof}

Now that we have suitably identified a subset of $G(M)$ satisfying (1) and (2) of Proposition~\ref{Prop:Embedding_into_geodesic_ACUG}, we pass on to prove (3) and (4) for $G_\varepsilon(M)$.

\subsection{Estimation of the Nagata dimension of \texorpdfstring{$G_\ep(M)$}{Gepsilon(M)}}

In the next proposition, we provide an estimation of the Nagata dimension of the complete geodesic graph $G(M)$ generated by a metric space $M$ with finite Nagata dimension. Since $G_\ep(M)$ is a subset of $G(M)$, it is clear that this estimation also works for $G_\ep(M)$. Obviously, if $M$ has infinite Nagata dimension, then the same holds for $G(M)$ since $M \subset G(M)$. 

\begin{proposition}
	Assume that $M$ has finite Nagata dimension, and let $n=\dim_N(M)$. Then $\dim_N(G(M)) \leq n+1$.
\end{proposition}

\begin{proof}
	Fix $s>0$. By assumption, there exists $\gamma>0$ and a $3\gamma s$-bounded cover 
	$\mathcal{C}$ of $M$ such that
	\begin{equation}
	\forall A \subset M, \quad \diam(A)<3s \implies 
	\bigl|\{C \in \mathcal{C} : A \cap C \neq \emptyset\}\bigr| \leq n+1.
	\end{equation}
	
	For each $C \in \mathcal{C}$, define
	\begin{equation}
	C_s := [C]_s = \{x \in G(M) : d(x,C)\leq s\}.
	\end{equation}
	Then $\diam(C_s) < (3\gamma+2)s$.  
	By construction, the complement $G(M)\setminus \bigcup_{C \in \mathcal{C}} C_s$ consists 
	of disjoint intervals which are $s$-separated from $M$ and $2s$-separated from each other. 
	Each of these intervals has Nagata dimension $1$; hence they can be covered by intervals 
	$I$ contained in $G(M)\setminus \bigcup_{C\in\mathcal{C}}C_s$ satisfying 
    %\reversemarginpar\mnote{\mch{Hola! Could you explain this inequality on $\diam (I)?$ (phew)} \cch{We are asking the each interval in $G(M)\setminus \bigcup_{C \in \mathcal{C}} C_s$ can be covered by (consecutive) intervals $I$ so that $s \leq \diam(I) < (3\gamma+2)s$. This is us who impose that constraint. Possible because intervals have Nagata dimension 1.} \mch{{Thank you Colin! I was confused between disjoint intervals forming the complement set and their covers $I$}}} 
	%\normalmarginpar
	\begin{equation}
	s \leq \diam(I) < (3\gamma+2)s,
	\end{equation}
	with $s$-multiplicity at most $2$ (here we implicitly assume, as we may, that $1=\gamma_1(J) < 3\gamma+2$ for any interval $J$ in $\R$).  
	Let $\mathcal{C}_G$ be the cover of $G(M)$ consisting of all such intervals $I$, together 
	with the sets $C_s$ for $C \in \mathcal{C}$.
	
	%\mnote{\mch{How could it intersect with two members of $\mathcal{C}_G$? Suppose that $A$ is contained in $I$, which is one of the disjoint intervals that make up the set $G(M) \setminus \cup_{C\in\mathcal{C}} C_s$. Since $A$ does not intersect any $C_s$, if $A$ intersects another element from $C_G$, then that element would have to be another interval $J$ which forms the set $G(M) \setminus \cup_{C\in\mathcal{C}} C_s$. However, in this case, it follows that $\diam (A)\geq 2s$ because $I$ and $J$ are $2s$-separated by construction. This contradicts $\diam(A)<s$. Am I missing something?} \cch{I am not sure that you understood properly the definition of the intervals $I$ above, that is why there might be a misunderstanding there. The intervals $I$ are not the full intervals that make up the set $G(M) \setminus \cup_{C\in\mathcal{C}} C_s$, they are the one that covers these full intervals that make up $G(M) \setminus \cup_{C\in\mathcal{C}} C_s$. In particular, they are not at all $2s$ separated as you claim, they are consecutive. (Imagine covering $[0,1]$ with intervals of length $s$).}} Thank you Colin!
	
	Now let $A \subset G(M)$ with $\diam(A)<s$. 
	Suppose $A$ does not intersect any $C_s$.  
	Then $A$ is contained in one of the disjoint intervals described above, and thus it 
	intersects at most two members of $\mathcal{C}_G$. 
	Suppose now that $A$ intersects some $C_s$.  
	Then
	\begin{equation}
	[A]_s \cap C 
	= \{x \in G(M): d(x,A)\leq s\} \cap C
	= \{x \in M: d(x,A)\leq s\} \cap C \neq \emptyset.
	\end{equation}
	In particular, $[A]_s \cap M$ intersects $C$.  
	Since $\diam([A]_s \cap M) < 3s$, the assumption on $\mathcal{C}$ implies that there 
	are at most $n+1$ such sets $C \in \mathcal{C}$.  
	Thus $A$ intersects at most $n+1$ sets of the form $C_s$. Finally, note that the additional intervals $I \in \mathcal{C}_G$ satisfy 
	$\diam(I) \geq s$ and $d(I,M)\geq s$, so $A$ can intersect at most one such interval.  %\reversemarginpar\mnote{\cch{Well, there this what I say right? Since $\diam(A)<s$ and $A$ touches some set $C_s$ (that it is not contained in one of the intervals that make up $G(M) \setminus \cup_{C\in\mathcal{C}} C_s$), we deduce that $A$ intersect at most one interval $I$ because if touches two we get a contradiction. Maybe with the right understanding of what the $I$'s are you will get it!?} \mch{Thank you Colin! However, I still find it difficult to get a contradiction. Maybe the picture in my head is not so good. Suppose we have two intervals $I, J \in \mathcal{C}_G$ such that $A$ intersects both $I$ and $J$, say $a \in A\cap I$ and $b \in A \cap J$. We also have points $x \in A\cap C_s$ and $y \in C$ such that... since $A$ touches some $C_s$. Maybe I am missing something very important. :-( Could you explain?} \cch{Ok I'll try :D. Say that $A$ touches two intervals $I$ and $J$. There are two cases. Case 1. If $I$ and $J$ belong to the same covering of an interval that make up $G(M) \setminus \cup_{C\in\mathcal{C}} C_s$, then it is clear because of the graph distance that we put on $G(M)$. I mean, the best case scenario is $I$ and $J$ are consecutive intervals, say $I$ comes right before $J$, and then $\diam(I)\geq s$ plus $\diam(A)<s$ forbids $A$ to touch $J$. Case 2. $I$ and $J$ come from two different interval of $G(M) \setminus \cup_{C\in\mathcal{C}} C_s$. In this case we have that $d(I,J)\geq 2s$, and so it is clearly impossible to touch $I$ and $J$ while having diameter less than $s$.}}
	
	%\mch{Similarly, I think $A$ can intersect at most one such intervals because if there are two such intervals, say $I$ and $J$, then $d(I,J)>2s$ by construction. This implies that it is not possible to have both intersections $A\cap I$ and $A \cap J$ nonempty because $\diam (A)<s$. Am I missing something?}
	%\includegraphics[height=17em]{Cs.jpeg}
	
	In total, $A$ intersects at most $(n+1)+1 = n+2$ sets from $\mathcal{C}_G$.  
	This proves the claim.
\end{proof}

We have shown now (1), (2) and (3) of Proposition \ref{Prop:Embedding_into_geodesic_ACUG}, which is enough to remove the quasi-convex assumption in Theorem \ref{Th:Superreflexive_ACUG_implies_V*} for spaces of finite Nagata dimension. We finish this section by dealing with the purely 1-unrectifiable case, thus completing the proof of Proposition \ref{Prop:Embedding_into_geodesic_ACUG}.

\subsection{The purely 1-unrectifiable case}

We begin with a simple observation.  
Suppose that $M$ is compact and purely $1$-unrectifiable. Then it follows from \cite[Theorem~3.2]{AGPP} that the set of uniformly locally flat Lipschitz functions $\lip_0(M)$ is weak$^*$-dense in $\Lip_0(M)$. Consequently, since $M$ is compact, for every $f \in \Lip(M)$ with $\|f\|_L \le 1$, there exists $g \in \lip(M)$ such that $\|g\|_L \le 2$ (or, more precisely, $\|g\|_L \le 1+\varepsilon$ for every $\varepsilon>0$) and $\|f-g\|_\infty \le \varepsilon$.  

Moreover, since $M$ is purely $1$-unrectifiable, it contains no rectifiable curve joining two distinct points $x,y \in M$. Therefore, for $x\ne y$, the quantification over $\gamma$ in Definition~\ref{def:UpperGradient} is over the empty set, and the upper gradient condition in Definition~\ref{Def:App_cont_upp_grad_structure} is vacuously true.  
Hence, taking $H=\lip(M)$ (or even $H=\Lip(M)$) and $\Phi_H \colon f \in H \mapsto 0 \in \{0\}$, we obtain that $M$ admits an ACUG $\{0\}$-structure.

We now turn to the proof of the last remaining assertion of Proposition~\ref{Prop:Embedding_into_geodesic_ACUG}. 
%\mnote{\mch{To be consistent with the notation, `one-dimensional ACUG structre' $\to$ `ACUG $\mathbb{R}$-structure'. :-) } \cch{\checkmark}}

\begin{proposition}
	If a metric space $M$ is compact and purely $1$-unrectifiable, then $G_\varepsilon(M)$ admits an ACUG $\mathbb{R}$-structure.
\end{proposition}

Before proceeding, we clarify some notation.  
For technical convenience, fix an arbitrary well-order $\le$ on $M$, the specific choice being irrelevant.  
Recall that, by construction, each segment $[x,y] \subset G_\varepsilon(M)$ can be naturally identified with the real interval $[0,d(x,y)] \subset \mathbb{R}$.  
Through this identification, any function $f : [x,y] \to \mathbb{R}$ becomes a real function defined on $[0,d(x,y)]$, allowing us to freely employ standard tools from real analysis.  
For instance, by abuse of notation, if $x \le y$ and $p \in [x,y]$, we write $\int_x^p f(t)\,dt$ for the usual Lebesgue integral of $f$ viewed as a real function on $[0,d(x,p)]$.  
We adopt this convention throughout the proof.

Although the proof below may appear somewhat technical, its main idea is simple: we approximate a $1$-Lipschitz map $f$ on $G_\varepsilon(M)$ by first approximating it on $M$ by a little Lipschitz function $g$, and then extending $g$ over each segment $[x,y]\subset G_\varepsilon(M)$ to a $\mathcal{C}^1$ function that still approximates $f$.

\begin{proof}
	Fix $\eta>0$.  
	We will construct a linear subspace $H \subset \Lip_0(G_\varepsilon(M))$ and a linear map $\Phi_H \colon H \to \mathcal{C}_b(G_\varepsilon(M),\mathbb{R})$ satisfying:
	\begin{enumerate}[(i)]
		\item For every $f \in B_{\Lip_0(G_\varepsilon(M))}$, there exists $g \in H$ such that 
		\begin{equation}
		\|g\|_L \le 8, \qquad
		\|\Phi_H g\|_\infty \le 8, \qquad
		\|f-g\|_\infty \le 4\eta.
		\end{equation}
		\item For every $g \in H$, the function 
		\begin{equation}
		|\cdot|\circ(\Phi_H g) \colon G_\varepsilon(M) \longrightarrow [0,+\infty)
		\end{equation}
		is an upper gradient of $g$.
	\end{enumerate}
	\smallskip
	
	We define $H$ as the space of all $g \in \Lip_0(G_\varepsilon(M))$ satisfying the following conditions:
    \begin{enumerate}
		\item $g\restricted_{M} \in \lip_0(M)$,
		\item $g\restricted_{(x,y)} \in \mathcal C^1((x,y))$ for every interval $(x,y)\subset G_\ep(M)\setminus M$, with $x,y \in M$, and moreover $g\restricted_{(x,y)}'(p) \to 0$ whenever $p \to x$ or $p \to y$ within $(x,y)$,
		\item the map $\Phi_H(g) : G_\ep(M) \to \R$ defined by 
		\begin{equation} \Phi_H(g)(p):= \left\{\begin{array}{cl}
			0 & \text{if } p \in M,\\[4pt]
			g\restricted_{(x,y)}'(p) & \text{if } p \in (x,y) \subset G_\ep(M) \setminus M,
		\end{array} \right. \end{equation}
		belongs to $\mathcal{C}_b(G_\varepsilon(M), \mathbb{R})$.
	\end{enumerate}
	Naturally, we thus consider the linear map $\Phi_H \colon g \in H \mapsto  \Phi_H(g) \in \mathcal{C}_b(G_\varepsilon(M), \mathbb{R})$.
	\smallskip
	
	By the fundamental theorem of calculus, and the comment before this proof,  condition~(ii) is immediate.  
	We now verify condition~(i). Fix $f\in B_{\Lip_0(G_\varepsilon(M))}$.  
	Since $M$ is compact and purely $1$-unrectifiable, there exists $g\in 2B_{\lip_0(M)}$ such that $\|f\restricted_M - g\|_\infty \le \eta$.  
	We extend $g$ to $\tilde g\in H$ in the following way.
	
	\medskip
	\noindent
	\textbf{Case 1.}  
	Let $x,y\in M$ with $(x,y)\subset G_\varepsilon(M)\setminus M$ and $d(x,y)\le \eta$, and assume $x\le y$.  
	Define $h_{g}^{[x,y]}:[x,y]\to\mathbb{R}$ by
	\begin{itemize}
		\item $h_{g}^{[x,y]}(x)=h_{g}^{[x,y]}(y):=0$, and $h_{g}^{[x,y]}(m):=2d(x,y)^{-1}(g(y)-g(x))$, where $m$ is the midpoint of $[x,y]$;
		\item $h_{g}^{[x,y]}$ is affine on each subsegment $[x,m]$ and $[m,y]$.
	\end{itemize}
	Then $g(y)-g(x)=\int_x^y h_{g}^{[x,y]}(t)\,dt$.  
	Extend $g$ to $[x,y]$ by
	\begin{equation}
	\tilde g(p) := g(x)+\int_x^p h_{g}^{[x,y]}(t)\,dt, \quad p\in [x,y].
	\end{equation}
	Thus $\tilde g\restricted_{(x,y)}$ satisfies (2) above, and
	\begin{equation}
    \label{eq:Lip_ctt_of_g_tilde_in_short_intervals}
	\|\tilde g\restricted_{[x,y]}\|_L \le \|h_g^{[x,y]}\|_\infty = 2d(x,y)^{-1}|g(y)-g(x)| \le 4.
	\end{equation}
	(The increase in the Lipschitz constant by a factor $2$ could be reduced to a factor $1+\varepsilon$ for any chosen $\varepsilon >0$, by using alternative techniques. For instance, convolution with a suitably chosen smooth symmetric mollifier. Since we are not concerned with optimizing constants, we opt for this quicker and simpler approach.)

    Note as well that, by construction,
    \begin{equation}
    \label{eq:Lip_ctt_of_h_g_[x,y]}
        \Lip(\tilde{g}\restricted_{[x,y]}')=\Lip(h_g^{[x,y]})=4\frac{|\langle g,m_{xy}\rangle|}{d(x,y)}.
    \end{equation}
	\medskip
	\noindent
	\textbf{Case 2.}  
	Let $x,y\in M$ with $[x,y]\subset G_\varepsilon(M)$ and $d(x,y)>\eta$, again with $x\le y$.  
	Pick points $x<p_1<\cdots<p_n<y$ such that
	\begin{equation}
	d(x,p_1)=\frac{\eta}{4},\quad d(p_n,y)=\frac{\eta}{4},\quad
	\frac{\eta}{4}\le d(p_i,p_{i+1})\le\frac{\eta}{2}\text{ for all }i.
	\end{equation}
	Then $n\ge 3$.  
	Define
	\begin{equation}
	\tilde g(p_1)=g(x),\qquad \tilde g(p_n)=g(y),\qquad \tilde g(p_i)=f(p_i)\ \text{for }2\le i\le n-1,
	\end{equation}
	and make $\tilde g$ constant on $[x,p_1]$ and $[p_n,y]$.  
	It is straightforward that $\|\tilde g\|_L \le 4$, since for instance
	\begin{align}
		|\tilde g(p_1)-\tilde g(p_2)|
		&=|g(x)-f(p_2)|
		\le |g(x)-f(x)|+|f(x)-f(p_2)|\\
		&\le \eta + d(x,p_1)+d(p_1,p_2)\le \tfrac{3}{4}\eta + d(p_1,p_2)
		\le 4d(p_1,p_2).\nonumber
	\end{align}
	Finally, extend $\tilde g$ smoothly on each subsegment $[p_i,p_{i+1}]$ as in Case~1, so that $\tilde g\in \mathcal{C}^1((x,y))$ and $\|\tilde g\|_L\le 8$.
	
	\medskip
	We thus obtain $\tilde g\in 8B_{\Lip_0(G_\varepsilon(M))}$ satisfying (1) and (2). We now show that $\Phi_H(\tilde g)\in \mathcal{C}_b(G_\varepsilon(M),\mathbb{R})$. Continuity is clear on $G_\varepsilon(M)\setminus M$. Let $p\in M$ and fix $\epsilon>0$ arbitrary. Since $\tilde g\restricted_M\in \lip_0(M)$, there exists $0<\delta_1<\min\{\epsilon/4,\eta/4\}$ such that $|\langle g,m_{uv}\rangle|<\varepsilon/2$ whenever $d(u,v)<\delta_1$. 
    Let $\delta_2:=\min \{ \delta_1, \epsilon\delta_1/8\}$. For $q\in G_\varepsilon(M)$ with $d(p,q)<\delta_2$, we distinguish cases:
	\begin{itemize}
		\item If $q\in M$, then $|\Phi_H (\tilde{g})(p)-\Phi_H (\tilde{g})(q)|=0$.
		\item If $q\in (x,y)\subset G_\varepsilon(M)\setminus M$ with $d(x,y)<\delta_1$, then by the definition of $\Phi_H(\tilde{g})$ and the first inequality of equation \eqref{eq:Lip_ctt_of_g_tilde_in_short_intervals} 
		\begin{equation}|\Phi_H (\tilde{g})(q)|\le 2|\langle g,m_{xy}\rangle|\le \epsilon.\end{equation}
		\item If 
        %\mnote{\mch{I think there was a slight gap in the previous estimate.. I have rewritten it by introducing a smaller parameter $\delta_2$} \ach{I think you're right. I found this part a bit hard to follow so I added equation \eqref{eq:Lip_ctt_of_h_g_[x,y]} to make a bit clearer hopefully. But we should let Colin look at this to confirm everything is correct, since he wrote the proof :D} \cch{Thanks, there was indeed a flaw in the computations. Your additions are nice. And sorry for this messy proof, I agree that it could probably be written in a better way...} \mch{Génial!}} 
        $q\in (x,y)\subset G_\varepsilon(M)\setminus M$ with $\delta_1\le d(x,y)\le \eta$, then, again by definition of $\Phi_H(\tilde{g})$ and by equation \eqref{eq:Lip_ctt_of_h_g_[x,y]}
		\begin{align}
        |\Phi_H (\tilde{g})(q)| &\le \min \{ d(q,x), d(q,y)\} 
        \frac{4|\langle g,m_{xy}\rangle| }{d(x,y)} < \frac{8}{\delta_1} \delta_2 \le \epsilon.
        \end{align}
        
		\item If $q\in (x,y)\subset G_\varepsilon(M)\setminus M$ with $d(x,y)>\eta$, then $|\Phi_H (\tilde{g})(p)-\Phi_H (\tilde{g})(q)|=0$.
	\end{itemize}
	Thus $\Phi_H (\tilde{g})$ is continuous at every point of $G_\ep(M)$.
	
	\medskip
	Finally, $\tilde g$ approximates $f$ uniformly:
	\begin{itemize}
		\item If $p\in M$, then $|\tilde g(p)-f(p)|=|g(p)-f(p)|\le \eta$.
		\item If $p\in (x,y)\subset G_\varepsilon(M)\setminus M$ with $d(x,y)\le \eta$, then $d(x,p)\le \eta/2$ or $d(p,y)\le \eta/2$, and in either case
		\begin{equation}
		|\tilde g(p)-f(p)|
		\le 4\cdot \tfrac{\eta}{2} + \eta + \tfrac{\eta}{2} < 4\eta.
		\end{equation}
		\item If $p\in (x,y)\subset G_\varepsilon(M)\setminus M$ with $d(x,y)>\eta$, then either some $p_i$ satisfies $d(p,p_i)\le \eta/4$, in which case
		\begin{equation}
		|\tilde g(p)-f(p)|\le \tfrac{\eta}{4} + 0 + \tfrac{\eta}{4} = \tfrac{\eta}{2},
		\end{equation}
		or $d(x,p)\le \eta/4$ (the other case being symmetric), yielding
		\begin{equation}
		|\tilde g(p)-f(p)|\le 0+\eta+\tfrac{\eta}{4}=\tfrac{5}{4}\eta.
		\end{equation}
	\end{itemize}
	This completes the proof.
\end{proof}

\section{On property (V*) in Lipschitz-free spaces}

\subsection{Overview of Lipschitz-free spaces with property (V*)} \label{subsection:V*}

We briefly review the main classes of metric spaces whose associated Lipschitz-free spaces are known to satisfy Pełczyński's property~(V*), in light of the results from the previous sections. We refer the reader to the recent preprint~\cite{APQ24} for a comprehensive exposition on this topic. All assertions below follow by a compact reduction argument plus a statement for compact metric spaces. Our only genuinely new contribution here is assertion~$(ii)$. However, it is worth noting that all known results for property~(V*) in Lipschitz-free spaces over compact metric spaces can be deduced by a combination of Proposition \ref{Prop:Embedding_into_geodesic_ACUG} and Theorem \ref{Th:Superreflexive_ACUG_implies_V*}. 

\begin{proposition}
	\label{prop:V*}
	The Lipschitz-free space $\F(M)$ has property~\emph{(V*)} in each of the following cases:
	\begin{enumerate}[$(i)$]
		\item $M$ is locally compact and purely $1$-unrectifiable (in particular, any discrete metric space);
		
		\item $M$ has finite Nagata dimension (this includes, in particular, all finite-dimensional Banach spaces, doubling metric spaces, Carnot groups, and Ostrovskii's Diamond $D_\infty$);
		
		\item $M$ is a compact subset of a superreflexive Banach space;
		
		\item $M=\ell_p$ for some $1<p<\infty$;
		
		\item $M=(B_{\ell_p},\|\cdot\|_p^\alpha)$ with $\alpha\in(0,1)$.
	\end{enumerate}
\end{proposition}

For background on Carnot groups, see Section~3 of~\cite{APQ24}. Additional details on $D_\infty$ are given in Section~\ref{subsection:OstrovskiiDiamond}.
Before proving Proposition~\ref{prop:V*}, we collect two tools from~\cite{APQ24}. 
First, \cite[Theorem~A]{APQ24} shows that property~(V*) is \emph{locally determined} for Lipschitz-free spaces: $\F(M)$ has~(V*) if every point $x\in M$ admits a neighborhood $U$ with $\F(U)$ having~(V*). 
It is open whether~(V*) is \emph{compactly determined} (see \cite[Question~2]{APQ24}), meaning that $\F(M)$ would have~(V*) whenever $\F(K)$ has~(V*) for every compact $K\subset M$. 
Nevertheless, \cite[Proposition~2.11]{APQ24} provides natural sufficient conditions ensuring compact determination, notably:
\begin{enumerate}[(a)]
	\item every compact subset of $M$ is contained in a compact Lipschitz retract of~$M$;
	\item for every compact $K\subset M$ there exists a linear extension operator $E:\Lip_0(K)\to\Lip_0(M)$ with $E(f)\restricted_K=f$ for all $f\in\Lip_0(K)$.
\end{enumerate}
Note also that if $M$ is locally compact, then local determination implies compact determination.

\begin{proof}
	Property~(V*) is stable under isomorphism and under passing to subspaces.
	
	\emph{Proof of $(i)$.} If $M$ is compact and purely $1$-unrectifiable, then Theorem~\ref{Th:Superreflexive_ACUG_implies_V*} together with items~(1), (2), and~(4) of Proposition~\ref{Prop:Embedding_into_geodesic_ACUG} implies that $\F(M)$ has~(V*). 
	Alternatively, \cite{APQ24} reached the same conclusion via a different route: in this setting $\F(M)$ is an $L$-embedded Banach space and $L$-embedded spaces have~(V*). 
	We also note that a straightforward adaptation of \cite[Proposition~3.5]{curveflat} yields the same result. 
	The general locally compact purely $1$-unrectifiable case then follows from local (hence compact) determination.
	
	\medskip
	
	\emph{Proof of $(ii)$.} Suppose $\dim_N(M)<\infty$. As in the proof of Theorem~\ref{thm:FNDimpliesACUG}, the extension theorem of Naor--Silberman \cite[Corollary~5.2]{NaorSilberman} ensures condition~(b) above. Hence~(V*) is compactly determined and we may reduce to the compact case. The conclusion then follows from Theorem~\ref{Th:Superreflexive_ACUG_implies_V*}, Theorem~\ref{thm:FNDimpliesACUG}, and items~(1), (2), (3) of Proposition~\ref{Prop:Embedding_into_geodesic_ACUG}.
	
	\medskip
	
	\emph{Proof of $(iii)$.} This is the main theorem of~\cite{KP18}. One may also deduce it from Theorem~\ref{Th:Superreflexive_ACUG_implies_V*} as follows.
    First, a superreflexive Banach space $X$ carries an ACUG $X$-structure (see
    \cite[Corollary~8]{HajekJohanis} and the discussion around \cite[Theorem~4]{KP18}).
    Let $M\subset X$ be compact, and set $C$ its closed convex hull.
    Then $C$ is compact and, with the metric induced by $X$, it is a geodesic metric space.
    Applying Theorem~\ref{Th:Superreflexive_ACUG_implies_V*} to $C$ yields property~(V*) for $\F(C)$,
    and since (V*) is stable under passing to subspaces, it follows for $\F(M)$ as well.

	\medskip
	
	\emph{Proof of $(iv)$.} This is stated in \cite[Remark~2.12]{APQ24}. It follows from~$(iii)$ together with the fact that condition~(a) above holds in $\ell_p$ for $1\le p<\infty$. 
	For $p=2$, this is a consequence of the nearest-point projection onto closed convex sets in Hilbert spaces. 
	For $p\neq 2$, this fact was communicated to the authors of~\cite{APQ24} by W.~B.~Johnson.
	
	\medskip
	
	\emph{Proof of $(v)$.} Let $M=(B_{\ell_p},\|\cdot\|_p^\alpha)$ with $\alpha\in(0,1)$. 
	Given a compact $K\subset B_{\ell_p}$, choose a compact $L\subset \ell_p$ with $K\subset L$ such that $L$ is a Lipschitz retract of $\ell_p$. Let $r:\ell_p\to L$ be a Lipschitz retraction. Then the restriction $r\restricted_{B_{\ell_p}}:(B_{\ell_p},\|\cdot\|_p^\alpha)\to (L,\|\cdot\|_p^\alpha)$ is still a Lipschitz retract. Thus condition~(a) holds, so~(V*) is compactly determined for $M$. Since $(B_{\ell_p},\|\cdot\|_p^\alpha)$ is purely $1$-unrectifiable, the conclusion follows from~$(i)$.
\end{proof}

\begin{remark}
	By the same method, one also obtains that $\F\big(\ell_p,\omega\circ\|\cdot\|_{\ell_p}\big)$ has property~(V*) for every nontrivial gauge $\omega$ in the sense of Kalton \cite[Section~3]{Kalton04}.
	
	Moreover, $\F\big(L_1(\mu),\|\cdot\|^{1/2}\big)$ has property~(V*). Indeed, by Schoenberg's theorem~\cite{Schoenberg}, the metric space $\big(L_1(\mu),\|\cdot\|^{1/2}\big)$ embeds isometrically into $\ell_2$ (see also \cite[Corollary~3.1]{NaorICM}). The claim then follows from Proposition~\ref{prop:V*}\,(iv).
\end{remark}

\subsection{Ostrovskii's 2-branching infinite diamond}
\label{subsection:OstrovskiiDiamond}

We now recall the construction of Ostrovskii's \emph{2-branching infinite diamond} $D_\infty$, a classical example in the nonlinear geometry of Banach spaces. We single out this metric space because its well known non-embeddability properties (recalled below) make it difficult to show property~(V*) for $\mathcal{F}(D_\infty)$ with previously existing techniques.

Let us define
\begin{equation}
D_1 = \{t_1,b_1\} \cup \{x_i : 0 \le i \le 1\},
\end{equation}
and consider the complete bipartite graph on these vertices: each vertex in $\{t_1,b_1\}$ is connected to each vertex in $\{x_0,x_1\}$.  
We denote the edge $\{t_1,x_i\}$ by $(i,+)$ and $\{b_1,x_i\}$ by $(i,-)$.  
Define a metric $d_1$ on $D_1$ as one half of the corresponding shortest-path distance, so that $d_1(t_1,b_1)=1$.  

We now move to the recursive construction. Assume $D_n$ has already been defined as a weighted graph with the metric $d_n$.  
The next level $D_{n+1}$ is obtained by replacing each edge of $D_1$ by a rescaled copy of $D_n$ of diameter $\frac12$; that is, each edge of $D_1$ is substituted by a copy of $D_n$ scaled by the factor $\frac12$.  
We denote these copies by $D_n^{(i,\pm)}$, corresponding to the edges $(i,\pm)$ of $D_1$.  
The four “corner” vertices inherited from $D_1$ inside $D_{n+1}$ are denoted
\begin{equation}
t_{n+1}, \ b_{n+1}, \ x^0_{n+1}, \ x^1_{n+1}.
\end{equation}
It is well known (and straightforward to verify) that $D_n$ embeds canonically and isometrically into $D_{n+1}$ in such a way that
\begin{equation}
t_n = t_{n+1}, \qquad b_n = b_{n+1}, \qquad x_n^i = x_{n+1}^i \quad (i=0,1).
\end{equation}
Hence $(D_n)_n$ forms a nested sequence of metric graphs, and we define $D_\infty$ to be the completion of their union.

The space $D_\infty$ plays an important role in the nonlinear geometry of Banach spaces.  
It was shown by Ostrovskii~\cite{Ostrovskii} that $D_\infty$ does not admit a bi-Lipschitz embedding into any Banach space with the Radon--Nikodým property.  
Moreover, by a theorem of Johnson and Schechtman~\cite{Johnson-Schechtman}, a Banach space $X$ is nonsuperreflexive if and only if it admits bi-Lipschitz embeddings with uniformly bounded distortions of diamonds $(D_n)_{n\ge1}$ of all sizes.

\begin{proposition}
	\label{p:NagataOfDiamond}
	The Nagata dimension of $D_\infty$ is~$1$.
\end{proposition}

To prove this, we use the following standard criterion.

\begin{lemma}
	\label{l:SufficientForNagata}
	Let $M$ be a metric space.  
	Assume that for every $n\in\mathbb{N}$ there exists a cover $\mathcal{C}_n$ of $M$ such that:
	\begin{itemize}
		\item for every $C \in \mathcal{C}_n$, one has $\diam(C) \le 2^{-n}$;
		\item every subset $S \subset M$ with $\diam(S)<2^{-n}$ intersects at most $d+1$ members of $\mathcal{C}_n$.
	\end{itemize}
	Then $\dim_N(M)\le d$.
\end{lemma}

\begin{proof}
	Let $\gamma=2$.  
	For any $s>0$, choose $n\in\mathbb{N}$ such that $2^{-n-1} \le \gamma s < 2^{-n}$.  
	Then $s < 2^{-n-1}$, and by hypothesis any set $S$ with $\diam(S)\le s$ intersects at most $d+1$ elements of $\mathcal{C}_{n+1}$.  
	This is precisely the definition of $\dim_N(M)\le d$.
\end{proof}

\begin{proof}[Proof of Proposition~\ref{p:NagataOfDiamond}]
	We verify that $D_\infty$ satisfies the assumptions of Lemma~\ref{l:SufficientForNagata} with $d=1$.
	\smallskip
	
	Denote by $\{t,b,x_0,x_1\}=D_1\subset D_\infty$.  
	Observe that $D_\infty$ can be decomposed as the almost disjoint union of four scaled-down copies of itself, each of diameter $\frac12$, denoted $D_\infty^{(i,\pm)}$ for $i=0,1$, with the following identifications:
	\begin{equation}
	t = t^{(0,+)} = t^{(1,+)}, \quad
	b = b^{(0,-)} = b^{(1,-)}, \quad
	x_i = b^{(i,+)} = t^{(i,-)} \text{ for } i=0,1.
	\end{equation}
	Moreover, any two distinct copies $D_\infty^{(i,+)}$ and $D_\infty^{(j,-)}$ with $i+j=1$ are at mutual distance $\frac12$.
	\medskip
	
	We construct, by induction on $n$, a cover $\mathcal{C}_n$ of $D_\infty$ satisfying:
	\begin{enumerate}[(1)]
		\item $\diam(C)\le 2^{-n}$ for all $C\in\mathcal{C}_n$;
		\item $\mathcal{C}_n$ consists of two distinguished balls $\{B(t,2^{-n-1}),B(b,2^{-n-1})\}$ and a family $\mathcal{B}_n$ such that:
		\begin{itemize}
			\item each $C\in\mathcal{B}_n$ satisfies $\dist(t,C),\dist(b,C)\ge 2^{-n-1}$;
			\item any set $S\subset D_\infty$ with $\diam(S)<2^{-n}$ meets at most two members of $\mathcal{C}_n$.
		\end{itemize}
	\end{enumerate}

	\textit{Base step ($n=1$).}
	Let $\mathcal{C}_1 = \{B(z,\tfrac14) : z\in D_1\}$.
	Then $\mathcal{B}_1 = \{B(x_i,\tfrac14) : i=0,1\}$.
	Clearly, $\diam(C)\le \tfrac12$ for all $C\in\mathcal{C}_1$, and non-intersecting members of $\mathcal{C}_1$ are at distance~$\tfrac12$, since $d(t,b)=1$ and $d(x_0,x_1)=1$.  
	Thus, any subset $S\subset D_\infty$ with $\diam(S)<\tfrac12$ intersects at most two members of~$\mathcal{C}_1$.
	\medskip
	
	\textit{Inductive step.}
	Assume $\mathcal{C}_n$ satisfies the above properties.  
	For each copy $D_\infty^{(i,\pm)}$ (where $i=0,1$), we consider the corresponding scaled cover $\mathcal{C}_n^{(i,\pm)}$.  
	We then define
	\begin{equation}
	\mathcal{C}_{n+1} = \{B(z,2^{-n-2}) : z\in D_1\} \ \cup\ 
	\bigcup_{\substack{i=0,1\\ j=\pm}} \mathcal{B}_n^{(i,j)}.
	\end{equation}
	Equivalently,
	\begin{equation}
	\mathcal{B}_{n+1} = \{B(x_i,2^{-n-2}) : i=0,1\} \cup 
	\bigcup_{\substack{i=0,1\\ j=\pm}} \mathcal{B}_n^{(i,j)}.
	\end{equation}
	This clearly covers $D_\infty$, since
	\begin{equation}
	\begin{aligned}
		B(x_i,2^{-n-2}) &= B(b^{(i,+)},2^{-n-2}) \cup B(t^{(i,-)},2^{-n-2}) &&(i=0,1),\\
		B(t,2^{-n-2}) &= B(t^{(0,+)},2^{-n-2}) \cup B(t^{(1,+)},2^{-n-2}),\\
		B(b,2^{-n-2}) &= B(b^{(0,-)},2^{-n-2}) \cup B(b^{(1,-)},2^{-n-2}).
	\end{aligned}
	\end{equation}
	
	Property (1) and the first item in (2) are immediate.  
	So let us check the intersection property.
	Let $S\subset D_\infty$ satisfy $\diam(S)<2^{-n-1}$.  
	From the structure above, $S$ can intersect at most two of the scaled copies $D_\infty^{(i,\pm)}$.  If $S$ lies entirely within one copy, the inductive hypothesis applies.  
	Otherwise $S$ intersects both, and let us assume these are the two meeting at $x_1$, denoted $D^+$ and $D^-$. 
	
	We will prove that $S$ intersects $B(x_1,2^{-n-2})$ and at most one additional element of $\mathcal{C}_{n+1}$.
	Suppose that $S$ meets distinct sets $A,B\in\mathcal{C}_{n+1}\setminus B(x_1,2^{-n-2})$. If say $A\in D^+$ and $B\in D^-$,  
	by the graph structure and the first item in (2), this would imply $\diam(S)\ge 2^{-n-1}$, a contradiction.  
	So the sets $A,B$ must lie in the same copy, say $D^+$. Consider the canonical $1$-Lipschitz retraction $r:D^+\cup D^-\to D^+$.  
	Then $r(S)$ meets three sets in $\mathcal{C}_n^{(1,+)}$, namely $A,B$, and the ball centered at $x_1$, while $\diam(r(S))\le\diam(S)<2^{-n-1}$, again contradicting the inductive assumption.  
	Hence $S$ intersects $B(x_1,2^{-n-2})$ and at most one additional element of $\mathcal{C}_{n+1}$.
	
	This completes the induction, establishing the hypothesis of Lemma~\ref{l:SufficientForNagata} with $d=1$.  
	Therefore $\dim_N(D_\infty)=1$.
\end{proof}

\begin{corollary}
	The Lipschitz-free space $\mathcal{F}(D_\infty)$ has property~\emph{(V*)}.
\end{corollary}

\subsection{WUC series in spaces of Lipschitz functions}

We now begin the preparatory part of the proof of Theorem~\ref{Th:Superreflexive_ACUG_implies_V*}, which asserts that if $M$ is a compact quasiconvex metric space admitting an ACUG superreflexive structure, then $\F(M)$ has property~(V*). 
The purpose of this section is to establish several technical tools that allow us to recognize when a series in $\Lip_0(M)$ is weakly unconditionally Cauchy (WUC).
\smallskip

For simplicity, throughout this section and the next one we shall assume that $M$ is \emph{geodesic} rather than merely quasiconvex. 
This entails no loss of generality thanks to Lemma~\ref{lemma:quasiconvexTOgeo}, together with the stability of property~(V*) under linear isomorphisms.
\smallskip

It is immediate that if $(z_n)_n$ is a bounded sequence in $\Lip_0(M)$ with pairwise disjoint supports, then $\sum_n z_n$ is a WUC series. 
However, in our applications this condition will be far too restrictive; we shall need a more flexible criterion that allows significant overlap between the supports. 
We begin with a useful sufficient condition involving upper gradients.

\begin{lemma}
	\label{l:UpperGradientWUC}
	\label{Lemma:Sufficient_condition_for_WUC}
	Let $M$ be a geodesic metric space, and let $\sum_n z_n$ be a series in $\Lip_0(M)$. 
	Suppose there exists a sequence of nonnegative functions $(f_n)_n$ and $(g_n)_n$ such that $f_n$ is an upper gradient of $z_n$ for each $n\in\N$. 
	If there exists $C>0$ such that
	\begin{equation}
	\sum_{i=1}^n g_i \le C \qquad \text{ and}\qquad \sum_{i=1}^n |g_i-f_i|\leq C\qquad\text{pointwise for all } n\in\N,
	\end{equation}
	then $\sum_n z_n$ is a WUC series in $\Lip_0(M)$.
\end{lemma}

\begin{proof}
	Let $\mu=\sum_{j=1}^k \lambda_j m_{p_j q_j}\in B_{\F(M)}$ be a convex combination of molecules, so that $\|\mu\|=\sum_{j=1}^k|\lambda_j|$. 
	For each $j=1,\dots,k$, fix a geodesic $\gamma_j\colon [a_j,b_j]\to M$ joining $p_j$ and $q_j$. Then
	\begin{align}
		\sum_{i=1}^n |\langle \mu, z_i\rangle|
		&= \sum_{i=1}^n \sum_{j=1}^k \frac{|\lambda_j|}{d(p_j,q_j)}\, |z_i(q_j)-z_i(p_j)| \\
		&\le \sum_{j=1}^k \frac{|\lambda_j|}{d(p_j,q_j)} \int_{a_j}^{b_j} \sum_{i=1}^n f_i(\gamma_j(t))\,dt \nonumber\\
		&\le \sum_{j=1}^k\frac{|\lambda_j|}{d(p_j,q_j)} \int_{a_j}^{b_j}\sum_{i=1}^n\left(g_i(\gamma_j(t))+|g_i-f_i|(\gamma_j(t)) \right)dt
        \nonumber\\
        &\le 2C \sum_{j=1}^k |\lambda_j| \le 2C.\nonumber
	\end{align}
	Hence $\sum_n z_n$ is WUC.
\end{proof}

\medskip
The next result describes the specific type of WUC series that will be used later in the proof of Theorem~\ref{Th:Superreflexive_ACUG_implies_V*}. 
It combines the previous lemma with a multiplicative construction based on controlled upper gradients.

\begin{lemma}
	\label{Lemma:Prod_is_WUC}
	Let $M$ be a geodesic metric space. 
	Let $(z_n)_n$ be $1$-Lipschitz functions in $\Lip_0(M)$, and let $(\Phi_n)_n$ and $(\Psi_n)_n$ be nonnegative Lipschitz functions on $M$ with $\|\Psi_n\|_\infty \le 1$, and such that each $\Phi_n$ is an upper gradient of $z_n$. 
	Define inductively
	\begin{equation}
	\Pi_1 \equiv 1, \qquad \Pi_n = \prod_{j=1}^{n-1} (1 - \Psi_j) \quad \text{for } n\ge2.
	\end{equation}
	If both series $\sum_n \|\Pi_n\|_L\, \|z_n\|_\infty$ and $\sum_n\|\Pi_n\|_L\|\Psi_n-\Phi_n\|_\infty$
	converge, then $\sum_n \Pi_n z_n$ is WUC in $\Lip_0(M)$.
\end{lemma}

\begin{proof}
	Fix $i\in\{1,\dots,n\}$. 
	Recall that $\Lip(\Pi_i,\cdot)$ denotes the local Lipschitz constant of $\Pi_i$. Lemma~1.7 in~\cite{Che99} implies that $f_i:=\Lip(\Pi_i,\cdot)|z_i| + \Pi_i \Phi_i$ is an upper gradient of the product $\Pi_i z_i$. On the other hand, defining $g_i:=\Lip(\Pi_i,\cdot)|z_i| + \Pi_i \Psi_i$ we get
    \begin{equation}
        g_i\leq \|\Pi_i\|_L\|z_i\|_\infty+\Psi_iz_i\qquad\text{and}\qquad\|g_i-f_i\|_\infty\leq \|\Pi_i\|_L\|\Psi_i-\Phi_i\|_\infty
    \end{equation}

    Since the series $\sum_n \|\Pi_n\|_L\, \|z_n\|_\infty$ and $\sum_n \|\Pi_n\|_L\|\Psi_n-\Phi_n\|_\infty$ converge by assumption, in order to apply Lemma~\ref{Lemma:Sufficient_condition_for_WUC} it suffices to verify that there exists $C>0$ such that
    \begin{equation}
	\sum_{i=1}^n \Pi_i \Psi_i \le C \quad \text{for all } n\in\N.
	\end{equation}
	A straightforward computation gives
	\begin{align}
		\sum_{i=1}^n \Pi_i \Psi_i
		&= \Psi_1 + \sum_{i=2}^n \Bigg[ \prod_{j=1}^{i-1} (1 - \Psi_j) \Bigg] \Psi_i \\
		&= \Psi_1 + \sum_{i=2}^n \left( \prod_{j=1}^{i-1} (1 - \Psi_j) - \prod_{j=1}^{i} (1 - \Psi_j) \right) \nonumber\\
		&= 1 - \prod_{j=1}^n (1 - \Psi_j) \le 1.\nonumber \qedhere
	\end{align}
\end{proof}

We conclude this part with a technical estimate that will be used at the end of the proof of Theorem~\ref{Th:Superreflexive_ACUG_implies_V*}. 

\begin{lemma}
	\label{Lemma:Prod_does_not_affect_critical_geodesics}
	Let $M$ be a geodesic metric space, and let $(\varepsilon_j)_{j=1}^k$ be a family of positive numbers. 
	Consider $\mu = \sum_{i=1}^n \lambda_i m_{p_i q_i} \in B_{\F(M)}$ a convex combination of molecules, and Lipschitz maps 
	$f_j \colon M \to [0,1]$ for $j=1,\dots,k-1$ such that
	\begin{equation}
	\sum_{i=1}^n \frac{|\lambda_i|}{d(p_i,q_i)} 
	\int_{a_i}^{b_i} f_j(\gamma_i(t))\, dt < \varepsilon_j
	\quad \text{for all } j=1,\dots,k-1,
	\end{equation}
	where each $\gamma_i\colon [a_i,b_i]\to M$ is a geodesic joining $p_i$ to $q_i$. 
	Then, for every $z\in\Lip_0(M)$,
	\begin{equation}
	\left|\left\langle \left(\prod_{j=1}^{k-1} (1-f_j)\right) z ,\, \mu \right\rangle\right|
	> |\langle z,\mu\rangle|
	- \left( \|z\|_L \sum_{j=1}^{k-1}\varepsilon_j 
	+ \left\|\textstyle\prod_{j=1}^{k-1}(1-f_j)\right\|_L \|z\|_\infty \right).
	\end{equation}
\end{lemma}

\begin{proof}
	Set $\psi = \prod_{j=1}^{k-1} (1-f_j)$. 
	By the triangle inequality,
	\begin{equation}
	|\langle \psi z, \mu \rangle|
	\ge |\langle z, \mu \rangle|
	- |\langle z(1-\psi), \mu \rangle|.
	\end{equation}
	Hence it suffices to estimate the second term and prove that
	\begin{equation}
		\label{eq:Lemma_Product_Gradients}
		|\langle z(1-\psi), \mu \rangle|
		< \|z\|_L \sum_{j=1}^{k-1}\varepsilon_j
		+ \|\psi\|_L \|z\|_\infty.
	\end{equation}
	Since the pointwise Lipschitz constant $\Lip(z(1-\psi),\cdot)$ is an upper gradient of $z(1-\psi)$, and for any Lipschitz maps $f,g$ one has 
	\begin{equation}
	\Lip(fg,\cdot) \le \Lip(f,\cdot)|g| + \Lip(g,\cdot)|f|,
	\end{equation}
	we deduce
	\begin{align}
		|\langle z(1-\psi),\mu\rangle|
		&\le \sum_{i=1}^n \frac{|\lambda_i|}{d(p_i,q_i)}
		\Bigl(
		\int_{a_i}^{b_i} \Lip(z,\gamma_i(t))\, |(1-\psi)(\gamma_i(t))|\,dt \\
		&\qquad \qquad \qquad \qquad \qquad \qquad + \int_{a_i}^{b_i} \Lip(1-\psi,\gamma_i(t))\,|z(\gamma_i(t))|\,dt
		\Bigr)\nonumber\\
		&\le \sum_{i=1}^n \frac{|\lambda_i|}{d(p_i,q_i)}\,
		\|z\|_L \int_{a_i}^{b_i}
		\Bigl(1 - \prod_{j=1}^{k-1}(1-f_j)\Bigr)(\gamma_i(t))\,dt
		+ \|\psi\|_L \|z\|_\infty.\nonumber
	\end{align}
	Using the elementary inequality
	\begin{equation}
	1 - \prod_{j=1}^{k-1} (1-\alpha_j)
	\le \sum_{j=1}^{k-1} \alpha_j
	\qquad (\alpha_j \in [0,1]),
	\end{equation}
	we obtain
	\begin{align}
		|\langle z(1-\psi),\mu\rangle|
		&\le \sum_{i=1}^n \frac{|\lambda_i|}{d(p_i,q_i)}\,
		\|z\|_L \int_{a_i}^{b_i} \sum_{j=1}^{k-1} f_j(\gamma_i(t))\,dt
		+ \|\psi\|_L \|z\|_\infty\\
		&\le \|z\|_L \sum_{j=1}^{k-1} \varepsilon_j
		+ \|\psi\|_L \|z\|_\infty.\nonumber
	\end{align}
	This establishes~\eqref{eq:Lemma_Product_Gradients} and completes the proof.
\end{proof}

\subsection{A further generalization of Bourgain's approach to property~(V*)}

We now move toward the proof of Theorem~\ref{Th:Superreflexive_ACUG_implies_V*}. 
The argument follows, in spirit, Bourgain's original proof for spaces of continuous functions, as revisited and refined in~\cite{KP18}. 
The role of superreflexivity will again be made explicit through the language of uniform convexity, whose key quantitative features we now recall. 
For background on these notions and their dual formulations, we refer the reader to~\cite[Chapter~11]{PisierBook}.

\medskip
\noindent
\textbf{Uniform convexity and smoothness.}
The \emph{modulus of convexity} of a Banach space $X$ is defined by
\begin{equation}
\delta_X(\varepsilon)
= \inf\Bigl\{
1 - \Bigl\| \tfrac{x+y}{2} \Bigr\|
: x,y\in S_X,\ \|x-y\|\ge \varepsilon
\Bigr\}, \qquad \varepsilon\in[0,2].
\end{equation}
The norm of $X$ is \emph{uniformly convex} if $\delta_X(\varepsilon)>0$ for every $\varepsilon>0$. 
A classical theorem of Enflo~\cite{Enflo} asserts that a Banach space is superreflexive if and only if it admits an equivalent uniformly convex norm. 
Equivalently, it may be renormed so that the \emph{modulus of smoothness}
\begin{equation}
\rho_X(t)
= \sup\Bigl\{
\tfrac{1}{2}\bigl(\|x+ty\|+\|x-ty\|\bigr) - 1
: x,y\in S_X
\Bigr\}, \qquad t\ge0,
\end{equation}
satisfies $\rho_X(t)=o(t)$ as $t\to0$.

Pisier's quantitative version of Enflo's theorem~\cite{Pisier} provides a convenient uniform formulation:  
every superreflexive Banach space can be renormed so that
\begin{equation}
\delta_X(\varepsilon)\ge c\,\varepsilon^q
\quad\text{for some }c>0,\ q\ge2,
\end{equation}
in which case $X$ is said to be \emph{$q$-convex}.  
Dually, one can renorm $X$ so that 
\begin{equation}
\rho_X(t)\le C\,t^p
\quad\text{for some }C>0,\ 1<p\le2,
\end{equation}
and then $X$ is said to be \emph{$p$-smooth}.  

\medskip
\noindent
\textbf{Permanence properties.}
We shall use two well-known consequences of these definitions.  
First, superreflexivity is self-dual: $X$ is superreflexive if and only if $X^*$ is superreflexive (see, e.g.,~\cite[Lemma~9.8]{Fabian}).  
Second, superreflexivity is stable under Bochner $L_2$-spaces: if $X$ is superreflexive, then so is $L_2(X)$.  
A quantitative version of this fact, due to Figiel~\cite{Figiel} and Figiel-Pisier~\cite{FigielPisier} (see also~\cite[Thm.~1.e.9]{LinTza}), asserts that there exist constants $a,b>0$ such that for all $\varepsilon\in(0,2]$,
\begin{equation}
a\,\delta_X(b\varepsilon)
\le \delta_{L_2(X)}(\varepsilon)
\le \delta_X(\varepsilon).
\end{equation}

\medskip
\noindent
\textbf{A quantitative combinatorial lemma.}
We next recall a strengthening of James's classical characterization of superreflexivity, obtained in~\cite{KP18}. 

\begin{lemma}[{\cite[Lemma~7 and Remark~8]{KP18}}]
	\label{Lemma:Quantitative_James}
	Let $X$ be a Banach space whose modulus of convexity satisfies $\delta_X(\varepsilon)\ge c\varepsilon^q$ for all $\varepsilon\in(0,2]$, with some $c>0$ and $q\ge2$.  
	Then there exists a constant $\theta_X>0$, depending only on $c$ and $q$, such that for every $n\in\N$ and every $x_1,\dots,x_n\in B_X$, one can find disjoint nonempty sets $A,B\subset\{1,\dots,n\}$ with $\max A<\min B$ satisfying
	\begin{equation}
	\left\|\frac{1}{|A|}\sum_{i\in A}x_i-\frac{1}{|B|}\sum_{j\in B}x_j\right\|
	\le \frac{\theta_X}{(\log_2 n)^{1/q}}.
	\end{equation}
\end{lemma}

\medskip
\noindent
\textbf{A Lipschitz approximation lemma.}
We shall also require the following well-known result, which may be found in~\cite[Lemma~6]{KP18} or equivalently in~\cite[Ch.~7, Lemma~40]{HajekJohanisBook}.

\begin{lemma}
	\label{lemma6}
	Let $M$ be a metric space and $f:M\to\R$ a uniformly continuous function with modulus of continuity
	\begin{equation}
	\omega_f(t)
	= \sup\{\,|f(x)-f(y)| : x,y\in M,\ d(x,y)\le t\,\},
	\qquad t\ge0.
	\end{equation}
	Suppose $\omega:[0,\infty)\to[0,\infty)$ is a subadditive modulus of $f$, i.e.,
	$\omega$ is nondecreasing, continuous at $0$, satisfies $\omega(0)=0$, and $\omega_f\le\omega$, with $\omega(t+u)\le\omega(t)+\omega(u)$ for all $t,u\ge0$.  
	Then, for every $\varepsilon,a>0$ such that $\omega(a)\le\varepsilon$, there exists an $\frac{\varepsilon}{a}$-Lipschitz map $g:M\to\R$ satisfying
	\begin{equation}
	\sup_{x\in M}|f(x)-g(x)|<\varepsilon.
	\end{equation}
\end{lemma}

Note that when $M$ is geodesic (as will be the case in our application), the minimal modulus $\omega_f$ of $f$ is automatically subadditive, and thus Lemma~\ref{lemma6} applies with $\omega=\omega_f$.

\medskip
\noindent
We are now ready to prove the desired theorem.

\begingroup
\renewcommand{\thethm}{\ref{Th:Superreflexive_ACUG_implies_V*}}
\begin{theorem}
	Let $M$ be a compact quasiconvex metric space admitting an approximate continuous upper gradient $X$-structure for some superreflexive Banach space~$X$. 
	Then the Lipschitz-free space $\F(M)$ has property~\emph{(V*)}.
\end{theorem}
\endgroup

\begin{proof}
	Thanks to Lemma~\ref{lemma:quasiconvexTOgeo}, we may assume that $M$ is a geodesic metric space. 
	We keep the notation from Definition~\ref{Def:App_cont_upp_grad_structure}, and in particular the constant $C>0$ associated with the ACUG $X$-structure. 
	Up to renorming $X^*$, we may assume that it is $q$-convex, i.e. $\delta_{X^*}(\varepsilon)\geq c\varepsilon^q$ for some $c>0$ and $q\in[2,\infty)$.  
	
	Let $\Gamma\subset \mathcal{F}(M)$ be a bounded non-relatively weakly compact subset. 
	We shall prove that $\Gamma$ is not a $V^*$-set. 
	Without loss of generality, assume that $\Gamma\subset B_{\F(M)}$. 
	Then there exist $\xi>0$, and sequences $(\mu_n)_n\subset \Gamma$ and $(f_n)_n\subset \frac{1}{C} B_{\Lip(M)}$ such that for all $k,n\in\mathbb{N}$:
	\begin{itemize}
		\item $\langle \mu_n,f_k\rangle > 3\xi$ if $k\leq n$;
		\item $|\langle \mu_n,f_k\rangle|<\xi$ if $k>n$.
	\end{itemize}
	(See the implication $(iii)\Rightarrow(i)$ in Theorem~4.47 of~\cite{Fab+01}.)
	
	By approximation, and passing to a smaller $\xi>0$ if necessary, we may assume that for every $n\in\mathbb{N}$ there exist $I_n\in\mathbb{N}$ and families $(\lambda_i^n)_{i=1}^{I_n}\subset \R$ and $(p_i^n,q_i^n)_{i=1}^{I_n}\subset M\times M\setminus\{(x,x):x\in M\}$ such that
	\begin{equation}
	\mu_n=\sum_{i=1}^{I_n}\lambda_i^n m_{p_i^nq_i^n},
	\qquad
	\|\mu_n\|=\sum_{i=1}^{I_n}|\lambda_i^n|\leq1.
	\end{equation}
	For each $i,n$, fix a geodesic $\gamma_i^n\colon[a_i^n,b_i^n]\to M$ joining $p_i^n$ to $q_i^n$.
	\smallskip
	
	Let $(\varepsilon_n)_n$ be a decreasing sequence of positive numbers such that $\sum_{n=1}^\infty \varepsilon_n \le \frac{1}{8}\xi$.  
	Our goal is to construct a WUC series $\sum z_n\Pi_n$, where the following objects will be defined by induction:
	
	\begin{enumerate}[(1)]
		\item a strictly increasing sequence $(M_n)_n$ of natural numbers,
		\item a sequence $(H_n)_n$ of subspaces of $\Lip_0(M)$ together with linear maps $\Phi_n:H_n\to\mathcal{C}(M,X^*)$ provided by the ACUG $X$-structure,
		\item a sequence $(z_n)_n\subset B_{\Lip_0(M)}$, with $z_n\in H_n$,
		\item a subsequence $(m_n)_n$ of $(\mu_n)$,
		\item a sequence $(\Psi_n)_n\subset \Lip(M)$, and $\Pi_n:=\prod_{k=1}^{n-1}(1-\Psi_k)\in \Lip(M)$ for $n\geq 2$, $\Pi_1\equiv1$,
	\end{enumerate}
	satisfying, for each $n\in\N$:
	\begin{enumerate}[$(i)$]
		\item $m_n \in \{\mu_k: M_{n-1}+1\le k\le M_n\}$,
		\item $|\langle z_n,m_n\rangle|\ge\xi$,
		\item $\Psi_n\geq 0$, $\|\Psi_n\|_\infty\leq 1$, and $\bigg\|\Psi_n - \|\cdot\|_{X^*}\circ\Phi_n(z_n)\bigg\|_\infty < \varepsilon_n/2$, 
		\item for all $j>M_n$,
		\begin{equation}
		\sum_{i=1}^{I_j}\frac{|\lambda_i^j|}{d(p_i^j,q_i^j)}\int_{a_i^j}^{b_i^j}\Psi_n(\gamma_i^j(t))\,dt<\varepsilon_n,
		\end{equation}
		\item If $n\geq 2$, then
        $\|\Pi_n\|_L\|z_n\|_\infty\le2\varepsilon_n$ and $\|\Pi_n\|_L\bigg\|\Psi_n-\|\cdot\|_{X^*}\circ\Phi_n(z_n)\bigg\|_\infty<\varepsilon_n/2$.
	\end{enumerate}
	
	\smallskip
	\noindent
	\textbf{Initialization.}
	Choose $M_1\in\N$ large enough so that 
	\begin{equation}
	\frac{\theta_{L_2(X^*)}}{(\log_2 M_1)^{1/q}}<\frac{\varepsilon_1}{2},
	\end{equation}
	where $\theta_{L_2(X^*)}$ is the constant from Lemma~\ref{Lemma:Quantitative_James} applied to the $q$-convex space $L_2(X^*)$.  
	Since $M$ admits an ACUG $X$-structure, there exist a subspace $H_1\subset\Lip_0(M)$ and a linear map $\Phi_1:H_1\to\mathcal{C}(M,X^*)$ such that for each $k\in\{1,\dots,M_1\}$, the function $f_k$ can be approximated by a $1$-Lipschitz $g_k\in H_1$ with $\|\Phi_1 g_k\|_\infty\le1$, and for every $k,i \in \{1,\dots,M_1\}$:
	\begin{equation}
	\langle\mu_i,g_k\rangle>3\xi\ \text{if }k\le i,
	\qquad
	|\langle\mu_i,g_k\rangle|<\xi\ \text{if }k>i.
	\end{equation}
	
	Fix $j>M_1$, and consider
	\begin{equation}
	\mathcal{L}_{2,j}
	=\Bigl(\bigoplus_{i=1}^{I_j} L_2([a_i^j,b_i^j],X^*)\Bigr)_2,
	\qquad
	\|(h_i)\|_{\mathcal{L}_{2,j}}
	=\Bigl(\sum_{i=1}^{I_j}\frac{1}{b_i^j-a_i^j}\int_{a_i^j}^{b_i^j}\|h_i(t)\|_{X^*}^2\,dt\Bigr)^{1/2}.
	\end{equation}
	For each $k\in\{1,\dots,M_1\}$, define
	\begin{equation}
	u_k^j
	=\bigl(\sqrt{|\lambda_i^j|}(\Phi_1 g_k)\circ\gamma_i^j\bigr)_{i=1}^{I_j}
	\in\mathcal{L}_{2,j}.
	\end{equation}
	Then
	\begin{equation}
	\|u_k^j\|_{\mathcal{L}_{2,j}}^2
	=\sum_{i=1}^{I_j}\frac{|\lambda_i^j|}{b_i^j-a_i^j}
	\int_{a_i^j}^{b_i^j}\|\Phi_1 g_k(\gamma_i^j(t))\|_{X^*}^2\,dt
	\le\sum_{i=1}^{I_j}|\lambda_i^j|\le1.
	\end{equation}
	By Lemma~\ref{Lemma:Quantitative_James} and the choice of $M_1$, there exist nonempty $A_j,B_j\subset\{1,\dots,M_1\}$ with $\max A_j<\min B_j$ such that
	\begin{equation}
	v_{(A_j,B_j)}^j
	:=\frac{1}{|A_j|}\sum_{k\in A_j}u_k^j
	-\frac{1}{|B_j|}\sum_{k\in B_j}u_k^j
	\quad\text{satisfies}\quad
	\|v_{(A_j,B_j)}^j\|_{\mathcal{L}_{2,j}}<\varepsilon_1/2.
	\end{equation}
	As there are finitely many such pairs $(A_j,B_j)$, we may pass to a subsequence and find fixed $A,B\subset\{1,\dots,M_1\}$, $\max A<\min B$, such that for all $j>M_1$,
	\begin{equation}
    \label{eq:average_flatness_norm_step_1}
	\|v_{(A,B)}^j\|_{\mathcal{L}_{2,j}}<\varepsilon_1/2.
	\end{equation}
	Define
	\begin{equation}
	z_1:=\frac{1}{2|A|}\sum_{k\in A}g_k-\frac{1}{2|B|}\sum_{k\in B}g_k\in B_{\Lip_0(M)},
	\qquad
	m_1:=\mu_{\max A}.
	\end{equation}
	Then
	\begin{equation}
	|\langle z_1,m_1\rangle|
	\ge \left| \frac{1}{2|A|}\sum_{k \in A}  \langle g_k ,m_1 \rangle \right| - \left|  \frac{1}{2|B|}\sum_{k \in B}  \langle g_k , m_1 \rangle  \right|  \ge\frac{3}{2}\xi-\frac{1}{2}\xi=\xi.
	\end{equation}
	Hence $(i)$ and $(ii)$ are satisfied.  
	\smallskip
	
	By Lemma~\ref{lemma6}, there exists $\Psi_1:M\to\R$ Lipschitz, which can be chosen with $\Psi_1\geq 0$, $\|\Psi_1\|_\infty\leq 1$, such that 
	\begin{equation}\bigg\|\Psi_1 - \|\cdot\|_{X^*}\circ\Phi_1(z_1)\bigg\|_\infty<\frac{\varepsilon_1}{2},\end{equation}
	which shows $(iii)$.
    To verify $(iv)$, fix $j>M_1$. Using Hölder's inequality, the linearity of the operator $\Phi_1$, and \eqref{eq:average_flatness_norm_step_1}, we get 
    \begin{align}
        \sum_{i=1}^{I_j}\frac{|\lambda_i^j|}{d(p_i^j,q_i^j)}&\int_{a_i^j}^{b_i^j}\|\Phi_1(z_1)(\gamma_i^j(t))\|_{X^*}dt\\
        &=\sum_{i=1}^{I_j}\sqrt{|\lambda_i^j|}\frac{1}{d(p_i^j,q_i^j)}\int_{a_i^j}^{b_i^j}\sqrt{|\lambda_i^j|}\|\Phi_1(z_1)(\gamma_i^j(t))\|_{X^*}dt \nonumber\\
        &\leq\left(\sum_{i=1}^{I_j}|\lambda_i^j|\right)^{\frac{1}{2}}\left(\sum_{i=1}^{I_j}\left(\frac{1}{d(p_i^j,q_i^j)}\int_{a_i^j}^{b_i^j}\sqrt{|\lambda_i^j|}\|\Phi_1(z_1)(\gamma_i^j(t))\|_{X^*}dt\right)^2\right)^{\frac{1}{2}} \nonumber\\
        &\leq \left(\sum_{i=1}^{I_j}\frac{1}{d(p_i^j,q_i^j)}\int_{a_i^j}^{b_i^j}|\lambda_i^j|\|\Phi_1(z_1)(\gamma_i^j(t))\|^2_{X^*}dt\right)^{\frac{1}{2}} \nonumber \\
        &= \left\|v^j_{(A,B)}\right\|_{\mathcal{L}_{2,j}}<\frac{\varepsilon_1}{2}.\nonumber
    \end{align}

    Consequently, using also the already established $(iii)$, we get

    \begin{equation}
	\sum_{i=1}^{I_j}\frac{|\lambda_i^j|}{d(p_i^j,q_i^j)}\int_{a_i^j}^{b_i^j}\Psi_1(\gamma_i^j(t))\,dt
	\le\frac{\varepsilon_1}{2}
	+\sum_{i=1}^{I_j}\frac{|\lambda_i^j|}{d(p_i^j,q_i^j)}
	\int_{a_i^j}^{b_i^j}\|\Phi_1 z_1(\gamma_i^j(t))\|_{X^*}\,dt
	<\varepsilon_1.
	\end{equation}
	Thus $(iv)$ is satisfied, and $(v)$ is vacuous for the first step.
	\medskip
	
	\noindent
	\textbf{Inductive step.}
	Assume the sequences are constructed up to index $n-1$.  
	Define $\Pi_n=\prod_{k=1}^{n-1}(1-\Psi_k)$ and set $L_n=\Lip(\Pi_n)$.  
	Let $S_n\subset M$ be a finite $\varepsilon_n/L_n$-dense subset containing~$0$.  
	Choose $M_n>M_{n-1}$ so large that
	\begin{equation}
	\frac{\theta_{L_2(X^*)}\sqrt{1+|S_n|L_n^2\rad(M)^2}}{(\log_2(M_n-M_{n-1}))^{1/q}}<\frac{\varepsilon_n}{2}.
	\end{equation}
	
	By the ACUG structure, there exist $H_n\subset\Lip_0(M)$ and $\Phi_n:H_n\to\mathcal{C}(M,X^*)$ such that for all $k\in\{M_{n-1}+1,\dots,M_n\}$, each $f_k$ is approximated by a $1$-Lipschitz $g_k\in H_n$ with $\|\Phi_n g_k\|_\infty\le1$, and for every $k,i\in\{M_{k-1}+1,\dots,M_k\}$,
	\begin{equation}
	\langle\mu_k,g_i\rangle>3\xi\text{ if }i\le k,
	\qquad
	|\langle\mu_k,g_i\rangle|<\xi\text{ if }i>k.
	\end{equation}
	
	Fix $j>M_n$ and define
	\begin{equation}
	\mathcal{L}_{2,j}=\Bigl(\bigoplus_{i=1}^{I_j}L_2([a_i^j,b_i^j],X^*)\Bigr)_2,
	\qquad
	\mathcal{E}_j=\mathcal{L}_{2,j}\oplus_2\ell_2^{|S_n|}.
	\end{equation}
	For $k\in\{M_{n-1}+1,\dots,M_n\}$, set
	\begin{equation}
	u_k^j=\bigl(\sqrt{|\lambda_i^j|}(\Phi_n g_k)\circ\gamma_i^j\bigr)_{i=1}^{I_j}\in\mathcal{L}_{2,j},\;
	v_k=(L_n g_k(p))_{p\in S_n}\in\ell_2^{|S_n|},\;
	w_k^j=(u_k^j,v_k)\in\mathcal{E}_j.
	\end{equation}
	Then $\|w_k^j\|_{\mathcal{E}_j}\le\sqrt{1+|S_n|L_n^2\rad(M)^2}$.  
	By Lemma~\ref{Lemma:Quantitative_James} and the choice of $M_n$, there exist nonempty $A_j,B_j\subset\{M_{n-1}+1,\dots,M_n\}$ with $\max A_j<\min B_j$ such that
	\begin{equation}
	W_{(A_j,B_j)}^j
	=\frac{1}{|A_j|}\sum_{k\in A_j}w_k^j
	-\frac{1}{|B_j|}\sum_{k\in B_j}w_k^j
	\quad\text{satisfies}\quad
	\|W_{(A_j,B_j)}^j\|_{\mathcal{E}_j}<\frac{\varepsilon_n}{2}.
	\end{equation}
	Passing to a subsequence, we may assume there are fixed $A,B\subset\{M_{n-1}+1,\dots,M_n\}$ with $\max A<\min B$ such that for all $j>M_n$,
	\begin{equation}\label{EqNormWj}
		\|W_{(A,B)}^j\|_{\mathcal{E}_j}<\frac{\varepsilon_n}{2}.
	\end{equation}
	Define
	\begin{equation}
	z_n:=\frac{1}{2|A|}\sum_{k\in A}g_k-\frac{1}{2|B|}\sum_{k\in B}g_k\in B_{\Lip_0(M)},
	\qquad
	m_n:=\mu_{\max A}.
	\end{equation}
	Then $|\langle z_n,m_n\rangle|\ge\xi$, satisfying $(i)$-$(ii)$.  
	From~\eqref{EqNormWj}, we deduce that $|L_n z_n(p)|\le\varepsilon_n$ for all $p\in S_n$, and hence $|L_n z_n(x)|\le2\varepsilon_n$ for all $x\in M$. 
	Therefore $\|\Pi_n\|_L\|z_n\|_\infty\le2\varepsilon_n$, establishing the first part of~(v).
	
	Applying Lemma~\ref{lemma6}, we find $\Psi_n:M\to\R$ Lipschitz with $\Psi_n\geq 0$ and $\|\Psi_n\|_\infty\leq 1$ such that $(iii)$ and the second part of $(v)$ are satisfied: \begin{equation}\label{Eq:approximation_of_upper_gradient_n_step}\max\{1,L_n\}\bigg\|\Psi_n-\|\cdot\|_{X^*}\circ\Phi_n(z_n)\bigg\|_\infty<\frac{\varepsilon_n}{2}.\end{equation} 
	Item~$(iv)$ can be verified with Hölder's inequality, as in the case $n=1$, this time using linearity of $\Phi_n$ and equations \eqref{EqNormWj} and \eqref{Eq:approximation_of_upper_gradient_n_step}.
	
	%\begin{align} \sum_{i=1}^{I_j} \frac{|\lambda_i^j|}{d(p_i^j,q_i^j)}&\int_{a_i^j}^{b_i^j} \| \Phi_k(z_k)(\gamma_i^j(t)) \|_{X^*} dt\\ 
    %&= \sum_{i=1}^{I_j} \sqrt{|\lambda_i^j|} \; \frac{1}{d(p_i^j,q_i^j)}\int_{a_i^j}^{b_i^j} \sqrt{|\lambda_i^j|} \; \| \Phi_k(z_k)(\gamma_i^j(t)) \|_{X^*} dt \nonumber\\ 
    %&\leq \left(\sum_{i=1}^{I_j} |\lambda_i^j| \right)^{\frac 1 2} \, \left(\sum_{i=1}^{I_j} \left( \frac{1}{d(p_i^j,q_i^j)}\int_{a_i^j}^{b_i^j} \; \sqrt{|\lambda_i^j|}\| \Phi_k(z_k)(\gamma_i^j(t)) \|_{X^*} dt \right)^2 \right)^{\frac 1 2} \nonumber\\ 
    %&\leq \left(\sum_{i=1}^{I_j} \frac{1}{d(p_i^j,q_i^j)}\int_{a_i^j}^{b_i^j}|\lambda_i^j| \; \| \Phi_k(z_k)(\gamma_i^j(t)) \|_{X^*}^2 dt \right)^{\frac 1 2}\nonumber\\ 
    %&\leq \left\|W^j_{(A,B)}\right\|_{\mathcal{E}_{j}} \leq \frac{\varepsilon_k}{2}. \nonumber
	%\end{align}
	
	%Together with the uniform approximation above, this yields
	%\begin{equation}
	%\sum_{i=1}^{I_j}\frac{|\lambda_i^j|}{d(p_i^j,q_i^j)}
	%\int_{a_i^j}^{b_i^j}\Psi_k(\gamma_i^j(t))\,dt<\varepsilon_k,
	%\end{equation}
	The inductive step is complete.
	\medskip
	
	\noindent
	\textbf{Conclusion.}
	The first part of $(v)$ in the induction ensures that the series $\sum_n\|\Pi_n\|_L\|z_n\|_\infty$ converges. Hence, using $(iii)$, the second part of $(v)$, and the fact that $\|\cdot\|_{X^*}\circ\Phi_n(z_n)$ is an upper gradient of $z_n$, Lemma~\ref{Lemma:Prod_is_WUC} implies that $\sum_n \Pi_n z_n$ is a WUC series in $\Lip_0(M)$. Moreover, $(i)$ and $(iv)$ let us use Lemma~\ref{Lemma:Prod_does_not_affect_critical_geodesics}, which, together with $(ii)$ yield that 
    \begin{align} \langle \Pi_n z_n , m_n \rangle &\geq |\langle z_n, m_n \rangle|-\left(\|z_n\|_L \sum_{k=1}^{n-1} \varepsilon_k+ \left\|\Pi_n \right\|_L\| z_n \|_\infty\right)\\ &\geq \xi - \frac{1}{8} \xi - \frac{1}{4}\xi \geq \frac{5}{8} \xi. \nonumber
    \end{align} 
    This finally proves that $\Gamma$ is not a $V^*$-set, and therefore $\F(M)$ has property $(V^*)$.
\end{proof}

\section*{Acknowledgments}

The authors are grateful to Assaf Naor for sharing with them the unpublished manuscript \cite{LeeNao03}, for generously allowing them to include in this article the proof of Theorem 2.2 found therein (Theorem \ref{Thm:From_padded_to_separating} here), and for valuable discussions and insights concerning the topics of this paper.
The first two and the last authors gratefully acknowledge several visits to the Universit\'e de Franche-Comt\'e (recently renamed Universit\'e Marie et Louis Pasteur). This work was partially supported by the French ANR projects ANR-20-CE40-0006 and ANR-24-CE40-0892-01.
M. Jung was supported by June E Huh Center for Mathematical Challenges (HP086601) at Korea Institute for Advanced Study and by the research fund of Hanyang University (HY-202500000003346).

%---------------------------------------------------

\end{document}